\documentclass[12pt]{amsart}
\usepackage{tikz-cd}
\usepackage{amssymb}
\usepackage{amsmath,amscd}
\usepackage{mathrsfs}
\usepackage{url}
\usepackage{cite}
\usepackage{fullpage}
\usepackage[hidelinks]{hyperref}
\usepackage{verbatim}
\usepackage{comment}
\usepackage{fancyvrb}
\usepackage{fvextra}
\usepackage{caption}
\usepackage{tikz}
\usepackage{tkz-graph}
\usetikzlibrary{matrix}
\usetikzlibrary{shapes}
\usepackage{todonotes}
\usepackage{tikz-cd}
\usepackage{stmaryrd}

\usepackage{color, colortbl}
\usepackage{array}
\usepackage{varwidth} %

\usepackage[enableskew]{youngtab}
\usepackage{ytableau}

\definecolor{Gray}{gray}{0.9}

\newtheorem{thm}{Theorem}[section]
\newtheorem{prop}[thm]{Proposition}
\newtheorem{lem}[thm]{Lemma}
\newtheorem{cor}[thm]{Corollary}
\newtheorem{conj}[thm]{Conjecture}

\theoremstyle{definition}

\newtheorem{example}[thm]{Example}
\newtheorem{rem}[thm]{Remark}
\newtheorem{question}[thm]{Question}

\numberwithin{equation}{section}
\newcommand{\semis}{\mathrm{ss}}
\newcommand{\s}{\mathsf{S}}
\renewcommand{\l}{\mathsf{L}}

\newcommand{\Z}{\mathbb{Z}}
\newcommand{\F}{\mathbb{F}}

\newcommand{\zz}{\mathbb{Z}}
\newcommand{\cc}{\mathbb{C}}

\renewcommand{\ss}{\mathbb{S}}

\newcommand{\M}{\mathcal{M}}

\newcommand{\A}{\mathcal{A}}

\renewcommand{\tilde}{\widetilde}

\DeclareMathOperator{\ch}{ch}

\DeclareMathOperator{\Aut}{Aut}

\DeclareMathOperator{\SL}{SL}

\DeclareMathOperator{\Spec}{Spec}
\DeclareMathOperator{\Supp}{Supp}

\DeclareMathOperator{\id}{id}

\DeclareMathOperator{\Hom}{Hom}

\DeclareMathOperator{\tr}{tr}
\DeclareMathOperator{\Frob}{Frob}

\DeclareMathOperator{\sign}{sign}

\DeclareMathOperator{\exc}{exc}

\DeclareMathOperator{\Gal}{Gal}
\DeclareMathOperator{\Tr}{Tr}

\newcommand{\bA}{\mathbf{A}}
\newcommand{\bB}{\mathbf{B}}

\newcommand{\PP}{\mathbb{P}}
\newcommand{\ZZ}{\mathbb{Z}}
\newcommand{\CC}{\mathbb{C}}
\newcommand{\QQ}{\mathbb{Q}}
\newcommand{\Q}{\mathbb{Q}}

\newcommand{\cL}{\mathcal{L}}
\newcommand{\ocM}{\overline{\mathcal{M}}}
\newcommand{\cM}{\mathcal{M}}
\newcommand{\cX}{\mathcal{X}}

\newcommand{\sgn}{\operatorname{sgn}}

\newcommand{\MM}{\overline{\mathcal{M}}}

\newcommand{\Feyn}{\mathrm{Feyn}}
\newcommand{\Ind}{\mathrm{Ind}}

\newcommand{\bbS}{\mathbb{S}}

\newcommand{\hide}[1]{}

\newcommand{\numcheck}[1]{\cite[Computation #1]{SupplementaryNotebook}}

\newcommand{\lstw}{\mathsf{L}\mathsf{S}_{12}}

\usepackage{wrapfig}

\newcommand{\Mb}{\overline{\M}}

\newcommand{\pb}{\mathrm{PB}}

\usepackage{ytableau,tikz,varwidth, tikz-cd}
\usepackage{hyperref}

\usetikzlibrary{calc, decorations, positioning}
\usepackage[enableskew]{youngtab}
\usepackage{enumerate}
\usepackage{amssymb, mathscinet}

\usepackage[all,cmtip]{xy}

\usepackage{multicol}
\usepackage{enumitem}

\usetikzlibrary{matrix}
\usetikzlibrary{shapes}
\usetikzlibrary{arrows, automata}
\usetikzlibrary{decorations,decorations.pathmorphing,decorations.pathreplacing,decorations.markings}
\usetikzlibrary{through}
\usetikzlibrary{decorations.pathreplacing,calligraphy}

\tikzset{every picture/.style={baseline=-.65ex} }
\tikzset{ext/.style={circle, draw,inner sep=1pt},int/.style={circle,draw,fill,inner sep=1pt},nil/.style={inner sep=1pt}}

\tikzset{every loop/.style={draw}}
\tikzset{
  crossed/.style={
    decoration={markings,mark=at position .5 with {\arrow{|}}},
    postaction={decorate},
    shorten >=0.4pt}}

  \tikzset{->-/.style={decoration={
    markings,
    mark=at position .5 with {\arrow{>}}},postaction={decorate}}}

  \theoremstyle{definition} %

\newtheorem{defn}[thm]{Definition}
  
\theoremstyle{remark}

\newcommand{\Res}{\mathrm{Res}}

\newcommand{\gr}{\mathrm{gr}}

\newcommand{\GK}{\mathsf{GK}}

\usepackage{color, colortbl}
\usepackage{array, adjustbox}

\definecolor{Gray}{gray}{0.9}

\newcolumntype{g}{>{\columncolor{Gray}}c}
\newcolumntype{M}{V{6cm}} %
\newcolumntype{N}{V{12cm}} %

\author{Samir Canning}
\author{Hannah Larson}
\author{Sam Payne}
\author{Thomas Willwacher}
\thanks{
S.C. was supported by a Hermann-Weyl-Instructorship from the Forschungsinstitut für Mathematik at ETH Z\"urich and the SNSF Ambizione grant PZ00P2\_223473.
This research was partially conducted during the period H.L served as a Clay Research Fellow. 
S.P. was supported in part by NSF grants DMS--2053261 and DMS--2302475, and he conducted parts of this research during a visit to the Institute for Advanced Study supported by the Charles Simonyi Endowment. T.W. has been supported by the NCCR SwissMAP, funded by the Swiss National Science Foundation}

\title{Moduli spaces of curves with polynomial point counts}

\begin{document}

\begin{abstract}
    We prove that the number of curves of a fixed genus $g$ over finite fields is a polynomial function of  the size of the field if and only if $g \leq 8$. Furthermore, we determine for each positive genus $g$ the smallest $n$ such that the moduli space $\cM_{g,n}$ does not have polynomial point count. A  key ingredient in the proofs, which is also a new result of independent interest, is the computation of $H^{13}(\Mb_{g,n})$ for all $g$ and $n$.
\end{abstract}

\maketitle

\setcounter{tocdepth}{1}
\tableofcontents

\section{Introduction}

Let $N_g(q)$ be the number of geometric isomorphism classes of smooth projective curves of genus $g$ that are defined over a finite field $\F_q$ of order $q$. Then $N_0(q) = 1$ and $N_1(q) = q$. The latter is because the isomorphism classes of curves of genus 1 over $\overline \F_q$ that are defined over $\F_q$ are exactly those whose $j$-invariant is in $\F_q$. Further well-known calculations show that $N_2(q) = q^3$, $N_3(q) = q^6 + q^5 +1$ \cite{Looijenga-M3}, and $N_4(q) = q^9 + q^8 + q^7 - q^6$ \cite{Tommasi}.  In each case, the function $N_g$ is \emph{polynomial}, meaning that there is a polynomial $f \in \Z[x]$ such that $N_g(q) = f(q)$ for every prime power $q$.  Our first main result is the following.

\begin{thm} \label{thm:polynomial}
The function $N_g$ is polynomial if and only if $g \leq 8$.
\end{thm}

\noindent Our proof uses the identification of the values of $N_g$ with point counts on moduli spaces and the Grothendieck--Lefschetz trace formula.

\subsection{Point counting on moduli spaces} For $g \geq 2$, let $\cM_g$ denote the moduli stack of smooth projective curves of genus $g$, and let $M_g$ be its coarse moduli space.  Then $N_g(q)$ is $ \#M_g(\F_q)$, the number of $\F_q$-rational points of $M_g$. It is also equal to the stacky count $$\# \cM_g(\F_q) := \sum_{[C] \in  \cM_g(\F_q)} \frac{1}{\# \Aut(C)}.$$ Here, the sum is over isomorphism classes of smooth projective curves of genus $g$ over $\F_q$. By Behrend's extension of the Grothendieck--Lefschetz trace formula to Deligne-Mumford stacks \cite[Theorem~3.1.2]{Behrend}, such point counts are given by the graded trace of Frobenius elements acting on the compactly supported $\ell$-adic cohomology of $\cM_g$:
\[
\# \mathcal{\cM}_g(\F_q) = \sum_i (-1)^i \Tr(\Frob_q^* \mid H^i_{c}(\cM_g, \Q_\ell)), 
\]
whenever $(\ell, q) = 1$. Recent work using the Chow--K\"unneth generation property to control the $\ell$-adic Galois representations appearing in these cohomology groups shows that $N_5$ and $N_6$ are also polynomial, without determining the coefficients \cite[Corollary 1.6]{CL-CKgP}. The first examples where $N_g$ is not polynomial are also recent; they arise when the weight eleven Euler characteristic $\chi_{11}(\cM_g) := \sum_i (-1)^i \dim_\Q \gr_{11}^W H^i_c(\cM_g)$ does not vanish. This nonvanishing was proved for $9 \leq g \leq 70$ and $g \neq 12$ in \cite[Section~7]{PayneWillwacher24}. However, $\chi_{11}(\M_{12}) = 0$.  

\medskip

Theorem~\ref{thm:polynomial} incorporates three new results, each of independent interest and proved by different techniques. First, we show that $N_7$ and $N_8$ are polynomials by combining previously known results with arguments involving symmetry groups of stable graphs, following ideas initiated in \cite[Section~4]{PayneWillwacher24}. Next, we compute $H^{13} (\ocM_{g,n})$ for all $g$ and $n$, following the arguments developed in \cite{CLP-STE}, and use this to show that the weight thirteen Euler characteristic of $\M_{12}$ does not vanish. Finally, we show that $\chi_{11}(\M_g)$ does not vanish for $g > 70$, by a careful analysis of the generating function for weight eleven Euler characteristics.%

\medskip

Throughout, we consider each rational singular cohomology group of a complex algebraic variety or Deligne--Mumford stack defined over $\Q$ with its associated motivic structure, i.e., as a $\Q$-vector space equipped with a mixed Hodge structure and a continuous action of $\Gal(\overline \Q | \Q)$ on the extension of scalars to $\Q_\ell$. The Tate structure $\mathsf{L} := H^2(\PP^1)$ plays a special role;  its associated Hodge structure is $1$-dimensional of type $(1,1)$, and each Frobenius element $\Frob_p \in \Gal(\overline \Q | \Q)$ acts on $\mathsf{L}_{\Q_\ell}$, for $\ell \neq p$, via multiplication by $p$.

We say that a rational cohomology group is of \emph{Tate type} if the associated graded of its weight filtration is a polynomial in $\mathsf{L}$, i.e., it is supported in even weights and the weight $2k$ piece is isomorphic to a direct sum of copies of the tensor power $\mathsf{L}^{k}$.

\begin{thm} \label{thm:Tatetype}
The rational cohomology of $\M_g$ is of Tate type %
if and only if $g \leq 8$.
\end{thm}

Neither Theorem~\ref{thm:polynomial} nor Theorem~\ref{thm:Tatetype} implies the other, although there are partial implications in both directions. If the rational cohomology of a smooth Deligne--Mumford stack is of Tate type, then its point count is polynomial. The converse fails when every non-Tate type irreducible $\ell$-adic Galois representation that appears in the semi-simplification does so equally often in even and odd degrees. %

\subsection{Polynomiality with marked points}

Let $N_{g,n}(q)$ denote the number of geometric isomorphism classes of smooth projective curves of genus $g$ with $n$ marked points that are defined over a finite field of order $q$. We have $N_{0,n}(q) = 1$ for $n \leq 2$, and $N_{1,0}(q) = q$. In all other cases, the moduli space $\cM_{g,n}$ is a Deligne--Mumford stack on which $N_{g,n}$ is given by point counting, $$N_{g,n}(q) = \# \cM_{g,n}(\F_q).$$ There is a rich literature of computations in low genus; see \cite{BergstromData} for compiled tables of known point counts and further references, as well as \cite{vdg} for a recent survey. It is well-known that $N_{1,n}$ is polynomial for $n \leq 10$ but not for $n = 11$. More recently, Bergstr\"om and Faber have shown that  $N_{2,n}$ is polynomial for $n \leq 9$ but not for $n = 10$ and $N_{3,n}$ is polynomial for $n \leq 7$ but not for $n = 8$ \cite{BergstromFaber}. 

\medskip

Kedlaya posed the following question:

\begin{question}\label{quest:Kedlaya}
    Given $g\geq 1$, what is the smallest $n$ such that $N_{g,n}$ is not polynomial?
\end{question}

Theorem~\ref{thm:polynomial} shows that the answer is zero if and only if $g \geq 9$. 

\begin{thm} \label{thm:smallestn}
For $g \geq 1$, the smallest integer $n$ such that $N_{g,n}$ is not polynomial is $$\max\{\lceil (25-3g)/2 \rceil, 0 \}.$$
\end{thm}

\noindent In other words, the answer to Question~\ref{quest:Kedlaya} is the smallest  integer $n$ such that $3g + 2n \geq 25$. This is also the smallest $n$ such that $H^*(\M_{g,n})$ is not of Tate type.

\begin{thm}\label{thm:Tatetype-markedpoints}
    The rational cohomology of $\M_{g,n}$ is of Tate type if and only if $g = 0$ or $3g + 2n < 25$.
\end{thm}

\noindent The bound in the above theorems first appears in a related context in \cite[Section~4]{PayneWillwacher24}, where it is observed that a certain graph complex computing the weight eleven cohomology of $\cM_{g,n}$ vanishes if and only if $g = 0$ or $3g + 2n < 25$. Our proof of Theorem~\ref{thm:Tatetype-markedpoints} involves a strengthening of this observation. We use symmetries of stable graphs and the properties of symmetric group actions on certain cohomology groups to show that the entire $E_1$-page of the weight spectral sequence for $\cM_{g,n} \subset \ocM_{g,n}$ is of Tate type if $3g + 2n < 25$. In fact, our arguments show that for $3g + 2n < 25$ the cohomology of $\M_{g,n}$ can be fully determined from knowledge of the tautological rings of $\MM_{g',n'}$ with $2g'+n' \leq 2g + n$, see Remark  \ref{rem:tautological}.

\begin{conj}\label{conj:polynomial}
Let $g$ and $n$ be nonnegative integers such that $2g + n \geq 3$.  Then the following are equivalent:
\begin{enumerate}
\item The function $N_{g,n}$ is polynomial;
\item The cohomology of $\M_{g,n}$ is of Tate type;
\item Either $g=0$, or $3g+2n<25$.
\end{enumerate}
\end{conj}

To prove the conjecture, it remains to show that $N_{g,n}$ is not polynomial when $g > 0$ and $3g + 2n \geq 25$. We have verified this computationally for $g + n < 150$. More precisely, in this range, we have checked by computer that $\chi_{11}(\M_{g,n})$ is nonzero except in two exceptional cases, namely $\cM_{12}$ and $\cM_{8,1}$. In these last two cases, we have shown that the weight thirteen Euler characteristic does not vanish (Corollary~\ref{cor:chi13}).

\subsection{The thirteenth cohomology group of \texorpdfstring{$\MM_{g,n}$}{mgn}}
One key new ingredient in the proofs of our main results is the following computation of the degree thirteen cohomology of $\MM_{g,n}$, for all $g$ and $n$. This computation necessarily requires a technical discussion. We begin by recalling the presentation for $H^{11}(\Mb_{g,n})$ from \cite{CanningLarsonPayne}. Let $$K^m_n := V_{n-m+1,1^{m-1}},$$ denote the Specht $\ss_n$-module associated to the hook shape of size $n$ whose vertical part has exactly $m$ boxes. 
\ytableausetup{mathmode, boxsize=6mm}
\[
m \begin{cases} \begin{ytableau}
\phantom{.} &  \phantom{.} & \phantom{.}  & \none[\cdots]  & \\
\\
\\
\none[\vdots] \\
\\
\\
\end{ytableau}
\end{cases}
\]
It is an irreducible $\ss_n$-representation of dimension $\binom{n-1}{m-1}$ generated by elements $k_P$, for ordered subsets $P \subset \{1, \ldots, n\}$ of size $m$, with relations  
\begin{equation} \label{11rels} 
k_{\sigma(P)} = \sgn(\sigma) \cdot k_P, \quad \mbox{and} \quad \sum_{j = 0}^{m}(-1)^j \cdot k_{\{a_0, \ldots, \widehat{a_j}, \ldots, a_{m}\}} = 0,
\end{equation}
for any permutation $\sigma$ of $P$, and any size $m+1$ ordered subset $\{a_0, \ldots, a_{m} \} \subset \{1, \ldots, n \}$. 

Let $\mathsf{S}_{12} := H^{11}(\Mb_{1,11})$ be the motivic structure corresponding to weight $12$ cusp forms for $\SL_2(\zz)$. Recall from \cite{CanningLarsonPayne} that $H^{11}(\Mb_{g,n}) = 0$ for $g \neq 1$, and as $\ss_n$-equivariant motivic structures (Hodge structures or Galois representations)
\[
H^{11}(\Mb_{1,n}) \cong K^{11}_n \otimes_\Q \s_{12}.
\]
The isomorphism identifies $(\Q \cdot k_P) \otimes \s_{12}$ with the pullback of $H^{11}(\Mb_{1, P}) \cong \s_{12}$ under the forgetful map $\Mb_{1,n} \to \Mb_{1,P}$, for any $P \subset \{1, \ldots, n\}$ of size $|P| = 11$. 

Now suppose $g \geq 1$, and consider the boundary divisors $D_{1,A} \subset \Mb_{g,n}$ for $A \subset \{1, \ldots, n\}$, i.e., the images of gluing maps
\[\iota_{A} \colon \Mb_{1, A \cup p} \times \Mb_{g-1,A^c \cup p'} \to \Mb_{g,n}.\]
For each such boundary divisor, there is a Gysin push forward 
\[\iota_{A*} \colon H^{11}(\Mb_{1,A\cup p}) \otimes H^0(\Mb_{g-1,A^c \cup p'}) \to H^{13}(\Mb_{g,n}). \]
Let $\lstw := \mathsf{L} \otimes \mathsf{S}_{12}$ denote the Tate twist of $\mathsf{S_{12}}$.

\begin{thm}\label{thm:H13}
The cohomology group $H^{13}(\Mb_{g,n})$ is spanned by the images of the Gysin pushforward maps $\iota_{A*}$, for $A \subset \{1, \ldots, n\}$. If $g \geq 2$  then $\bigoplus_A \iota_{A*}$ is also injective, and there is an $\ss_n$-equivariant isomorphism of Hodge structures or $\ell$-adic Galois representations
\[
H^{13}(\Mb_{g,n}) \cong \bigg ( \bigoplus_{m = 10}^n \Ind_{\ss_m \times \ss_{n-m}}^{\ss_n} \Big(\big(\mathrm{Res}^{\ss_{m+1}}_{\ss_m}K^{11}_{m+1}\big) \boxtimes \mathbf{1}\Big) \bigg) \otimes_\Q \lstw.
\]
    \end{thm}

\noindent 
For $g = 0$, there are no maps $\iota_A$ and the theorem recovers the fact that $H^{13}(\Mb_{0,n}) = 0$.  For $g= 1$, the kernel of $\bigoplus_A \iota_{A*}$ and resulting presentation for $H^{13}(\Mb_{1,n})$ are discussed in \S\ref{basisg1}. See, in particular, Corollary~\ref{gid}. For $n<10$ and $g$ arbitrary, $H^{13}(\Mb_{g,n})=0$. For $g\geq 2$,  $\dim H^{13}(\Mb_{g,n})=2\sum_{m=10}^{n}\binom{n}{m}\binom{m}{10}$.

From Theorem~\ref{thm:H13}, one can write down all of the terms in the weight thirteen row of the $E_1$-page of the weight spectral sequence, count dimensions, and compute 
\[
\chi_{13}(\M_{g,n}) := \sum_i (-1)^i \dim_\Q \gr_{13}^W H^i_c(\cM_{g,n}).
\]

\begin{cor} \label{cor:chi13}
The weight thirteen Euler characteristics of $\M_{12}$ and $\M_{8,1}$ are 
\[
\chi_{13}(\cM_{12}) = -6 \quad \quad \mbox{ and } \quad \quad \chi_{13}(\cM_{8,1}) = -2.
\]
\end{cor}

\noindent This corollary is used in our proofs of Theorems~\ref{thm:polynomial}, \ref{thm:Tatetype}, and \ref{thm:smallestn}.

\subsection{Asymptotics and non-vanishing of the weight eleven Euler characteristic}
Another key ingredient in the proof of our main results is a non-vanishing statement for $\chi_{11}(\M_g)$.
A generating function for this Euler characteristic is given in \cite[Theorem~7.1]{PayneWillwacher24}. Here we study its asymptotics:
\begin{thm}
\label{thm:wt11 intro}
    Asymptotically as $g\to \infty$ we have
\[
\chi_{11}(\M_g)
\sim 
\begin{cases}
    2 C_\infty^{ev} \frac{(-1)^{g/2} (g-2)! }{(2\pi)^g } & \text{for $g$ even} \\
    2 C_\infty^{odd} \frac{(-1)^{(g-1)/2} (g-2)! }{(2\pi)^{g} }
     & \text{for $g$ odd} 
\end{cases}
\]
with the constants
\begin{align*}
    C_\infty^{ev}
    &=
    \frac{1}{10!} 1024 \pi^2 (14175 - 4725 \pi^2 + 630 \pi^4 - 45 \pi^6 + 
       2 \pi^8)  \approx 12.8765, \mbox{ and}
\\
C_{\infty}^{odd}
    &
    =\frac{4\pi}{10!}
    1280 (2835 + 2 \pi^2 (-945 + 189 \pi^2 - 18 \pi^4 + \pi^6)) 
    \approx 
    23.7991.
\end{align*}
Furthermore, $\chi_{11}(\M_g)\neq 0$ for all $g\geq 9$ except for $g=12$.
\end{thm}

\noindent The proof of Theorem \ref{thm:wt11 intro} is elementary, intricate, and involves computer calculations. First, we identify the leading order term contributing to $\chi_{11}(\M_g)$ and show that its asymptotic behavior follows the formula given in Theorem \ref{thm:wt11 intro} above. See Section~\ref{sec:l1terms}. We then derive an estimate on the remaining terms and show that they vanish relative to the leading order term as $g\to \infty$. Our estimates are strong enough to bound the value away from zero for $g\geq 600$. In the remaining cases, for $ g< 600$, we calculate $\chi_{11}(\M_g)$ explicitly on the computer and verify that it does not vanish for all $g \geq 9$ except for $g = 12$.

\begin{rem}
We may use Stirling's formula to rewrite, asymptotically as $g\to \infty$,
\[
\frac{(g-2)!}{(2\pi)^g}
\sim 
\sqrt{2\pi(g-2)}
\frac{(g-2)^{g-2} }{ e^{g-2} (2\pi)^g}
\sim 
  \sqrt{2\pi}  \frac{g^g}{g^{3/2}(2\pi e)^g}.
\]
On the other hand, the asymptotic behavior of the weight zero Euler characteristic has been determined by Borinsky \cite{Borinsky} to be 
\[
\chi_0(\M_g) 
\sim 
\begin{cases}
 (-1)^{g/2} 2\sqrt{2\pi} \frac{g^g}{g^{3/2}(2\pi e)^g}   & \text{for $g$ even} \\
 \frac{\sqrt{2}}{g} 
 \cos\left( \sqrt{\frac{\pi g}{4}} -\frac{\pi g}{4}-\frac{\pi}8 \right)
 e^{\sqrt{\frac{\pi g}{4}}} 
 \frac{g^{g/2}}{(2\pi e)^{g/2}}
 & \text{for $g$ odd} 
\end{cases}.
\]
Hence we find that 
\[
\lim_{h\to\infty} \frac{\chi_{11}(\M_{2h})}{\chi_{0}(\M_{2h})}
=
C_\infty^{ev}
\]
is a positive constant. Thus, the growth rate of $\chi_{11}$, like that of $\chi_0$, is superexponential but too small to account for any positive fraction of $\chi(\cM_g)$, whose asymptotics were computed by Harer and Zagier  \cite{HarerZagier}.
\end{rem}

\begin{rem} The results of this paper reflect recent improvements in our understanding of the unstable cohomology of $\M_g$. The stable cohomology of $\M_g$ and, more generally, the tautological subring, is of Tate type and hence makes a polynomial contribution to the counting function $N_g$.  Until just a few years ago, the only known cohomology groups of moduli spaces $\M_g$ not in this subring were $H^6(\M_3)$ and $H^5(\M_4)$. However, these groups are also of Tate type. The first infinite families of unstable and non-tautological cohomology groups were given in weights zero and two \cite{CGP21, PayneWillwacher24b}. Once again, all of these groups are of Tate type; they can be expressed in terms of tautological cohomology on the compact moduli spaces $\ocM_{g,n}$ via the weight spectral sequence.  By the results of \cite{BergstromFaberPayne, CanningLarsonPayne}, which confirm arithmetic predictions of Chenevier and Lannes \cite{ChenevierLannes}, the lowest weight non-Tate cohomology on moduli spaces of curves appears in weight eleven and is of type $\mathsf{S}_{12}$.  It is a pleasant surprise that the cohomology of type $\mathsf{S}_{12}$, as studied in \cite{PayneWillwacher24}, is nearly enough to prove Theorems~\ref{thm:polynomial} and \ref{thm:smallestn}, and the two exceptional cases, namely $\cM_{12}$ and $\cM_{8,1}$, can be handled using the cohomology of type $\lstw$.
\end{rem}

\begin{rem}
The analogous questions for the moduli spaces of stable curves $\Mb_g$ and $\Mb_{g,n}$ are simpler, in some respects, because these spaces are smooth and proper. Indeed, if a variety or Deligne--Mumford stack is smooth and proper, then its point count is polynomial if and only if its cohomology is of Tate type. Moreover, pullback under a surjective map between proper Deligne-Mumford stacks with projective coarse moduli spaces, such as the forgetful map $\Mb_{g,n+1} \to \Mb_{g,n}$, is injective on cohomology. Thus, if the cohomology of $\Mb_{g,n}$ is not of Tate type, then neither is the cohomology of $\Mb_{g,n+1}$.  On the other hand, finding cohomology that is not of Tate type on $\Mb_g$ requires considering motivic structures of higher weight, because $\mathsf{S}_{12}$ and $\lstw$ never appear, and the cohomology of even weight less than or equal to fourteen is of Tate type \cite{CLP-STE}. Nevertheless, one can find non-Tate cohomology on moduli spaces of stable curves both with and without marked points by studying appearances of the motivic structure of the $k$-fold Tate twists $\mathsf{L}^k \mathsf{S_{12}}$. Indeed, by \cite[Theorem 1.5]{CanningLarsonPayne}, if $g \geq \binom{k+1}{2} + 1$ and $n \geq 11 -k$ then $\mathsf{L}^k \mathsf{S_{12}}$ appears in the cohomology of $\Mb_{g,n}$ and hence $\# \Mb_{g,n}(\F_q)$ is not polynomial.  In particular, $\#\Mb_g(\F_q)$ is not polynomial for $g \geq 67$. The resulting bound on the smallest $n$ such that $\#\Mb_{g,n}(\F_q)$ is not polynomial is sharp for $g \leq 3$ and for $g \geq 67$, but we do not expect it to be so in general.
\end{rem}

\subsection{Structure of the paper}
In Section \ref{pointcounting}, we recall well-known conditions for polynomial point counts, based on the Grothendieck--Lefschetz trace formula. In particular, the vanishing of odd weight Euler characteristics is a necessary condition for polynomial point counts, and having cohomology of Tate type is a sufficient condition.  We deduce that Theorems~\ref{thm:polynomial}, \ref{thm:Tatetype}, and \ref{thm:smallestn} follow from the ``if" part of Theorem~\ref{thm:Tatetype-markedpoints} combined with Corollary~\ref{cor:chi13}, and Theorem~\ref{thm:wt11 intro}.

In Section \ref{sec:polycount}, we prove the ``if" part of Theorem~\ref{thm:Tatetype-markedpoints} and show that the ``only if" part follows from Theorem~\ref{thm:smallestn}. To show that $H^*(\M_{g,n})$ is of Tate type for $3g + 2n < 25$, we use known results on non-Tate cohomology when $g=1$, recent results on the pure weight cohomology of $\M_{g,n}$ when $g\geq 2$, and symmetries of stable graphs. 

In Sections~\ref{h13sec}--\ref{sec:GK13}, we compute $H^{13}(\Mb_{g,n})$ for all $g$ and $n$ and apply this to compute the weight thirteen Euler characteristic of $\M_{g,n}$ for small $g$ and $n$.  In particular, we prove  Theorem~\ref{thm:H13} and deduce Corollary~\ref{cor:chi13}. 

Finally, in Section~\ref{sec:wt11euler}, we analyze the generating function for the weight eleven Euler characteristics of the moduli spaces $\cM_{g,n}$ and prove Theorem~\ref{thm:wt11 intro}.

\medskip

\noindent \textbf{Acknowledgments.} We are grateful to K. Kedlaya for helpful conversations and encouragement, and for posing Question~\ref{quest:Kedlaya}. We thank D.~Petersen for suggesting the argument of Proposition~\ref{prop:nonTate-induction}, which allowed us to sharpen the statement of Theorem~\ref{thm:Tatetype-markedpoints}. We also thank W. Sawin for helpful comments on an earlier draft of this work.

\section{Preliminaries on point counting and cohomology}\label{pointcounting}

In this section, we briefly state and prove a sufficient condition and a necessary condition for polynomial point counts in terms of weight graded compactly supported cohomology. These statements are well-known to experts, cf. \cite{BogaartEdixhoven} for the proper case. We present them in the form that will be used in the proofs of our main theorems. See \cite{MilneLectures} for background on $\ell$-adic \'etale cohomology and its base change and comparison theorems.

\subsection{Weights in cohomology} Let $X$ be a smooth algebraic variety defined over $\Q$. The algebraic structure on $X$ induces a mixed Hodge structure on the compactly supported singular cohomology of the associated complex manifold and, in particular, an increasing weight filtration with rational coefficients
\[
W_0H^*_c(X) \subset W_1 H^*_c(X) \subset \cdots.
\] 
The mixed Hodge structure is strongly functorial for maps between cohomology groups induced by algebraic morphisms and is computable from any normal crossings compactification \cite{DeligneHodgeII, DeligneHodgeIII}.

By the Artin comparison theorem relating singular cohomology to $\ell$-adic \'etale cohomology, $H^*_c(X, \Q_\ell)$ carries a continuous action of the the absolute Galois group $\Gal(\overline \Q | \Q)$, for any prime $\ell$. Moreover, maps between $\ell$-adic cohomology groups induced by algebraic morphisms over $\QQ$ commute with the Galois action.  Since the weight filtration is computable from any normal crossings compactification and such compactifications exist over $\Q$, the Galois action preserves each of the weight subspaces $W_k(H^*_c(X), \Q_\ell) := W_kH^*_c(X) \otimes \Q_\ell$.

Let $\overline X$ be a normal crossing compactification of $X$.  For $s$ sufficiently divisible, we can extend the pair $(\overline X, X)$ to a pair of smooth schemes $(\overline {\mathfrak X}, \mathfrak X)$ over $\Z[1/s]$ such that the boundary $\overline {\mathfrak X} \smallsetminus \mathfrak X$ is a divisor with normal crossings relative to $\Z[1/s]$. Then, by the base change theorems in $\ell$-adic \'etale cohomology and Grothendieck--Lefschetz trace theorem, we have
\begin{equation} \label{eq:trace}
\# \mathfrak X(\F_q) = \sum_i (-1)^i \Tr((\Frob_p)^m | H^i_c(X, \Q_\ell)),
\end{equation}
for all primes $p \nmid \ell s$, and $q = p^m$.  

\medskip

The discussion above extends to smooth Deligne--Mumford stacks over $\Q$ and their associated complex analytic orbifolds; the extension of motivic structures can be seen by presenting such a stack as a simplicial scheme and following the arguments of \cite{DeligneHodgeIII}. For stacks such as $\cM_{g}$ that are global quotients of smooth varieties by finite groups, this extension is particularly straightforward.  If $G$ is a finite group acting on a variety $X$, the rational cohomology of the quotient stack $[X/G]$ is the invariant subspace $H^*(X)^G$, and the standard arguments go through by taking $G$-invariants at every step. The extension of the Grothendieck--Lefschetz trace theorem is due to Behrend \cite[Theorem~3.1.2]{Behrend}. The point count on a Deligne--Mumford stack $\cX$ is the groupoid cardinality $$\# \cX(\F_q) := \sum_{[x] \in \cX(\F_q)} \frac{1}{\# \Aut(x)};$$ the stack structure makes $\cX(\F_q) := \{ x\colon \Spec(\F_q) \to \cX \}$ a groupoid, and each isomorphism class $[x]$ is counted with multiplicity $1/\#\Aut(x)$.

\subsection{Conditions for polynomial point counts}

Let $\cX$ be a smooth Deligne--Mumford stack with a normal crossing compactification relative to $\ZZ[1/s]$. By \eqref{eq:trace}, the point count $\# \cX(\F_{p^m})$, for $p \nmid s$, is determined by the eigenvalues of $\Frob_p$ acting on $\ell$-adic cohomology, for $\ell \neq p$. The sizes of these eigenvalues are related to the weight filtration as follows.

Recall that a \emph{$p$-Weil number of weight $k$} is an algebraic integer whose image under any embedding $\overline \Q \hookrightarrow \CC$ has absolute value $p^{k/2}$. By \cite{DeligneWeilI}, each eigenvalue of $\Frob_p$ acting on 
\[
\gr_k^W H^*_c(\cX,\Q_\ell) := W_k H^*_c(\cX, \Q_\ell) / W_{k-1} H^*_c (\cX, \Q_\ell)
\]
is a $p$-Weil number of weight $k$. Moreover, the characteristic polynomial of $\Frob_p$ acting on $H^i_c(\cX, \Q_\ell)$ is independent of the auxiliary choice of $\ell \neq p$. 

Since $\cX$ is smooth, it follows from Poincar\'e duality that $H^*(\cX)$ is of Tate type if and only if $H^*_c(\cX)$ is of Tate type. In this case, the odd graded pieces of the weight filtration vanish and $\Frob_p$ acts on $\gr_{2k}^W H^*_c(\cX, \Q_\ell)$ via multiplication by $p^k$.

\begin{defn}
    Let $\cX$ be a complex algebraic variety or Deligne--Mumford stack. The \emph{weight $k$ Euler characteristic} of $\cX$ is
\[
\chi_k(\cX) := \sum_i (-1)^i \dim_\Q \gr_k^W H^i_c(\cX).
\]
\end{defn}

\begin{prop} \label{prop:Tate}
    Suppose $\cX$ is a smooth Deligne--Mumford stack with a normal crossings compactification relative to $\Z[1/s]$ and $H^*(\cX)$ is of Tate type. Then there is a polynomial $f \in \Z[x]$ such that $\# \cX(\F_q) = f(q)$, for $(q,s) = 1$.
\end{prop}

\begin{proof}
If $H^*(\cX)$ is of Tate type then so is $H^*_c(\cX)$, and the right hand side of \eqref{eq:trace} is equal to the polynomial $ f(x) = \sum_{k} \chi_{2k}(\cX) \, x^k$ evaluated at $q$.
\end{proof}

\begin{prop}\label{prop:polycount euler}
Let $\cX$ be a smooth Deligne--Mumford stack with a normal crossings compactification relative to $\Z[1/s]$. Suppose there is a polynomial $f \in \Z[x]$ and a prime power $q$, with $(q,s) = 1$,  such that $\#\cX(\F_{q^m}) = f(q^m)$ for all $m$.  Then $\chi_k(\cX) = 0$ for all odd $k$.
\end{prop}

The proof uses the next lemma, which is a variant of \cite[Lemma~4.1]{BogaartEdixhoven}. 

\medskip

Let $\varphi \colon \CC \to \CC$ be a function with finite support and let $q \geq 2$ be an integer.  Define a complex-valued function on $\ZZ_{>0}$ by
\begin{equation} \label{eq:Nphi}
N_\varphi(m) := \sum_{\alpha \in \Supp(\varphi)} \varphi(\alpha) \cdot \alpha^m.
\end{equation}

\begin{example} \label{ex:Nphi}
Let $\cX$ be a smooth Deligne--Mumford stack with a normal crossings compactification relative to $\Z[1/s]$.  Let $\varphi(\alpha)$ denote the virtual multiplicity of $\alpha$ as an eigenvalue of $\Frob_p$ acting on $H^*_c(\cX,\QQ_\ell)$, for some prime $p \nmid \ell s$. In other words, if $m_i(\alpha)$ is the multiplicity of $\alpha$ as an eigenvalue of $\Frob_p$ acting on $H^i(\cX, \QQ_\ell)$, then $\varphi(\alpha) := \sum_i (-1)^i m_i(\alpha).$ Then the trace formula \eqref{eq:trace} says that, for all positive integers $m$, we have
\[
\#\cX(\F_{p^m}) = N_\varphi(m).
\]
\end{example}

\begin{lem} \label{lem:Nphi}
Let $\varphi$ and $\varphi'$ be complex valued functions on $\CC$ with finite support, and define $N_\varphi$ and $N_{\varphi'}$ as in \eqref{eq:Nphi}. If $N_\varphi = N_{\varphi'}$ then $\varphi = \varphi'$.
\end{lem}

\begin{proof}
It suffices to show that if $N_\varphi = 0$ then $\varphi = 0$. Let $\{\alpha_1, \ldots, \alpha_r\}\subset \cc$ be a finite set containing the support of $\varphi$, ordered so that $|\alpha_1| \leq \cdots \leq |\alpha_r|$. We will show that $\varphi(\alpha_r) = 0$, and the lemma follows by induction on $r$. 

We first compute a limit of finite geometric sums:
\[
\lim_{t\to \infty} \frac 1 t \sum_{m=1}^t \alpha_i^m / \alpha_r^m
=
\begin{cases}
1 & \text{if $i=r$}, \\
0 & \text{otherwise}.
\end{cases}.
\]
This uses the assumption that $\alpha_i \neq \alpha_r$ for $i \neq r$ and $|\alpha_i| \leq |\alpha_r|$.  (In the case where $|\alpha_i| = |\alpha_r|,$ the infinite sum $\sum_{m = 1}^\infty \alpha_i^m / \alpha_r^m$ may not converge, but each finite sum is bounded in absolute value by $2/ |1 - \alpha_i/\alpha_m|$, since $1 + x + \cdots + x^t = (1-x^{t+1})/ (1 -x)$.)  Thus, $\lim_{t \to \infty} \sum_{m = 1}^t N_\varphi(m)/\alpha_r^{m}  = \varphi(\alpha_r)$.  In particular, if $N_\varphi(m) = 0$ for all $m$, then $\varphi(\alpha_r) = 0$. 
\end{proof}

\begin{proof}[Proof of Proposition \ref{prop:polycount euler}]
By the trace formula (Example~\ref{ex:Nphi}) we have $\#\cX(\F_{q^m}) = N_\varphi(m)$, where $\varphi(\alpha)$ is the virtual multiplicity of $\alpha$ as an eigenvalue of $\Frob_q$ acting on $H^*_c(\cX, \Q_\ell)$, for $\ell \nmid q s$.  By Lemma~\ref{lem:Nphi}, if $\#\cX(\F_{q^m})$ is polynomial then $\varphi$ is supported on powers of $q$, i.e., the virtual multiplicity of any $\alpha \not \in \{1, q, q^2, \ldots \}$ is zero. In particular, since $\chi_k(\cX)$ is the sum of the virtual multiplicities of the eigenvalues of size $q^{k/2}$, we have $\chi_k(\cX) = 0$ for all odd $k$.
\end{proof}

\subsection{Proof that Theorems \ref{thm:polynomial}, \ref{thm:Tatetype}, and \ref{thm:smallestn} follow from Theorem~\ref{thm:Tatetype-markedpoints}, Corollary~\ref{cor:chi13}, and Theorem~\ref{thm:wt11 intro}.}
Theorem \ref{thm:Tatetype-markedpoints} says that the cohomology $H^*_c(\cM_{g,n})$ is of Tate type for $3g + 2n < 25$. In particular, setting $n = 0$, we have that $H^*(\M_g)$ is of Tate type for $g \leq 8$.  Hence, by Proposition~\ref{prop:Tate}, $N_{g}$ is polynomial for $g \leq 8$. 

To prove Theorems~\ref{thm:polynomial} and  \ref{thm:Tatetype}, it remains to check, for $g \geq 9$, that $H^*_c(\M_{g})$ is not of Tate type and $N_g$ is not polynomial. Both properties follow from the non-vanishing of $\chi_k(\cM_g)$ for some odd $k$, the former by definition and the latter by Proposition~\ref{prop:polycount euler}.  Theorem \ref{thm:wt11 intro} says that $\chi_{11}(\cM_g)$ is non-zero for $g\geq 9$ and $g\neq 12$, and Corollary \ref{cor:chi13} says that $\chi_{13}(\cM_{12}) \neq 0$.

Theorem~\ref{thm:smallestn} says that for $g \geq 1$, the smallest integer $n$ such that $N_{g,n}$ is not polynomial is $\max\{ \lceil (25-3g)/2 \rceil, 0 \}.$ By Theorem~\ref{thm:Tatetype-markedpoints}, $N_{g,n}$ is polynomial whenever $n$ is smaller than this bound. It remains to see that $N_{g,n}$ is not polynomial for $n = \max\{ \lceil (25-3g)/2 \rceil, 0 \}$.  Again by Proposition~\ref{prop:polycount euler}, it suffices to show that $\chi_k(\M_{g,n})$ is nonzero for some odd $k$.  Theorem~\ref{thm:wt11 intro} says that $\chi_{11}(\cM_g) \neq 0$ for $g \geq 9$ and $g \neq 12$ and Corollary~\ref{cor:chi13} says that $\chi_{13}(\cM_{12})$ and $\chi_{13}(\cM_{8,1})$ are both nonzero.  (The computations of $\chi_{13}$ are needed in these two cases because $\chi_{11}(\cM_{12}) = \chi_{11}(\cM_{8,1}) = 0$%
.)

The last cases to consider are $1 \leq g \leq 7$. Here, we have $\chi_{11}(\cM_{g,\lceil (25-3g)/2 \rceil}) \neq 0$ by previously known calculations.  See \cite{BergstromData} for $1 \leq g \leq 3$ and the tables in \cite[Section~7]{PayneWillwacher24} for $4 \leq g \leq 7$.  The $\ss$-equivariant weight 11 Euler characteristics are:
\begin{align*}
\chi_{11}^\ss(\cM_{1,11}) & = -s_{1^{11}}, & \chi_{11}^\ss(\cM_{2,10}) & = -s_{21^8}, & \chi_{11}^\ss(\cM_{3,8}) & = s_{1^8}, &
\chi_{11}^\ss(\cM_{4,7}) & = s_{21^5}, \\
\chi_{11}^\ss(\cM_{5,5}) & = -s_{1^5}, & \chi_{11}^\ss(\cM_{6,4}) & = s_{21^2}, &
\chi_{11}^\ss(\cM_{7,2}) & = s_{1^2}. &  
\end{align*}
\vskip-19pt
\hfill \qed

\section{Polynomiality for \texorpdfstring{$3g+2n<25$}{gleq8}}\label{sec:polycount}
Here, we prove Theorem \ref{thm:Tatetype-markedpoints}, the first of three steps in the proof of our main results; it says that $H^*(\cM_{g,n})$ is of Tate type if and only if $g = 0$ or $3g + 2n < 25$. For the ``if" direction, we prove a stronger fact: the $E_1$-page of the weight spectral sequence for $\cM_{g,n} \subset \ocM_{g,n}$ is of Tate type. For $g = 0$, this is well-known, by \cite{Keel}.  Our arguments for $3g + 2n < 25$ are framed in terms of graph complexes and use the automorphism groups of stable graphs in an essential way. We begin with a short, inductive proof that the ``only if" direction follows from Theorem~\ref{thm:smallestn}.

\begin{prop} \label{prop:nonTate-induction}
If $H^*(\M_{g,n})$ is not of Tate type, then $H^*(\M_{g,n+1})$ is not of Tate type.
\end{prop}

\begin{proof}
The cohomology of $\M_{0,n}$ is of Tate type for all $n$, and $H^*(\M_{1,n})$ is of Tate type if and only if $n \leq 10$. So, we assume $g \geq 2$. Also, by Poincar\'e duality, $H^*(\M_{g,n})$ is of Tate type if and only if $H^*_c(\M_{g,n})$ is of Tate type. 

Let $\mathcal{C}_{g,n} \to \M_{g,n}$ be the universal curve. The map
\begin{equation} \label{eq:pullback+2}
H^k_c(\M_{g,n})\otimes \mathsf{L} \to H^{k+2}_c(\mathcal{C}_{g,n})
\end{equation}
given by pullback followed by cup product with the relative canonical divisor is injective, because further composition with the  Gysin push-forward to $\M_{g,n}$ is multiplication by $2g-2$.

Let $k$ be the largest degree for which $H^k_c(\M_{g,n})$ is not of Tate type. %
There is a natural open inclusion $\M_{g,n+1} \subset \mathcal{C}_{g,n}$ as the complement of the union of the $n$ tautological sections. Consider the induced excision sequence
\begin{equation} \label{eq:excision}
\cdots \to H^{k+2}_c (\M_{g,n+1}) \to H^{k+2}_c(\mathcal{C}_{g,n}) \to \bigoplus_{i=1}^n H^{k+2}_c (\M_{g,n}) \to  \cdots.
\end{equation}
By hypothesis, $H^{k+2}_c(\M_{g,n})$ is of Tate type. Moreover, since $H^k_c(\M_{g,n})$ is not of Tate type and \eqref{eq:pullback+2} is injective, we have that $H^{k+2}_c(\mathcal{C}_{g,n})$ is not of Tate type. By exactness of \eqref{eq:excision}, the non-Tate part of $H^{k+2}_c(\mathcal{C}_{g,n})$ must be in the image of $H^{k+2}_c (\M_{g,n+1})$. Thus $H^*_c(\M_{g,n+1})$ is not of Tate type and, by Poincar\'e duality, neither is $H^*(\M_{g,n+1})$.
\end{proof}

\noindent When $g$ is positive, Theorem~\ref{thm:smallestn} tells us that $H^*(\M_{g,n})$ is not of Tate type for the smallest $n$ such that $3g + 2n \geq 25$. Thus, by Proposition~\ref{prop:nonTate-induction}, the ``only if" part of Theorem~\ref{thm:Tatetype-markedpoints} follows from Theorem~\ref{thm:smallestn}. The remainder of this section is devoted to the ``if" direction of Theorem~\ref{thm:Tatetype-markedpoints}, showing that $H^*(\M_{g,n})$ is of Tate type when $3g + 2n < 25$.

\subsection{The weight spectral sequence and the Getzler--Kapranov graph complex}
\label{sec:GK def}
The rational singular cohomology groups $\{H^*(\Mb_{g,n})\}_{g,n}$ together with their symmetric group actions and pullbacks to boundary strata form a structure called a 
 \emph{modular cooperad}, defined precisely in \cite{GetzlerKapranov}.
This cooperad, which we denote $H(\Mb)$, is one of motivic structures, meaning the symmetric group actions and pullbacks to boundary strata respect the associated Hodge structures and $\ell$-adic Galois representations. 

In order to prove results about the cohomology of the open moduli spaces $\M_{g,n}$, we study the Feynman transform of $H(\Mb)$,
which we denote $\GK$.
For each $(g, n)$, the component $\GK_{g,n}$ is a \emph{Getzler-Kapranov graph complex}, c.f.\ \cite[Definition 2.2]{PayneWillwacher24b}. We shall describe generators of this complex explicitly below.
Roughly speaking, they are elements of the cohomology of the sources of gluing maps that arise in the boundary stratification of $\Mb_{g,n}$.
The differential on $\GK_{g,n}$ is constructed from pullbacks along gluing maps. The cohomology of the complex $\GK_{g,n}$ provides information about the compactly supported cohomology of $\M_{g,n}$.
For a short introduction with more details, we refer the reader to  \cite[Sections~2.3-2.5]{PayneWillwacher24b}.

The generators for $\mathsf{GK}_{g,n}$ are stable graphs $\Gamma$ of genus $g$ with $n$ legs, in which each vertex $v$ is decorated by a cohomology class $\gamma_v\in H^{k_v}(\Mb_{g_v,n_v})$. Here, $g_v$ and $n_v$ are the genus and valence of the vertex $v$, respectively. We denote such a generator by $[\Gamma,\gamma]$, where $\Gamma$ is the underlying stable graph and $\gamma=(\gamma_v)_{v\in V(\Gamma)}$ is the decoration. The \emph{weight} of a generator is the sum of the cohomological degrees of the decorations at each vertex: $k=\sum_v k_v$. The differential respects the weights, and the complex splits as a direct sum 
\[
\mathsf{GK}_{g,n} \cong \bigoplus_{k}\mathsf{GK}_{g,n}^k.
\]
The summand $\mathsf{GK}_{g,n}^k$ computes the weight $k$ compactly supported cohomology of $\M_{g,n}$:
\[
H^*(\mathsf{GK}_{g,n}^k)\cong \gr_k^W H^*_c(\M_{g,n}).
\]
Indeed, $\mathsf{GK}_{g,n}^k$ is identified with the $k$th row of the first page of the weight spectral sequence
\begin{equation}\label{weightspectral}
E_1^{j,k}=\bigoplus_{|E(\Gamma)|=j} (H^k(\Mb_{\Gamma})\otimes \det E(\Gamma))^{\Aut(\Gamma)},
\end{equation}
which degenerates at $E_2$ and abuts to $\gr^W H^{*}_c(\cM_{g,n})$ The sum is over stable graphs $\Gamma$ of genus $g$ with $n$ legs,  $E(\Gamma)$ is the edge set of $\Gamma$ on which $\Aut(\Gamma)$ acts by permutations, and 
\[
\Mb_{\Gamma}=\prod_{v\in V(\Gamma)} \Mb_{g_v,n_v}.
\]
Identifying invariants with coinvariants in the usual way, by averaging over a finite group action on a rational vector space, the isomorphism from $\GK_{g,n}$ to \eqref{weightspectral} takes a generator $[\Gamma, \gamma]$ to its image in $(H^*(\Mb_\Gamma) \otimes \det E(\Gamma))_{\Aut(\Gamma)}$. 

\subsection{Non-Tate contributions to the Getzler--Kapranov complex}
If $H^*(\Mb_{g',n'})$ is Tate for all $g',n'$ such that $2g+n\geq 2g'+n'$, then $H^*(\Mb_\Gamma)$ is Tate for every stable graph $\Gamma$ of genus $g$ with $n$ leaves.  In this case, every term in  \eqref{weightspectral} is a substructure of a Tate structure and hence Tate. This is enough to conclude that $H^*(\cM_{g,n})$ is of Tate type.

However, it sometimes happens that every term in \eqref{weightspectral} is Tate even though some $H^*(\Mb_\Gamma)$ is not. This is due to the action of automorphisms of stable graphs; the $\Aut(\Gamma)$-invariant subspace of $H^*(\Mb_\Gamma) \otimes \det E(\Gamma)$ may still be Tate,  as in the following example.

\begin{example} \label{ex:29}
Consider the Getzler--Kapranov complex $\GK_{2,9}$.  The generators have vertex decorations in $H^*(\MM_{g',n'})$ for $g',n'$ with $g'\leq 2$ and $2g'+n'\leq 13$.
In this range, the cohomology $H^*(\MM_{g',n'})$ is completely known, see e.g. \cite{BergstromData}. The only non-Tate group in this range is $H^{11}(\Mb_{1,11})$, on which the symmetric group $\ss_{11}$ acts by the sign representation.

There is only one stable graph $\Gamma$ of type $(2,9)$ with a vertex of type $(1,11)$, namely 
\[
\begin{tikzpicture}
    \node[ext] (v) at (0,0) {$\scriptstyle 1$};
    \draw (v) edge[loop] (v);
    \foreach \i in {-4,...,4}
       \draw (v) edge +(0.3*\i,-.7);
\end{tikzpicture}\,.
\]
The $\Aut(\Gamma)$-invariant subspace of $H^{11}(\Mb_{1,11}) \otimes \det E(\Gamma)$ is trivial; the automorphism flipping the loop acts by $-1$ on any decoration in $H^{11}(\Mb_{1,11})$ and acts trivially on the edge set. Thus, we recover the known fact that $H^*(\cM_{2,9})$ is of Tate type even though $H^*(\Mb_{1,11})$ is not.
\end{example}

Recall the indexing of finite-dimensional irreducible representations of the symmetric group $\ss_m$ by integer partitions $\lambda = (\lambda_1, \ldots, \lambda_r)$ with $\lambda_1 \geq \cdots \geq \lambda_r > 0$ and $\lambda_1 + \cdots + \lambda_r = m$. We identify such a partition $\lambda$ with the corresponding Young diagram, which has $r$ rows, and write $H^*(\Mb_{g,n})_\lambda \subset H^*(\Mb_{g,n})$ for the $\lambda$-isotypic subspace.

In Example~\ref{ex:29}, we have $(g_v,n_v) = (1,11)$, and $r = 11$. Then $3g_v + n_v + r = 25$, which is strictly greater than $3g + 2n = 24$.  The following lemma abstracts and generalizes the vanishing phenomenon exhibited in this example. It is phrased in terms of the cohomology of moduli spaces (the relevant context for this paper), but its essential content is about representations of symmetric groups and automorphism groups of stable graphs.

\begin{lem} \label{lem:excess}
Let $[\Gamma,\gamma]$ be a generator for $\GK_{g,n}$ with a vertex decoration $\gamma_v \in H^*(\Mb_{g_v, n_v})_\lambda$, where $\lambda$ is a partition with $r$ rows. If $3g_v + n_v + r > 3g + 2n,$ then $[\Gamma,\gamma] = 0$ in $\GK_{g,n}$.
\end{lem}

\begin{proof}
Say $\Gamma$ has $\ell$ loops based at $v$, and let $G \subset \Aut(\Gamma)$ be the subgroup generated by the involutions exchanging the two half-edges in each of these loops (so $G \cong (\Z/2\Z)^\ell$). We will show that $[\gamma_v] = 0$ in the $G$-coinvariants $(H^*(\Mb_{g,n})_\lambda)_{G}$.  

Note that $G$ acts as a subgroup of the symmetric group $\ss_{n_v}$, so we may restrict attention to an irreducible subrepresentation $V_\lambda \subset H^*(\Mb_{g,n})$ that contains $\gamma_v$.  Also, by the branching rule, since $\lambda$ has $r$ rows, $V_\lambda$ appears in $W := \Ind_{\ss_r}^{\ss_{n_v}} (\sgn)$. Thus, it will suffice to show that $W_G = 0$. We give a graphical presentation for $W$ as follows. 

Consider a rational vector space $W$ generated by copies of $\Gamma$ in which $r$ of the half-edges incident to $v$ are labeled $\omega$ and the remaining $m = n_v -r$ half-edges are labeled $\epsilon_1, \ldots, \epsilon_m$.  Each generator is also equipped with orientation data: a choice of ordering of the $\omega$ labels, such that reordering the set of $\omega$-labeled half-edges induces multiplication by the sign of the corresponding permutation. Then $G$ acts by permuting these generators, and we have an isomorphism of $G$-representations $W \cong \Ind_{\ss_r}^{\ss_{n_v}} (\sgn)$, by construction.  

We claim that each generator of $W$ has a loop based at $v$ in which both half-edges are decorated by $\omega$. Flipping this loop gives an element of $G$ that maps this generator to its negative, and hence the coinvariants vanish. To prove the claim, we remove the vertex $v$ from $\Gamma$, splitting the resulting graph into connected components
    \[
    \Gamma \smallsetminus \{v\} =  C_1 \cup \dots \cup C_k,
    \]
as in \cite[Section 3.2]{PayneWillwacher24}. Note that each component has some dangling edges labeled by $\epsilon_i$ or $\omega$; a loop edge becomes a single edge with two such decorations. For instance, the graph in Example~\ref{ex:29} becomes a disjoint union of 10 edges, one from each of the legs in the original graph (labeled $1, \ldots, 9$), and one from the loop, as shown here. 
\[
\begin{tikzpicture}
    \node[ext] (v) at (0,0) {$\scriptstyle 1$};
    \draw (v) edge[loop] (v);
    \foreach \i in {1,...,9}
    {
    \node (v\i) at (\i*.3-1.5,-.8) {$\scriptstyle \i$};
       \draw (v) edge (v\i);
    }
\end{tikzpicture}
\quad \quad \mapsto \quad \quad 
\begin{tikzpicture}
    \node (v) at (0,.4) {$\scriptstyle \omega$};
    \node (w) at (0,-.4) {$\scriptstyle \omega$};
    \draw (v) edge (w);
      \foreach \i in {1,...,9}
       { \node (v\i) at (\i*.6,.4) {$\scriptstyle \omega$};
    \node (w\i) at (\i*.6,-.4) {$\scriptstyle \i$};
    \draw (v\i) edge (w\i);
    }  
\end{tikzpicture}
\]

By hypothesis, $3g_n + n_v + r > 3g + 2n$. We write the difference $3(g-g_v) + 2n - n_v - r$ as a sum over the components $C_i$, as follows. Let $e_i$, $\omega_i$, and $n_i$ be the number of leaves of $C_i$ labeled by an element of $\{ \epsilon_1, \ldots, \epsilon_m\}$, $\omega$, or an element of $\{1, \ldots, n\}$, respectively. Let
$$g_i := h^1(C_i) + \sum_{v' \in V(C_i)} g_{v'} + e_i + \omega_i - 1;$$
this is the contribution of $C_i$ to the total genus $g$ of $\Gamma.$ We define the \emph{excess} of $C_i$ by
\begin{equation} \label{eq:excCi}
\exc(C_i) := 3g_i + 2n_i -2 \omega_i - e_i.
\end{equation}
Note that  
\[
    \sum_i \exc(C_i) = 3(g-g_v) + 2n - n_v - r. 
\]
By hypothesis $\sum_i \exc(C_i) < 0$, so there must be some component $C_i$ of negative excess.

Using $\eqref{eq:excCi}$ to substitute for $g_i$, we can express the excess of $C_i$ as  $$\exc(C_i) = 3h^1(C_i) + 3 \sum_{v' \in V(C_i)} g_{v'} + 2 e_i + \omega_i - 3.$$ If $h^1(C_i)$ or some $g_{v'}$ is strictly positive, then $\exc(C_i) \geq 0$. Thus, we consider the case where $C_i$ is a tree in which all vertices have genus 0. Let $C_i$ be such a tree. If it has a vertex then it must have at least three leaves, each with a label from $\{ \epsilon_1, \ldots, \epsilon_m, \omega, 1, \ldots, n \}$.  Again, the excess is nonnegative.

Thus any component of negative excess must have no vertices; it is a single edge with two labels. The excess is negative exactly when both labels are $\omega$. In particular, we have shown that some component $C_i$ is an edge with two $\omega$-labels, i.e., $\Gamma$ has a loop based at $v$ in which both half-edges are labeled $\omega$. This proves the claim, and the lemma follows. \end{proof}

\begin{defn}
The \emph{excess} of an isotypic component $H^*(\Mb_{g,n})_{\lambda}$ is
\[
e(g,n,\lambda) = 3g + n + r,
\]
where $r$ is the number of rows in the Young diagram associated to $\lambda$. 
\end{defn}

\begin{cor}\label{cor:excessivenessnonTate}
    Suppose for all $g',n'$ such that $2g'+n'\leq 2g+n$, each non-Tate isotypic component of $H^*(\Mb_{g',n'})$ has excess greater than $3g+2n$. Then $H^*(\M_{g,n})$ is Tate.
\end{cor}
\begin{proof}
    By Lemma \ref{lem:excess}, the non-Tate part of the $\mathsf{GK}_{g,n}$ vanishes, and hence $H^*(\M_{g,n})$ is Tate by the spectral sequence \eqref{weightspectral}.
\end{proof}

Let $e^iH^*(\Mb_{g,n})$ denote the sum of all isotypic components of excess at least $i$.  Thus we have a decreasing excess filtration
\[
\cdots \supset e^{i} H^*(\Mb_{g,n}) \supset e^{i+1} H^*(\Mb_{g,n}) \supset \cdots 
\]
Let $\xi\colon \Mb_{g-1,n+2}\rightarrow \Mb_{g,n}$ and $\vartheta\colon \Mb_{g_1,n_1+1}\times \Mb_{g_2,n_2+1}\rightarrow \Mb_{g,n}$ be divisorial boundary gluing maps, where $g_1+g_2=g$ and $n_1+n_2=n$. 

\begin{lem} \label{lem:boundaryexcess}
The excess does not decrease under push-forward by $\xi$ and $\vartheta$, i.e., we have
\begin{enumerate}
\item  $\xi_* \big(e^iH^*(\Mb_{g-1,n+2})\big) \, \subset \, e^i H^*(\Mb_{g,n})$, and \label{item:exc1}
\item  $\vartheta_* \big(e^{i_1} H^* (\Mb_{g_1,n_1+1}) \otimes e^{i_2} H^* (\Mb_{g_2, n_2 + 1})\big) \, \subset \, e^{\max\{i_1,i_2\}}H^*(\Mb_{g,n})$. \label{item:exc2}
\end{enumerate}
\end{lem}

\begin{proof}
For \eqref{item:exc1}, suppose $x \in H^*(\Mb_{g-1,n+2})_\lambda)$ and let $r$ be the number of rows in $\lambda$. The representation generated by $\xi_*x$ is a quotient of the restricted representation 
\begin{equation} \label{res} \Res^{\ss_{n+2}}_{\ss_n}V_\lambda = \bigoplus_{\mu \subset \nu \subset \lambda} V_{\mu}
\end{equation}
where the sum runs over ways to successively remove two boxes from $\lambda$.
This sum carries a natural $\ss_2$-action.
If the boxes removed from $\lambda$ to obtain $\mu$ are in the same column, then $\ss_2$ acts on $V_{\mu}$ by a sign. If the two boxes removed are in the same row, then $\ss_2$ fixes $V_\mu$ and if the two boxes can be removed in either order, then the $\ss_2$-action switches the corresponding two factors in \eqref{res}.
In our geometric context, we have the additional piece of information that the map $\xi_*$ factors through $H^*(\Mb_{g-1,n+2})^{\ss_2}$.
Thus, the $V_{\mu}$ where two boxes are removed from the same column to form $\mu$ cannot appear in the representation generated by $\xi_*x$.
Consequently, $\xi_*x$ lies in the span of isotypic components having at least $r - 1$ rows.
The excess of such isotypic components is at least
\[3g + n + r - 1 = e(g-1,n+2, \lambda).\]

For \eqref{item:exc2}, let $y_i \in H^*(\Mb_{g_i, n_i+1})_{\lambda_i}$, and let $r_i$ be the number of rows of $\lambda_i$, for $i \in \{1,2\}$.
The representation generated by $\vartheta_*(y_1 \otimes y_2)$ is a quotient of \[
\Ind_{\ss_{n_1} \times \ss_{n_2}}^{\ss_n}((\Res^{\ss_{n_1+1}}_{\ss_{n_1}}V_{\lambda_1}) \boxtimes (\Res^{\ss_{n_2+1}}_{\ss_{n_2}}V_{\lambda_2})).\]
By the branching rules, each
$\Res^{\ss_{n_i+1}}_{\ss_{n_i}}V_{\lambda_i}$ is a sum of components with at least $r_i - 1$ rows, so every component of the induced representation also has at least $\max\{r_1,r_2\} - 1$ rows. Thus,
$\vartheta_*(y_1 \otimes y_2)$ lies in the span of 
isotypic components of excess at least 
\begin{align*} 3g + n + r_i - 1 &\geq 3g + n + e(g_i,n_i+1,\lambda_i) - (3g_i + n_i +1) - 1 \\
&= 3g_j + n_j - 2 + e(g_i,n_i+1,\lambda_i) \geq e(g_i,n_i+1,\lambda_i),
\end{align*}
where $i,j\in \{1,2\}$ and $i\neq j$.
The last inequality follows because if $g_j = 0$, then $n_j \geq 2$.
\end{proof}

Using semi-tautological extensions (STEs), as introduced in \cite[Definition~1.2]{CLP-STE}, we now give a concrete description of generators for $H^*(\Mb_{g,n})$ for small $g$ and $n$. Recall that an STE is a collection $S^*(\Mb_{g,n})$ of subrings of $H^*(\Mb_{g,n})$, one for each pair $(g,n)$, that contains the tautological rings and is closed under pullback by forgetting and permuting markings and under pushforward for gluing marked points. The STE generated by a cohomology group $H^{k'}(\Mb_{g',n'})$ is the smallest STE containing $H^{k'}(\Mb_{g',n'})$.
\begin{lem}\label{lem: onlygenus1}
    For $g\geq 2$ and $2g+n\leq 16$, $H^*(\Mb_{g,n})$ is contained in the STE generated by $H^{11}(\Mb_{1,11})$, and 
    all non-Tate parts of $H^*(\Mb_{g,n})$ are generated by boundary pushforwards from stable graphs whose only non-Tate decorations are on vertices of type $(1, m)$ for $m\leq 14$.
\end{lem}
\begin{proof}
We first treat the cases when $g = 2$. If $k$ is even, then $H^k(\Mb_{2,n})$ is tautological by \cite[Theorem 3.8]{Petersen}. 
Meanwhile, known vanishing results for odd cohomology \cite{ArbarelloCornalba,BergstromFaberPayne,CanningLarsonPayne} ensure $H^k(\Mb_{2,n}) = H_k(\Mb_{2,n}) = 0$ for odd $k \leq 11$. Moreover, $H_{13}(\Mb_{2,n})$ and $H_{15}(\Mb_{2,n})$ lie in the STE generated by $H^{11}(\Mb_{1,11})$ by \cite[Theorem 1.6 and Lemma 8.1]{CLP-STE}.
The only cases not covered by these results are $H^{13}(\Mb_{2,11})$, $H^{13}(\Mb_{2,12})$, and $H^{15}(\Mb_{2,12})$, which follow from \cite[Lemmas 7.1, 7.7, and 7.3]{CLP-STE} respectively.

We now treat the cases when $g \geq 3$. For each $k$, we have a right exact sequence
    \[
    H^{k-2}(\tilde{\partial \M_{g,n}})\rightarrow H^{k}(\Mb_{g,n})\rightarrow W_k H^{k}(\M_{g,n})\rightarrow 0. 
    \]
Inducting on $g$ and $n$, it suffices to show that $W_kH^k(\M_{g,n})$ is tautological for each of the claimed $(g, n)$.
By \cite[Lemma 4.3]{CLP-STE}, this follows whenever $\M_{g,n}$ has the Chow--K\"unneth generation property (CKgP) and
the Chow ring of $\M_{g,n}$ is generated by tautological classes. This condition is known to hold for the claimed $(g, n)$ by 
\cite[Theorem 1.4]{CL-CKgP} (for the cases with $3 \leq g \leq 6$), \cite[Theorem 1.10]{CLP-STE} (for the cases with $g = 7$) and
\cite[Theorem 1.2]{CanningLarson789} (for the case $g = 8$).
\end{proof}

We now record the $\ss_m$-isotypical types $\lambda$ occurring in the non-Tate parts of $H^*(\Mb_{1,m})$ for $m\leq 14$ \cite{BergstromData}. In each case, the excess is at least $25$.

\begin{table}[ht]
    \centering
    \begin{tabular}{|c||c|c|c|c|}
        \hline
         $m$ & 11 & 12 & 13 & 14 \\
         \hline
         $\lambda$ & $(1^{11})$ & $(2,1^{10})$ & $(4,1^9)$, $(2^2,1^9)$, $(2,1^{11})$,$(3,1^{10})$ &   \begin{tabular}{c} $(5,1^9)$, $(4,1^{10})$, $(3,1^{11})$, \\ $(4,2,1^8)$, $(3,2,1^9)$,$(2^2,1^{10})$\end{tabular} \\ 
         \hline
    \end{tabular}
    \caption{The $\ss_m$-isotypical types $\lambda$ occurring in the non-Tate parts of $H^*(\Mb_{1,m})$ for $m\leq 14$.}
    \label{Tatetable}
\end{table}

\begin{proof}[Proof of Theorem \ref{thm:Tatetype-markedpoints}]
By Lemmas \ref{lem:boundaryexcess}, \ref{lem: onlygenus1}, and Table \ref{Tatetable}, the non-Tate part of $H^*(\Mb_{g',n'})$ has excess at least $25>3g+2n$ for all $g'$ and $n'$ such that $2g'+n'\leq 2g+n$. The result thus follows from Corollary \ref{cor:excessivenessnonTate}.
\end{proof}

\begin{rem}\label{rem:tautological}
We note that the above argument in fact shows a slightly stronger statement. The tautological rings $R_{g,n}^*\subset H^*(\MM_{g,n})$ assemble into a modular sub-cooperad we denote by $R\subset H(\Mb)$. By functoriality of the Feynman transform we then obtain an injective map of differential graded vector spaces $\iota_{g,n}\colon\Feyn(R)(g,n)\to \GK_{g,n}$ from the Feynman transform of $R$ into the Feynman transform of $H(\MM)$. The arguments above show that $\iota_{g,n}$ is an isomorphism as long as $3g+2n<25$. Hence, in this range we may compute $H^*_c(\M_{g,n})\cong H^*(\GK_{g,n})$ solely from knowledge of the tautological rings. 
\end{rem}

\section{The thirteenth cohomology group of \texorpdfstring{$\Mb_{g,n}$}{Mbgn}} \label{h13sec}
In this section and the next, we prove Theorem~\ref{thm:H13} and Corollary~\ref{cor:chi13}, respectively, completing the second step in the proof of our main result. In \S\ref{reduction}, we reduce to proving Theorem~\ref{thm:H13} in the category of Hodge structures, where it suffices to produce generators and relations for $H^{12,1}(\Mb_{g,n})$. In \S\ref{candidate}, we describe an explicit set $\{ Z_{B\subset A} \}$ that we eventually show is a generating set for $H^{12,1}(\Mb_{g,n})$, and a basis when $g \geq 2$ . The proof is by induction on $g$ and $n$, using pullback formulas under tautological maps that we establish in \S\ref{tm}. To start the induction, we first show that $\{ Z_{B \subset A} \}$ is a generating set when $g = 1$  in \S\ref{basisg1} and describe the relations among these generators. The independence of $\{ Z_{B \subset A} \}$ for $g \geq 2$ follows easily in \S\ref{indsec}. The inductive step to show that $\{Z_{B \subset A}\}$ generates $H^{12,1}(\Mb_{g,n})$ in genus $g$ is more difficult when there are nontrivial relations in genus $g - 1$. We give a special argument for generation in genus $2$ in \S\ref{g2gens}, followed by the general argument for $g \geq 3$ in \S\ref{ind}.

\subsection{Reduction to the category of Hodge structures} 
\label{reduction} 
By \cite[Theorem 1.1]{CLP-STE}, the semisimplification of the rational cohomology $H^{13}(\Mb_{g,n})$ is a direct sum of copies of $\mathsf{L}\mathsf{S}_{12}$.

\begin{rem}\label{rem:homs}
Let $V$ be a rational polarized pure Hodge structure together with a continuous Galois action on $V \otimes \Q_\ell$.  Suppose we know that $V^\semis \cong \bigoplus \s$ where $\s$ is a simple rational polarized Hodge structure and $\s \otimes \Q_\ell$ is a simple $\ell$-adic Galois representation.

Because the category of polarized Hodge structures is semisimple, we have
\[\mathsf{S} \otimes \Hom_{\mathrm{Hodge}}(\mathsf{S}, V) \cong V.
\]
Although the category of $\ell$-adic Galois representations is not  semisimple, since $\s \otimes \Q_\ell$ is simple, there is at least an inclusion
\begin{equation} \label{in} \mathsf{S} \otimes \Hom_{\mathrm{Galois}}(\mathsf{S}, V) \hookrightarrow V.
\end{equation}
Therefore, if we give a basis for the $\Q$-vector space $\Hom_{\mathrm{Hodge}}(\mathsf{S}, V)$, and each element of this basis induces a map of $\ell$-adic Galois representations, then \eqref{in} is an isomorphism. In this situation, it follows in particular that $V \otimes \Q_\ell$ splits. 
\end{rem}

In this section, we describe a basis for $\Hom_{\mathrm{Hodge}}(\l \s_{12}, H^{13}(\Mb_{g,n}))$ in which all of the maps arise from algebraic geometry and hence induce maps of $\ell$-adic Galois representations. 

For $B \subset A \subset \{1, \ldots, n\}$, with $|B| = 10$, let $\varphi_{B \subset A}$ be the composition of the forgetful pullback and boundary pushforward 
\begin{equation} \label{dba} 
\varphi_{B \subset A} \colon \lstw \cong  H^{11}(\Mb_{1,B \cup p}) \otimes \mathsf{L} \xrightarrow{f_{B \cup p}^* \otimes \id^*} H^{11}(\Mb_{1,A \cup p}) \otimes \mathsf{L} \xrightarrow{\iota_{A*}( - \otimes 1)} H^{13}(\Mb_{g,n}).
\end{equation}
(See Section~\ref{candidate} for the precise definition of $f_{B \cup p}$ and $\iota_A$.) 

Choosing a distinguished generator $\omega \in H^{11,0}(\Mb_{1,11})$, we see that $\varphi_{B \subset A}$ is determined by an element $Z_{B \subset A} \in H^{12,1}(\Mb_{g,n})$, the image of $\omega \otimes 1$, where $\omega$ is the holomorphic form corresponding to the weight $12$ cusp form for $\SL_2(\zz)$. We then prove Theorem~\ref{thm:H13} via Remark~\ref{rem:homs}; we show that $\{ Z_{B \subset A} \}$ is a basis for $H^{12,1}(\Mb_{g,n})$, when $g \geq 2$, and directly calculate the $\ss_n$-action via signed permutations of this basis, giving $$H^{12,1}(\Mb_{g,n}) \cong \bigg ( \bigoplus_{m = 10}^n \Ind_{\ss_m \times \ss_{n-m}}^{\ss_n} \Big(\big(\mathrm{Res}^{\ss_{m+1}}_{\ss_m}K^{11}_{m+1}\big) \boxtimes \mathbf{1}\Big) \bigg) \otimes_\Q \CC.$$

\subsection{Generators for \texorpdfstring{$H^{12,1}(\Mb_{g,n})$ and the $\mathbb{S}_n$-action on them}{H13}}\label{candidate}
We recall the description of $H^{11}(\Mb_{1,n})$ in \cite[Proposition 2.3]{CanningLarsonPayne}. Let $\omega \in H^{11,0}(\Mb_{1,11})$ be the holomorphic form given, via Eichler-Shimura, by the cusp form $\tau$ of level 1 and weight 12 as in \cite[\S2]{FaberPandharipande13}. Given an ordered subset $P \subset \{1, \ldots, n\}$ of size $11$, we define $\omega_{P} \in H^{11,0}(\Mb_{1, n})$ to be the pullback of $\omega$ under the forgetful map \[f_P \colon \Mb_{1,n} \to \Mb_{1,P}.
\]
Then $\{\omega_P\}$ generates $H^{11,0}(\Mb_{1,n})$ and the relations are given by
\begin{equation} \label{h11rels} \omega_{\sigma(P)} = \sign(\sigma) \omega_P \qquad \text{and} \qquad 0 = \sum_{j = 0}^{11}(-1)^j \omega_{\{a_0, \ldots, \widehat{a_j}, \ldots, a_{11}\}}
\end{equation}
for any permutation $\sigma$ of $P$ and any subset $\{a_0, \ldots, a_{11}\} \subset \{1, \ldots, n\}$ of $12$ markings.

Now, suppose $g \geq 1$. Consider the parametrized boundary divisors of $\Mb_{g,n}$ of the form
\[\iota_{A} \colon \Mb_{1, A \cup p} \times \Mb_{g-1,A^c \cup p'} \to \Mb_{g,n}.\]
For each such boundary divisor, there is a push forward map
\[
\iota_{A*} \colon H^{11,0}(\Mb_{1,A\cup p}) \otimes H^0(\Mb_{g-1,A^c \cup p'}) \subset H^{11,0}(\Mb_{1, A \cup p} \times \Mb_{g-1,A^c \cup p'}) \to H^{12,1}(\Mb_{g,n}). 
\]
Pushing forward the relations in \eqref{h11rels} induces relations on $\{\iota_{A*}(\omega_{P}\otimes 1)\}$, from which one sees that the image of $\iota_{A*}$ is spanned by classes of the form
\begin{equation} \label{zdef} 
Z_{B \subset A} := \iota_{A*}(\omega_{B \cup p}\otimes 1),
\end{equation}
where $B \subset A \subset \{1, \ldots, n\}$ with $B$ increasing of size $10$. In the remainder of this section, we will show that $\{Z_{B\subset A}\}$ generates $H^{12,1}(\Mb_{g,n})$ and is a basis when $g \geq 2$.

The $\ss_n$-action on $H^{12,1}(\Mb_{g,n})$ permutes $\{Z_{B \subset A}\}$, up to a sign. More precisely, for $\sigma \in \ss_n$, 
\begin{equation}\label{eq:snaction}
\sigma(Z_{B \subset A}) = \sgn(\sigma_B) Z_{\sigma(B) \subset \sigma(A)},
\end{equation}
where $\sigma_B$ is the permutation of $B$ induced by the ordering of $\sigma(B) \subset \{1, \ldots, n\}$. Note, in particular, that the $\ss_n$-action preserves the subspace spanned by those $Z_{B \subset A}$ where $|A|$ is fixed. For a fixed $A \subset \{1, \ldots, n\}$ of size $m$, the group of permutations of $A$ 
is a copy of $\mathbb{S}_m$ that acts
on $H^{11,0}(\Mb_{1,A \cup p})$ by fixing $p$. This is the $\mathbb{S}_m$-representation $\mathrm{Res}^{\ss_{m+1}}_{\ss_m} K^{11}_{m+1}$. 
Meanwhile, the subgroup $\mathbb{S}_{n-m}$ of permutations that fix every element of $A$ fixes $Z_{B \subset A}$ for each $B \subset A$. Thus, the vector space with basis $\{Z_{B \subset A}\}$ on which $\ss_n$ acts via \eqref{eq:snaction} decomposes into a sum of induced representations according to the size of $A$:
\[
\bigoplus_{m = 10}^n \Ind_{\ss_m \times \ss_{n-m}}^{\ss_n} \Big(\big(\mathrm{Res}^{\ss_{m+1}}_{\ss_m}K^{11}_{m+1}\big) \boxtimes \mathbf{1}\Big).
\]

\subsection{Pullbacks of \texorpdfstring{$Z_{B \subset A}$}{ZBA} along tautological maps} \label{tm}

The results here are analogous to the first three lemmas in \cite[Section 3]{ArbarelloCornalba}, which describe the tautological pullbacks of boundary divisors and $\psi$-classes. 

To begin, let $\pi \colon \Mb_{g,n+1} \to \Mb_{g,n}$ be the map that forgets the marked point $q$.

\begin{lem}\label{forgetq}
We have
$\pi^*Z_{B \subset A} = Z_{B \subset A} + Z_{B \subset A \cup q}.$
\end{lem}
\begin{proof}
Note that there is a fiber diagram
\begin{center}
\begin{tikzcd}
(\Mb_{1, A \cup p} \times \Mb_{g-1,A^c \cup \{p', q\}}) \coprod (\Mb_{1, A \cup \{p,q\}} \times \Mb_{g-1,A^c \cup p'}) \arrow{d} \arrow{r} & \Mb_{g,n+1} \arrow{d} \\
\Mb_{1, A \cup p} \times \Mb_{g-1,A^c \cup p'} \arrow{r}{\iota_A} & \Mb_{g,n}
\end{tikzcd}
\end{center}
and apply the push-pull formula.
\end{proof}

Next let $\xi \colon \Mb_{g-1,n+2} \to \Mb_{g,n}$ be the map that glues two points labeled $q$ and $r$.
\begin{lem} \label{xilem}
We have
$
 \xi^* Z_{B \subset A} = Z_{B \subset A}   
$.
\end{lem}

\begin{proof}
Consider the fiber diagram
\begin{center}
    \begin{tikzcd}
(\Mb_{1,A \cup p} \times \Mb_{g-2, A^c \cup \{p',q,r\}}) \coprod (\Mb_{0,A \cup \{p,q,r\}} \times \Mb_{g-1,A^c \cup p'}) \arrow{d}[swap]{\xi'} \arrow{r}{\iota'} & \Mb_{g-1,n+2} \arrow{d}{\xi} \\
\Mb_{1,A \cup p} \times \Mb_{g-1,A^c \cup p'} \arrow{r}[swap]{\iota_A} & \Mb_{g,n}.
    \end{tikzcd}
\end{center}
Note that $H^{11}(\Mb_{0, A \cup \{p,q,r\}}) = 0$.
Using the push-pull formula, we then see that
\[\xi^*Z_{B \subset A} = \xi^*\iota_{A*}(\omega_{B \cup p} \otimes 1) = \iota'_*\xi'^*(\omega_{B\cup p} \otimes 1)= Z_{B \subset A}. \qedhere \]
\end{proof}

Finally, let $\vartheta \colon \Mb_{a, S \cup s} \to \Mb_{g,n}$ be the map that attaches a fixed genus $g - a$ curve at the point $s$. The pullback along $\vartheta$ determines the $H^{12,1} \otimes H^{0,0}$ K\"unneth components of pullbacks of classes to boundary divisors with a disconnecting node.
Given an ordered set $B$, if $b \in B$ and $s \notin B$, we write $B[b \mapsfrom s]$ for the ordered set where $s$ replaces $b$.

\begin{lem}
We have 
\begin{equation} \label{theta}
 \vartheta^* Z_{B \subset A} =   \begin{cases} Z_{B \subset A} & \text{if $A \subset S$ and $A \neq S$ if $a = 1$}  \\
- \psi_s \cdot \omega_{B \cup p} & \text{if $A = S$ and $a = 1$} \\
 Z_{B \subset (S\smallsetminus A^c) \cup s} & \text{if $A^c \subset S$ and $S^c \cap B = \varnothing$} \\
 Z_{B[b \mapsfrom s] \subset (S \smallsetminus A^c) \cup s} & \text{if $A^c \subset S$ and $S^c \cap B = \{b\}$} \\
 0 & \text{otherwise.} \end{cases}
\end{equation}
\end{lem}
\begin{proof}
We consider the fiber product $F$ of $\iota_A$ and $\vartheta$:
\begin{center}
    \begin{tikzcd}
F \arrow{d}{\vartheta'} \arrow{r}{\iota'} & \Mb_{a,S \cup s} \arrow{d}{\vartheta} \\
\Mb_{1,A \cup p} \times \Mb_{g-1,A^c \cup p'} \arrow{r}{\iota_A} & \Mb_{g,n}.
    \end{tikzcd}
\end{center}
Note that $F$ is empty unless $A \subset S$ or $A^c \subset S$.
First, suppose $A \subset S$ and, if $a = 1$, then $A \neq S$. Then $F = \Mb_{1,A \cup p} \times \Mb_{a-1, (S \smallsetminus A) \cup \{p', s\}}$, and we have
\[\vartheta^*Z_{B \subset A} = \vartheta^*\iota_{A*}(\omega_{B \cup p} \otimes 1) = \iota'_*\vartheta'^*(\omega_{B \cup p} \otimes 1) = \iota'_*(\omega_{B \cup p} \otimes 1) = Z_{B \subset A}.\]
If $a = 1$ and $A = S$, then use the self-intersection formula for $\iota_{A}^*\iota_{A*}$ and take the $\M_{1, A \cup s}$ K\"unneth component to obtain $\vartheta^* Z_{B \subset A} = -\psi_s \cdot \omega_{B \cup p}$.

When $A^c \subset S$, we have $F = \Mb_{1,(S \smallsetminus A^c) \cup \{p, s\} } \times \Mb_{a-1, A^c \cup p'}$. In this case, the pullback $\vartheta'^*(\omega_{B \cup p} \otimes 1)$ equals the pullback of $\omega \in H^{11,0}(\Mb_{1, B \cup p})$ along the composition
\begin{equation} \label{compo}\Mb_{1,(S \smallsetminus A^c) \cup \{p, s\}} \times \Mb_{0, S^c \cup s'} \to \Mb_{1, A \cup p} \to \Mb_{1, B \cup p}. 
\end{equation}
(To keep track of the markings above, note $(S \smallsetminus A^c) \cup S^c = (A^c)^c = A$.)
For the remaining cases, if $|S^c \cap B| > 1$, then the pull back of $\omega$ from $\Mb_{1, B \cup p}$ vanishes, since \eqref{compo} factors through a proper divisor in $\Mb_{1,B \cup p}$. If $S^c \cap B = \{b\}$, then $\omega$ pulls back to $\omega_{B[b \mapsfrom s]  \cup p}$. If $S^c \cap B = \varnothing$, then $\omega$ pulls back to $\omega_{B \cup p}$.
\end{proof}

We also require the $H^{11,0} \otimes H^{1,1}$ K\"unneth components of pullbacks of classes to boundary divisors with a disconnecting node. We compute these using the same strategy as in the previous lemma. 
The $H^{11, 0} \otimes H^{1,1}$ K\"unneth component is only non-vanishing if the genus of the first factor is $1$.
It therefore suffices to consider the pullbacks along gluing maps of the form $\iota_S \colon \Mb_{1, S \cup s} \times \Mb_{g-1, S^c \cup s'} \to \Mb_{g,n}$. 
\begin{lem} \label{ok}
The $H^{11,0} \otimes H^{1,1}$ K\"unneth component of $\iota_S^*(Z_{B \subset A})$ is
\[\begin{cases} -\omega_{B \cup p} \otimes \psi_{p'} & \text{if $A = S$} \\
\omega_{B \cup p} \otimes \delta_{g-1,A^c} & \text{if $A^c \subsetneq S^c \subset B^c$} \\
 0 & \text{otherwise.}\end{cases} \]
\end{lem}
\begin{proof}
We consider the fiber product $F$ of $\iota_A$ and $\iota_S$:
\begin{center}
    \begin{tikzcd}
F \arrow{d}{\iota_S'} \arrow{r}{\iota_A'} & \Mb_{1,S \cup s} \times \Mb_{g-1,S^c \cup s'} \arrow{d}{\iota_S} \\
\Mb_{1,A \cup p} \times \Mb_{g-1,A^c \cup p'} \arrow{r}{\iota_A} & \Mb_{g,n}.
    \end{tikzcd}
\end{center}
We have
\[\iota_S^*(Z_{B \subset A}) = \iota_S^*\iota_{A*}(\omega_{B \cup p} \otimes 1) = \iota_{A*}' \iota_S'^*(\omega_{B \cup p} \otimes 1).\]
In order to contribute a non-zero $H^{11,0} \otimes H^{1,1}$ K\"unneth component, one factor of the fiber product must be $\Mb_{1,S \cup s}$ and $\iota_S'$ must be a forgetful map on that factor. Thus, we only obtain a non-zero K\"unneth component when $S \subset A$.

If $A = S$, then the claim follows from the self-intersection formula. If $S \subsetneq A$ (or equivalently $A^c \subsetneq S^c$), then $A = S \cup (S^c \smallsetminus A^c)$.
In this case, the fiber product is
\[F = \Mb_{1, S \cup s} \times \Mb_{0, (S^c \smallsetminus A^c) \cup \{s',p\}} \times \Mb_{g-1, A^c \cup p'}.\]
Thus, $\iota_{A*}' \iota_S'^*(\omega_{B \cup p} \otimes 1) = \iota_{A*}'(\omega_{B \cup p} \otimes 1 \otimes 1)$. The map $\iota_{A}'$ is the identity on $\Mb_{1,S \cup s}$ and the gluing map on the second two factors of $F$. The pushforward of $1 \otimes 1$ along this gluing map is precisely $\delta_{g-1,A^c}$. 
\end{proof}

\subsection{Relations and preferred basis in genus 1} \label{basisg1}
In this section, we determine the relations among the generators $Z_{B \subset A}$ in genus $1$ and give a preferred basis. 

In genus $1$, for $Z_{B \subset A}$ to be defined, we must have $|A^c| \geq 2$. In particular, 
on $\Mb_{1,12}$, there are
$\binom{12}{10} = 66$ classes $Z_{B \subset A}$ with $B = A$ of size $10$. 
However, we know that $H^{12,1}(\Mb_{1,12})$ is $11$-dimensional, as it is Poincar\'e dual to $H^{0,11}(\Mb_{1,12})$. For convenience, we will write $Z_B := Z_{B \subset B}$.
The following lemma determines the relations among the $Z_B$ in $H^{12,1}(\Mb_{1,12})$.
This lemma has appeared previously in \cite[Proposition 5]{GraberPandharipande}, but we include a proof in our notation for completeness.

\begin{lem} \label{12basis}
For any $1 \leq i < j < k \leq 12$, we have
\begin{equation} \label{m12rel} 0 = (-1)^{i+j} Z_{\{1, \ldots, \widehat{i}, \ldots, \widehat{j}, \ldots, {12}\}} - (-1)^{i+k}
Z_{\{1, \ldots, \widehat{i}, \ldots, \widehat{k}, \ldots, 12\}} + (-1)^{j+k}
Z_{\{1, \ldots, \widehat{j}, \ldots, \widehat{k}, \ldots, 12\}}.
\end{equation}
Furthermore,  $\{Z_{B}\colon 1 \notin B\}$ form a basis for $H^{12,1}(\Mb_{1,12})$. 
\end{lem}
\begin{proof} By \cite[Proposition 2.2]{CLP-STE}, we have $W_{13}H^{13}(\M_{1,n}) = 0$. Hence, $\{Z_B\}$ spans  $H^{12,1}(\Mb_{1,12})$. Once the relations \eqref{m12rel} are established, it follows that $\{Z_{B}: 1 \notin B\}$ spans $H^{12,1}(\Mb_{1,12})$. The fact that it is a basis follows by dimension counting; we have $\dim H^{12,1}(\Mb_{1,12}) = \dim H^{0,11}(\Mb_{1,12})$ by Poincar\'e duality, and the latter is $11$ by \cite{Getzler}.

To prove \eqref{m12rel}, it suffices to show that the right-hand side $R_{ijk}$ of \eqref{m12rel} pairs to $0$ with every element in a basis for $H^{0,11}(\Mb_{1,12})$.
Let $\pi_i\colon \Mb_{1,12} \to \Mb_{1,11}$ forget the $i$th point.
We know that $\{ \pi_i^*\omega : 2 \leq i \leq 12 \} \subset   H^{11,0}(\Mb_{1,12})$ 
is a basis, by \cite[Corollary 2.4]{CanningLarsonPayne}. Let $\overline{\omega} \in H^{0,11}(\Mb_{1,11})$ be the class Poincar\'e dual to $\omega \in H^{11,0}(\Mb_{1,11})$. (Up to rescaling, it is the complex conjugate of $\omega$.)
Then $\{\pi_i^*\overline{\omega} : 2 \leq i \leq 12\}$ is a
basis for $H^{0,11}(\Mb_{1,12})$.  

For $\ell \neq i, j$, the composition $\pi_{\ell} \circ \iota_{\{i,j\}^c} $ factors through a proper boundary divisor on $\Mb_{1,11}$, which has no $H^{11}$. Using the push-pull formula, it follows that
\[Z_{\{i,j\}^c} \cdot \pi_{\ell}^*\overline{\omega} = \omega \cdot \iota_{\{i,j\}^c}^*(\pi_{\ell}^*\overline{\omega}) = 0.\]
It remains to show that $R_{ijk} \cdot \pi_{\ell}^* \overline{\omega}= 0$ when $\ell = i, j$ or $k$. Note that $R_{ijk}$ is invariant under the $\mathbb{S}_3$-action permuting $i, j, k$. In particular, the transposition that flips $\{i, j, k\} \smallsetminus  \ell$ acts by $1$ on $R_{ijk}$ but by $-1$ on $\pi_{\ell}^* \overline{\omega}$. Since this transposition acts trivially on $H^{24}(\Mb_{1,12})$ (spanned by the point class), it
follows that $R_{ijk} \cdot \pi_{\ell}^*\overline{\omega} = -R_{ijk} \cdot \pi_{\ell}^*\overline{\omega}$, so 
$R_{ijk} \cdot \pi_{\ell}^*\overline{\omega}  = 0$.

It is also easy to see directly that the basis $Z_{\{1,i\}^c}$ is dual to $\pi_i^*\overline{\omega}$. Indeed,
\begin{align*}
Z_{\{1,i\}^c} \cdot \pi_j^*\overline{\omega} = (\iota_{\{1,i\}^c})_*(\iota_{\{1,i\}^c}^*\pi_i^*\omega) \cdot \pi_j^*\overline{\omega} = (\iota_{\{1,i\}^c})_*(\iota_{\{1,i\}^c}^*\pi_i^*\omega \cdot \iota_{\{1,i\}^c}^*\pi_j^*\overline{\omega}) = \begin{cases} 1 & \text{if $j = i$} \\ 0 & \text{otherwise.}\end{cases}
\end{align*}
To see this, note that $\phi_i^*\pi_i^*\omega = \omega \in H^{11,0}(\Mb_{1,\{1,i\}^c \cup p})$ and the pullback $\iota_{\{1,i\}^c}^*\pi_j^*\overline{\omega}$ vanishes unless $i = j$, in which case it is dual to $\omega$.
\end{proof}

Pulling back the relations \eqref{m12rel} in $H^{12,1}(\Mb_{1,12})$ determines relations among the $Z_{B \subset A}$ in $H^{12,1}(\Mb_{1,n})$ for $n \geq 12$. Our next goal is to show that these are the \emph{only} relations. Let $E \subset \{1, \ldots, n\}$ be a subset of $12$ markings and write $f_E \colon \Mb_{1,n} \to \Mb_{1,E}$ for the forgetful map. For $B \subset E$, repeated application 
of Lemma \ref{forgetq} yields the formula
\begin{equation} \label{fE}
f_E^*Z_{B} = \sum_{\substack{B \subset A \\ E \smallsetminus B \subset A^c} } Z_{B \subset A}.
\end{equation}

\begin{lem} \label{indep}
Let $f_E\colon \Mb_{1,n} \to \Mb_{1,E} = \Mb_{1,12}$ be the forgetful map. Let \[\pb_E := f_E^*H^{12,1}(\Mb_{1,E}) \subset H^{12,1}(\Mb_{1,n})\]
be the subspace of classes pulled back from $\Mb_{1,E}$. The subspaces $\pb_E$ as $E \subset \{1,\ldots, n\}$ ranges over all subsets of size $12$
are all independent. Their span gives a subrepresentation
\[\pb :=\bigoplus_{|E| = 12} \pb_E \cong \mathrm{Ind}^{\mathbb{S}_n}_{\mathbb{S}_{12} \times \mathbb{S}_{n-12}}(V_{2,1^{10}} \boxtimes \mathbf{1}) \subset H^{12,1}(\Mb_{1,n}).\]
\end{lem}
\begin{proof}
We first treat the case $n = 13$. 
There are $\binom{13}{2} = 78$ boundary divisors of the form 
\[D_F := \Mb_{1,F \cup p} \times \Mb_{0, F \cup p'} \xrightarrow{\iota_F} \Mb_{1,13}\] where $|F| = 11$, which have non-zero $H^{12,1}$.
We will show that the subspaces $\pb_E$ are independent by showing that the composition
\begin{equation} \label{ef} \bigoplus_{|E| = 12} H^{12,1}(\Mb_{1,E}) \xrightarrow{\oplus f_E^*} H^{12,1}(\Mb_{1,13}) \xrightarrow{\oplus \iota_F^*} \bigoplus_{|F| = 11} H^{12,1}(\Mb_{1,F \cup p})
\end{equation}
is injective. 
The term on the left is just an induced representation, which we decompose using Pieri's rule:
\[\mathrm{Ind}^{\mathbb{S}_{13}}_{\mathbb{S}_{12}} (V_{2, 1^{10}}) =  V_{3,1^{10}} \oplus V_{2,2,1^{9}} \oplus V_{2,1^{11}}.\]
These irreducible representations have dimensions $66, 65,$ and $12$ respectively. In particular, if \eqref{ef} has a non-trivial kernel, the kernel must have dimension at least $12$.

The $E, F$ component of the block matrix representing \eqref{ef} is $0$ unless $F \subset E$ in which case it is given by the map $H^{12,1}(\Mb_{1,E}) \to H^{12,1}(\Mb_{1,F \cup p})$ that identifies $p$ with the unique marking of $E$ not in $F$. Projecting onto the $12$ terms on the right sum in \eqref{ef} where $1 \notin F$ shows  that the rank of \eqref{ef} is at least $11 \cdot 12$. Hence, the kernel of \eqref{ef} has dimension at most $11$, and is therefore trivial.

Now assume $n \geq 14$. This time, we use the boundary divisors $D_F$ with $|F| = 12$. Consider
\begin{equation} \label{ef2}
\bigoplus_{|E| = 12} H^{12,1}(\Mb_{1,E}) \xrightarrow{\oplus f_E^*} H^{12,1}(\Mb_{1,n})  \xrightarrow{\oplus \iota_F^*} \bigoplus_{|F| = 12} H^{12,1}(\Mb_{1,F \cup p}).
\end{equation}
If $|E \cap F| \leq 10$, then $f_E$ sends $D_F$ to a proper divisor on $\Mb_{1,E}$, which has no $H^{13}$, so the map is zero. Two cases remain where the $E, F$ block in \eqref{ef2} is non-zero:
\begin{itemize}
\item \textit{Case 1:} $|E \cap F| = 11$. Let $i \in E$ be the unique element not in $F$ and $j \in F$ the element not in $E$. Then $f_E$ sends $D_F$ to all of $\Mb_{1,E}$. The pullback map $H^{12,1}(\Mb_{1,E}) \to H^{12,1}(\Mb_{1,F \cup p})$ is essentially the pullback from forgetting the marking $j$, after identifying the marking $i$ with $p$. In particular, the image of $H^{12,1}(\Mb_{1,E}) \to H^{12,1}(\Mb_{1,F \cup p})$ lies in the subspace $\pb_{F \smallsetminus j \cup p} \subset H^{13}(\Mb_{1, F \cup p})$.
\item \textit{Case 2:} $E = F$. The map $H^{12,1}(\Mb_{1,E}) \to H^{12,1}(\Mb_{1,F \cup p})$ is the pullback along forgetting $p$. In other words $H^{12,1}(\Mb_{1,E})$ is sent isomorphically to $\pb_F \subset H^{12,1}(\Mb_{1,F \cup p})$.
\end{itemize}
By the $n = 13$ case, the subspace $\pb_F \subset H^{12,1}(\Mb_{1,F \cup p})$ in Case 2 is independent from the contributions from Case 1. From this, it is clear that \eqref{ef2} is an injection. The left-hand side of \eqref{ef2} is the claimed induced representation.
\end{proof}

We now give some classes that generate the complement of $\mathrm{PB}$.

\begin{lem} \label{maing1}
The classes
\begin{equation} \label{z1} \{Z_{B \subset A} \in H^{12,1}(\Mb_{1,n}): |A^c| \geq 3\}
\end{equation}
have independent image in
 $H^{12,1}(\Mb_{1,n})/\pb$. In other words, the subspaces $\pb_E$ and the classes \eqref{z1} are independent and generate $H^{12,1}(\Mb_{1,n})$.
\end{lem}
\begin{proof}
Consider the map
\begin{equation} \label{112}
H^{12,1}(\Mb_{1,n}) \rightarrow \bigoplus_{|F^c| \geq 3} H^{13}(D_F) \rightarrow \bigoplus_{|F^c| \geq 3} H^{11,0}(\Mb_{1,F \cup p}) \otimes H^{1,1}(\Mb_{0, F^c \cup p'}),
\end{equation}
where the first map restricts to the specified boundary components and the second map projects onto the $H^{11,0} \otimes H^{1,1}$ K\"unneth components. Note that if $|F^c| \geq 3$ and $|E| = 12$, then $f_E$ sends $D_F$ to a proper divisor in $\Mb_{1,E}$ or factors through the projection $D_F \to \Mb_{1,F \cup p}$. In either case, the image of $f_E^*H^{12,1}(\Mb_{1,E}) \subset H^{12,1}(\Mb_{1,n}) \to H^{12,1}(D_F)$ lies in the $H^{12,1} \otimes H^0$ component. Hence, the composition \eqref{112} sends $\pb \subset H^{12,1}(\Mb_{1,n})$ to $0$.

Next we show that the restriction of \eqref{112} to the span of the classes in \eqref{z1} is injective. We do so by showing it is
represented by a block lower triangular matrix.
The columns of our matrix correspond to the classes $Z_{B \subset A}$, ordered so that $|A^c|$ is non-decreasing as we go from left to right. The rows of our matrix correspond to the spaces $H^{11,0}(\Mb_{1, F \cup p}) \otimes H^{1,1}(\Mb_{0, F^c \cup p'})$ so that $|F^c|$ is non-decreasing as we go from top to bottom.
By Lemma \ref{ok}, the entry in the block for column $B \subset A$ and row $F$ is given by
\[\text{$H^{11,0} \otimes H^{1,1}$ K\"unneth component of }\iota_{F}^*(Z_{B \subset A}) =\begin{cases} -\omega_{B \cup p} \otimes \psi_{p'} & \text{if $A = F$} \\
\omega_{B \cup p} \otimes \delta_{0,A^c} & \text{if $A^c \subsetneq F^c$} \\
 0 & \text{otherwise.}\end{cases} \]
 In particular, the matrix is lower triangular and evidently full rank.

 Finally, we must see that the classes in \eqref{z1} together with $\pb$ generate all of $H^{12,1}(\Mb_{1,n})$. Since $W_{13}H^{13}(\M_{1,n}) = 0$ (by \cite[Proposition 7]{Petersenappendix}), we know $H^{12,1}(\Mb_{1,n})$ is spanned by $Z_{B \subset A}$, so it suffices to see that each $Z_{B \subset A}$ with $|A^c| \geq 2$ lies in the span of $\pb$ and classes in \eqref{z1}. Given $Z_{B \subset A}$ with $|A^c| = 2$, let $E = B \cup A^c$. Then \eqref{fE} shows that $Z_{B \subset A}$ is equal to $f_E^*Z_{B}$ minus terms in \eqref{z1}.
 \end{proof}

\begin{cor}
\label{gid}
As an $\mathbb{S}_n$-equivariant  Hodge structure or $\ell$-adic Galois representation,
\[H^{13}(\Mb_{1,n}) \cong \mathsf{L}\mathsf{S}_{12} \otimes \Big(\mathrm{Ind}^{\mathbb{S}_n}_{\mathbb{S}_{12} \times \mathbb{S}_{n-12}}(V_{2,1^{10}} \boxtimes \mathbf{1}) \oplus \bigoplus_{10 \leq k \leq n-3} \mathrm{Ind}^{\mathbb{S}_n}_{\mathbb{S}_k \times \mathbb{S}_{n-k}}((V_{k-10,1^{10}} \oplus V_{k-9,1^{9}})
\boxtimes \mathbf{1}) \Big). \]
\end{cor}

\begin{rem}
The expression in Corollary \ref{gid} is implicitly
determined by \cite[Theorem 2.6]{Getzler}. The main contribution of this section is our geometric description of the generators and the explicit independent set of Lemmas \ref{maing1} and \ref{altg1}.
\end{rem}

Lemma \ref{maing1} does not involve any choices, but it is also convenient to write down an explicit subset of $\{Z_{B \subset A}\}$ that is a basis.

\begin{lem}\label{altg1}
The set $\{Z_{B \subset A}: |A^c| \geq 3, \mbox{ or } |A^c| = 2  \mbox{ and } \min(A^c) < \min(B)\}$ is a basis for $H^{12,1}(\Mb_{1,n})$. 
\end{lem}
\begin{proof}
Lemma \ref{12basis} establishes the case $n = 12$ because, when there are only $12$ points, $A = B$ and $\min(A^c) < \min(B)$, implies $1 \in A^c$, or equivalently $1 \notin B$.
More generally, when $|A^c| = 2$, the pullbacks $f_{B \cup A^c}^*Z_{B}$ with $\min(A^c) < \min(B)$ form a basis for $\pb_{B \cup A^c}$. Lemma \ref{maing1} then says that the $Z_{B \subset A}$ with $|A^c| \geq 3$, together with $f_{B \cup A^c}^*Z_{B}$ with $|A^c| = 2$ and $\min(A^c) < B$, form a basis for $H^{12,1}(\Mb_{1,n})$. By \eqref{fE}, we have \[f_{B \cup A^c}^*Z_{B} = Z_{B \subset A} + (\text{terms $Z_{B \subset A}$ with $|A^c| \geq 3$}),\]
which completes the proof.
\end{proof}

\subsection{Independence of \texorpdfstring{$\{Z_{B \subset A}\}$}{ZBA} for \texorpdfstring{$g \geq 2$}{g>=2}} \label{indsec}
We now use the preferred basis in genus $1$ to prove that $\{Z_{B \subset A}\}$ is independent when $g \geq 2$.
\begin{lem} \label{g2indep}
For $g \geq 2$, the subset $\{Z_{B \subset A} \} \subset H^{12,1}(\Mb_{g,n})$ is independent.
\end{lem}
\begin{proof}
We first prove the case $g = 2$.
Consider $\xi^{*}\colon H^{13}(\Mb_{2,n}) \to H^{13}(\Mb_{1,n+2})$. Let us label the points on $\Mb_{1,n+2}$ by $\{x, y, 1, \ldots, n\}$ and suppose they are ordered as written, so that $x, y$ are minimal. By Lemma \ref{xilem}, we have $\xi^*Z_{B \subset A} = Z_{B \subset A}$. The complement of $A$ in $\{x, y,1, \ldots, n\}$ automatically contains $x$ and $y$. In particular, if the complement of $A$ in $\{x, y,1, \ldots, n\}$ only consists of two elements, we necessarily have $\min(A^c) < \min(B)$. Lemma~\ref{altg1} implies that the $\xi^*Z_{B \subset A}$ are all independent, so the $Z_{B \subset A}$ must be independent.

Now suppose $g > 2$. By induction on $g$, we may assume $\{ Z_{B \subset A} \} \subset H^{13}(\Mb_{g-1,n+2})$ is independent. Applying Lemma \ref{xilem} shows that the pullback of the subset $\{Z_{B \subset A} \} \subset H^{13} (\Mb_{g,n})$ is independent in $H^{13}(\Mb_{g-1,n+2})$, and hence the subset is itself independent. 
\end{proof}

\subsection{Generators in genus 2} \label{g2gens}
Using Petersen's results on cohomology of local systems on $\A_2$ \cite{Petersenlocalsystems} and computer calculations, Bergstr\"om and Faber determined $H^*(\Mb_{2,n})$ for a range of values of $n$ \cite{BergstromData}. 
The results we need are:
\begin{align} \label{d1}
\dim H^{12,1}(\Mb_{2,10}) &= 1 &\qquad \dim H^{12,1}(\Mb_{2,11}) &= 22 \\
\dim H^{12,1}(\Mb_{2,12}) &= 264 &\qquad \dim H^{12,1}(\Mb_{2,13}) &= 2288. \label{d2}
\end{align}

\begin{lem} \label{g2span}
For all $n$, the subset $\{Z_{B \subset A}\}$ spans $H^{12,1}(\Mb_{2,n})$.
\end{lem}
\begin{proof}
For any $n$, there is an exact sequence
\[
H^{11}(\tilde{\partial \M_{2,n}})\rightarrow H^{13}(\Mb_{2,n})\rightarrow W_{13}H^{13}(\M_{2,n})\rightarrow 0.
\]

For $n \leq 9$, we have $H^{13}(\Mb_{2,n}) = 0$, either by \cite[Theorem 1.4]{CL-CKgP} or \cite{BergstromData}.
In each case with $10 \leq n \leq 13$, a straightforward count shows that the number of $Z_{B \subset A}$ is the dimension of $H^{12,1}(\Mb_{2,n})$ listed in \eqref{d1} and \eqref{d2}. By Lemma \ref{g2indep}, it follows that the $Z_{B \subset A}$ must be a basis for $H^{12,1}(\Mb_{2,n})$. In particular, since the classes $Z_{B\subset A}$ are pushed forward from the boundary, $W_{13}H^{13}(\M_{2,n}) = 0$ for $n \leq 13$.

For $n \geq 14$, $W_{13} H^{13}(\M_{2,n})$ is generated by pullbacks from $W_{13} H^{13}(\M_{2,A})$, for subsets $A \subset \{1, \ldots, n\}$ of size $|A|\leq 13$ \cite[Lemma 3.1(b)]{CLP-STE}. Hence, $W_{13}H^{13}(\M_{2,n}) = 0$ for all $n$, and it follows that
\[
H^{11}(\tilde{\partial \M_{2,n}})\rightarrow H^{13}(\Mb_{2,n})
\] is surjective. The normalization of the boundary $\tilde{\partial \M_{2,n}}$ is a disjoint union of moduli spaces $\Mb_{\Gamma}$ corresponding to stable graphs $\Gamma$ of genus $2$ with $n$ legs and exactly one edge. For such $\Gamma$, $H^{11}(\Mb_{\Gamma})=0$ if $\Gamma$ has no vertices of genus $1$. If $\Gamma$ has two vertices of genus $1$, then the image of 
\[
H^{11}(\Mb_{\Gamma})\rightarrow H^{13}(\Mb_{2,n})
\]
is in the span of the $Z_{B\subset A}$ by definition. 

Let $x, y$ be the last two points on $\Mb_{1,n+2}$ and let $\xi\colon \Mb_{1,n+2} \to \Mb_{2,n}$ glue $x$ and $y$.
It remains to show that the image of $\xi_*\colon H^{11,0}(\Mb_{1,n+2}) \to H^{12,1}(\Mb_{2,n})$ is contained in the span of the $Z_{B \subset A}$.
Given $\omega_P \in H^{11,0}(\Mb_{1,n+2})$ let $P' = P \cup \{x, y\}$, which has size $11,12,$ or $13$ depending on how many of $x, y$ are contained in $P$.
Consider the fiber diagram
\begin{center}
\begin{tikzcd}
\Mb_{1,n+2} \arrow{d}[swap]{f} \arrow{r}{\xi} & \Mb_{2,n} \arrow{d}{f'} \\
\Mb_{1,P'} \arrow{r}[swap]{\xi'} & \Mb_{2, P' \smallsetminus \{x,y\}}.
\end{tikzcd}
\end{center}
where the vertical maps forget markings and the horizontal maps glue $x$ and $y$. Note that $f^*$ sends the class $\omega_P \in H^{11,0}(\Mb_{1,P'})$ to $\omega_P \in H^{11,0}(\Mb_{1,n+2})$. Now, we see that
\[\xi_*(\omega_P) = \xi_*f^*(\omega_P) = f'^*\xi'_*(\omega_P).\]
 Since $|P' \smallsetminus \{x,y\}| \leq 11$, the class $\xi'_*(\omega_P) \in H^{12,1}(\Mb_{2, P' \smallsetminus \{x, y\}})$ lies in $\mathrm{span}\{Z_{B \subset A}\}$. Applying Lemma \ref{forgetq} repeatedly, it follows that $f'^*\xi'_*(\omega_P)$ lies in $\mathrm{span}\{Z_{B \subset A}\} \subset H^{12,1}(\Mb_{2,n})$.
\end{proof}

Combining Lemmas \ref{g2indep} and \ref{g2span} completes the proof of Theorem \ref{thm:H13} when $g = 2$.

\subsection{Inductive argument for \texorpdfstring{$g \geq 3$}{g>=3}}
\label{ind}
The basic idea, inspired by the inductive arguments in \cite[\S4]{ArbarelloCornalba}, is to show that the image of $\xi^*H^{12,1}(\Mb_{g,n})$ inside $H^{12,1}(\Mb_{g-1,n+2})$ coincides with the subspace spanned by $\xi^*\{Z_{B \subset A}\}$. We do so by observing that elements in $\xi^*H^{12,1}(\Mb_{g,n}) \subset H^{12,1}(\Mb_{g-1,n+2})$ satisfy strong symmetry conditions upon pulling back further to $H^{12,1}(\Mb_{g-2,n+4})$. 
To finish the argument, we then show that $\xi^*$ is injective.

Assume $g \geq 3$ and that we have proven Theorem \ref{thm:H13} for all $(g', n')$ with $g' < g$ or $g = g'$ and $n' < n$.  
Consider the commutative diagram
\begin{equation} \label{dpull}
\begin{tikzcd}
\Mb_{g-2, n \cup \{x_1, y_1, x_2, y_2\}}  \arrow{r}{\varphi_2} \arrow{d}[swap]{\varphi_1}  &\Mb_{g-1, n \cup \{x_1, y_1\}} \arrow{d}{\xi_1} \\
\Mb_{g-1, n \cup \{x_2, y_2\}} \arrow{r}[swap]{\xi_2} & \Mb_{g,n}
\end{tikzcd}
\end{equation}
where vertical maps glue $x_1, y_1$ and horizontal maps glue $x_2, y_2$. The reader familiar with \cite{ArbarelloCornalba} might like to think of \eqref{dpull} as  replacing the role of \cite[Equation 4.2]{ArbarelloCornalba}, which also studied a codimension $2$ boundary stratum. The diagram \eqref{dpull} captures stronger symmetry conditions and actually simplifies the argument.

By the induction hypothesis, we know that the $Z_{B \subset A}$ form a basis for $H^{12,1}(\Mb_{g-1,n+2})$. Thus, given $z \in H^{12,1}(\Mb_{g,n})$, we can write
\begin{align} \label{heyo}
\xi_i^*(z) &= \sum_{B \subset A} c_{B \subset A} Z_{B \subset A} \in H^{13}(\Mb_{g-1,n+2}),
\end{align}
for some constants $c_{B \subset A}$.
We know that $\xi_i^*(z)$ must be 
symmetric under exchanging $x_i$ and $y_i$. The classes $Z_{B \subset A}$ with $x_i, y_i \in B$ are anti-symmetric in swapping $x_i, y_i$. Because the $Z_{B \subset A}$ are a basis for $H^{13}(\Mb_{g-1,n+2})$, it follows that
$c_{B \subset A}$ must be zero when $x_i, y_i \in B$. Next, let
\begin{equation} \label{vdef} v := \sum_{\substack{\{x_i, y_i\} \subset A^c}} c_{B \subset A} Z_{B \subset A} \in H^{13}(\Mb_{g,n}).
\end{equation}
Our goal is to show that $z = v$. 
Set $\alpha = z - v$.
The first step is the following.
\begin{lem} \label{firststep}
We have $\xi_i^*(\alpha) = 0$.
\end{lem}
\begin{proof}
Using Lemma \ref{xilem} to compute $\xi_i^*(v)$ and subtracting it from \eqref{heyo}, we have
\begin{equation} \label{apull} \xi_i^*(\alpha) = \sum_{\substack{\{x_i, y_i\} \not\subset A^c \\  \{x_i, y_i\} \not\subset B}} c_{B \subset A} Z_{B \subset A} \in H^{13}(\Mb_{g-1,n+2}).
\end{equation}
Now consider the equation $\varphi_2^*\xi_1^*(z) = \varphi_1^*\xi_2^*(z)$:
\begin{equation} \label{comp} \sum_{\substack{\{x_2,y_2\} \subset A^c \\ \{x_1, y_1\} \not\subset A^c \\  \{x_1, y_1\} \not\subset B}} c_{B \subset A} Z_{B \subset A}
= 
\sum_{\substack{\{x_1,y_1\} \subset A^c \\ \{x_2, y_2\} \not\subset A^c \\  \{x_2, y_2\} \not\subset B}} c_{B \subset A} Z_{B \subset A}
\in H^{13}(\Mb_{g-2,n+4}). \end{equation}
The $B \subset A$ appearing on the left and right are all distinct from each other.
If $g - 2 \geq 2$, then we know they are all independent by the induction hypothesis. It follows that the coefficients $c_{B \subset A}$ appearing in \eqref{comp} all vanish, and hence $\xi_i^*(\alpha) = 0$ in \eqref{apull}.

If $g - 2 = 1$, we must look a little more closely at which $Z_{B \subset A}$ are actually appearing in \eqref{comp} to know they are independent. Looking at \eqref{comp}, it will suffice to know that the collection of $Z_{B \subset A}$ such that one of the following holds is independent in $H^{13}(\Mb_{1,n+4})$:
\begin{itemize}
    \item Type 1: $|A^c| \geq 3$
    \item Type 2: $A^c = \{x_1, y_1\}$ and $\{x_2, y_2\} \not\subset B$
    \item Type 3: $A^c = \{x_2, y_2\}$ and $\{x_1, y_1\} \not\subset B$.
\end{itemize}
Indeed, by \eqref{fE}, modulo classes of type 1, classes of type 2 are pulled back along $f_{B \cup \{x_1, y_1\}}$, while classes of type 3 are pulled back along $f_{B \cup \{x_2,y_2\}}$. Since we have $\{x_2, y_2\} \not\subset B \cup \{x_1,y_1\}$ 
in the second type, we have $f_{B \cup \{x_1, y_1\}} \neq f_{B \cup \{x_2, y_2\}}$. In other words,
classes of the second and third type come from pullbacks under different forgetful maps. Thus, by Lemma \ref{indep} they are independent.

Having established the independence of the $Z_{B \subset A}$ appearing in \eqref{comp}, it follows that all $c_{B \subset A}$ appearing in \eqref{comp} vanish and consequently $\xi^*(\alpha) = 0$.
\end{proof}

Our next task is to prove that $\xi^*\colon H^{12,1}(\Mb_{g,n}) \to H^{12,1}(\Mb_{g-1,n+2})$ is injective. When $g$ is sufficiently large relative to the cohomological degree, \cite[Theorem 2.10]{ArbarelloCornalba} says that $\xi^*$ is injective; however, we need injectivity for all $g$ in degree $13$. For this, we combine injectivity of restriction to the full boundary with a study of pullbacks to codimension $2$ strata. Let $\tilde{\partial \M_{g,n}}$ be the normalization of the boundary of $\M_{g,n}\subset \Mb_{g,n}$. It is a disjoint union of moduli spaces $\Mb_{\Gamma}$ corresponding to stable graphs $\Gamma$ of genus $g$ with $n$ legs and exactly one edge.

\begin{lem} \label{bdryinj}
The pullback map $H^{13}(\Mb_{g,n}) \to H^{13}(\widetilde{\partial \M_{g,n}})$ is injective.
\end{lem}
\begin{proof}
We may suppose $g\geq 3$. From taking the associated graded in degree 13 of the long exact sequence in compactly supported cohomology, we have a left exact sequence
\[
    0\rightarrow \gr_{13}^W H^{13}_c(\M_{g,n})\rightarrow H^{13}(\Mb_{g,n})\rightarrow H^{13}(\tilde{\partial \M_{g,n}}).
\]
We claim that $\gr^W_{13} H^{13}_c(\M_{g,n})=0$. If $13<2g-2+n$ or $g = 7$ and $n = 0,1$, this follows from \cite[Proposition 2.1]{BergstromFaberPayne}. Otherwise, we have $13\geq 2g-2+n$, and $3 \leq g \leq 6$. In these cases, \cite[Theorem 1.4]{CL-CKgP} combined with \cite[Lemma 4.3]{CLP-STE} shows that $\gr_{13}^W H^{13}_c(\M_{g,n})=0$. In all cases, this vanishing implies that the pullback to the boundary is injective.
\end{proof}

\begin{lem} \label{xiin}
For $g \geq 3$, the pullback map $\xi^*\colon H^{12,1}(\Mb_{g,n}) \to H^{12,1}(\Mb_{g-1,n+2})$ is injective.
\end{lem}

\begin{proof}
Suppose $\alpha \in H^{12,1}(\Mb_{g,n})$ is any element with $\xi^*(\alpha) = 0$.
Let 
\[\epsilon\colon \Mb_{\Gamma} = \Mb_{a, A \cup p} \times \Mb_{g - a, A^c \cup q} \to \Mb_{g,n}
\]
be any one-edge graph with a disconnecting node. Without loss of generality, we assume $a \leq g -a$.
By Lemma \ref{bdryinj}, to show that $\alpha  = 0$, it suffices to show that that $\epsilon^*(\alpha) = 0$ for all one-edged $\Gamma$.

Let $\Gamma_1$ and $\Gamma_2$ be the two graphs obtained from $\Gamma$ by inserting a loop into either of the two vertices of $\Gamma$, so
\[\Mb_{\Gamma_1} = \Mb_{a, A \cup p} \times \Mb_{g-a-1,A^c \cup \{q, x, y\}} \qquad \text{and} \qquad \Mb_{\Gamma_2} = \Mb_{a - 1, A \cup \{p, x, y\}} \times \Mb_{A^c \cup q}.\]
(If $a = 0$, then we omit $\Gamma_2$ since it is not defined.)
For each $i$, there is a commutative diagram
\begin{center}
\begin{tikzcd}
\Mb_{\Gamma_i} \arrow{d}[swap]{\xi_i'} \arrow{r}{\epsilon_i'} & \Mb_{g-1,n+2} \arrow{d}{\xi} \\
\Mb_{\Gamma} \arrow{r}[swap]{\epsilon} & \Mb_{g,n},
\end{tikzcd}
\end{center}
where $\epsilon_i'$ glues $p$ and $q$ and $\xi_i'$ glues $x$ and $y$. In particular, we have
\[\xi_i'^*\epsilon^*(\alpha) = \epsilon_i'^*\xi^*(\alpha) = 0. \]
To show that $\epsilon^*(\alpha) = 0$, it will now suffice to show that 
\[\xi_1'^* \oplus \xi_2'^*\colon H^{12,1}(\Mb_{\Gamma}) \to H^{12,1}(\Mb_{\Gamma_1}) \oplus H^{12,1}(\Mb_{\Gamma_2})\]
is injective. (Or, when $a = 0$, we wish to show $\xi_1'^*$ is injective, since $\Gamma_2$ is not defined.)

First suppose that $a, g - a\geq 2$. Then, by vanishing results on on odd cohomology in degree $\leq 11$ \cite{ArbarelloCornalba,BergstromFaberPayne,CanningLarsonPayne}, we have
\begin{equation} \label{init} H^{12,1}(\Mb_{\Gamma}) = H^{12,1}(\Mb_{a, A \cup p}) \otimes H^{0}(\Mb_{g-a,A^c \cup q}) \oplus H^{0}(\Mb_{a, A \cup p}) \otimes H^{12,1}(\Mb_{g-a,A^c \cup q}), \end{equation}
which clearly injects into the K\"unneth components
\[H^{12,1}(\Mb_{a, A \cup p}) \otimes H^0(\Mb_{g - a - 1, A^c \cup \{q, x, y\}}) \oplus H^0(\Mb_{a-1, A \cup \{p, x, y\}}) \otimes H^{12,1}(\Mb_{g-a,A^c \cup q}) \]
inside $H^{12,1}(\Mb_{\Gamma_1}) \oplus H^{12,1}(\Mb_{\Gamma_2})$.

If $a = 1$, then $H^{12,1}(\Mb_{\Gamma})$ has the K\"unneth components listed in \eqref{init} together with an $H^{11,0} \otimes H^{1,1}$ K\"unneth term, which we claim injects into a K\"unneth term of $\Mb_{\Gamma_1}$:
\begin{equation} \label{11t2}
H^{11,0}(\Mb_{1,A \cup p}) \otimes H^{1,1}(\Mb_{g-1,A^c \cup q}) \hookrightarrow H^{11,0}(\Mb_{1,A \cup p}) \otimes H^{1,1}(\Mb_{g-2,A^c \cup \{q, x, y\}}). %
\end{equation}
To see that \eqref{11t2} is injective, we need to know that $H^2(\Mb_{g-1,A^c \cup q}) \to H^2(\Mb_{g-2,A^c \cup \{q, x, y\}})$ is injective. Since $g \geq 3$, this injection on $H^2$ holds by \cite[Theorem 2.10]{ArbarelloCornalba}.

Finally, if $a = 0$, then $H^{12,1}(\Mb_{0,A \cup p}) = 0$, so
\[H^{12,1}(\Mb_{\Gamma}) = H^0(\Mb_{0,A \cup p}) \otimes H^{12,1}(\Mb_{g, A^c \cup q})\]
and
\[H^{12,1}(\Mb_{\Gamma_1}) = H^0(\Mb_{0, A \cup p}) \otimes H^{12,1}(\Mb_{g,A^c \cup \{q, x, y\}}).\]
Inducting on $n$, we may assume that Theorem \ref{thm:H13} is known for $n' < n$, so
$H^{12,1}(\Mb_{g, A^c \cup q})$ has a basis given by $Z_{B \subset A}$. By Lemma \ref{xilem}, the map $H^{12,1}(\Mb_{g,A^c \cup q}) \to H^{12,1}(\Mb_{g-1,A^c \cup \{q, x, y\}})$ is injective. Hence, $\xi_1'^*$ is injective.
\end{proof}

\begin{proof}[Proof of Theorem \ref{thm:H13} for $g \geq 3$]
By \cite[Theorem 1.1]{CLP-STE}, we have
\[
H^{13}(\Mb_{g,n})= H^{12,1}(\Mb_{g,n})\oplus H^{1,12}(\Mb_{g,n}).
\]
We have shown in Lemma \ref{firststep} that for any $z \in H^{12,1}(\Mb_{g,n})$, there exists an element $v \in \mathrm{span}\{Z_{B \subset A}\}$ (defined by \eqref{vdef}) such that $\xi^*(z) = \xi^*(v)$. We know that $\xi^*$ is injective by Lemma \ref{xiin}, so $z = v \in \mathrm{span}\{Z_{B \subset A}\}$. Hence, the $Z_{B \subset A}$ span $H^{12,1}(\Mb_{g,n})$. Moreover, by Lemma \ref{g2indep}, $\{Z_{B \subset A}\}$ is independent for all $g \geq 2$.
For the identification of the $\mathbb{S}_n$-action on our basis, see the last paragraph of Section \ref{candidate}.
\end{proof}

\section{Computing the weight thirteen Euler characteristic of \texorpdfstring{$\cM_{g,n}$}{Mgn}}
\label{sec:GK13}
Our next goal is a computation of the weight 13 Euler characteristic  of $\M_{g,n}$ for low $g$, $n$. In particular, in two cases where $3g + 2n \geq 25$ but $\chi_{11}(\cM_{g,n}) = 0$, we show that 
\[
\chi_{13}(\cM_{12}) =-6 \qquad \text{ and } \qquad \chi_{13}(\cM_{8,1}) =-2.
\]
In other words, here we prove Corollary~\ref{cor:chi13}.

\subsection{Recollections on characters and symmetric functions}
We begin by recalling well-known facts about characters of symmetric group representations, see for example \cite[Section 7]{GetzlerKapranov} for details.
Let $V$ be a finite dimensional representation of the symmetric group $\ss_n$. Then we may associate to $V$ the character 
\[
\ch_n(V) = \frac1{n!} \sum_{\sigma\in \ss_n} \tr_V(\sigma) 
 \prod_{k\geq 1} p_k^{i_k(\sigma)}
\in \Lambda := \Q\llbracket p_1,p_2,\dots \rrbracket
\]
where $i_k(\sigma)$ is the number of $k$-cycles in the cycle decomposition of $\sigma$ and the $p_i$ are formal variables that represent the power sums in the ring of symmetric functions $\Lambda$.

A symmetric sequence $A$ is a collection $\{A(n)\}_{n\geq 0}$ of graded vector spaces such that $A(n)$ carries a representation of $\ss_n$.
If all $A(n)$ are finite dimensional, we define the equivariant Euler characteristic 
\[
\chi^\ss(A) := \sum_{n\geq 0} \sum_i (-1)^i \ch_n(A(n)^i) \in \Lambda.
\]
The category of symmetric sequences has a symmetric monoidal product $\boxtimes$ defined such that 
\begin{equation*}
(A\boxtimes B)(n) = \bigoplus_{p+q=n} \mathrm{Ind}_{\ss_p\times \ss_q}^{\ss_n}
A(p)\otimes A(q).
\end{equation*}
The associated equivariant Euler characteristics satisfy 
\begin{equation}\label{equ:chi box}
\chi^\ss (A\boxtimes B) = \chi^\ss(A) \chi^\ss(B).
\end{equation}
Furthermore, one has a non-symmetric operation $\bullet$ defined on symmetric sequences $A$, $B$, such that
\[
(A\bullet B)(n) = \bigoplus_{p} A(p) \otimes_{\ss_p} \mathrm{Res}^{\ss_{n+p}}_{\ss_n\times \ss_p}B(n+p).  
\]
In the finite dimensional case one has
\begin{equation}\label{equ:chi bullet}
\chi^{\ss}(A\bullet B) = \chi^{\ss}(A)\Big(\frac{\partial}{\partial p_1}, 2 \frac{\partial}{\partial p_2 }, \dots, j \frac{\partial}{\partial p_j },\dots\Big) \chi^\ss(B). 
\end{equation}
Here, the notation means that in the series $\chi^{\ss}(A)\in \Q\llbracket p_1,p_2,\dots\rrbracket$ one formally replaces each $p_j$ by $j \frac{\partial}{\partial p_j } $ and then applies the resulting differential operator to $\chi^\ss(B)$.

In practice, we will often consider symmetric sequences with an additional grading. In this case, we consider sequences of $\ss_n$-modules $\{A(g,n)\}_{g,n}$.
We then remember the grading by introducing a new formal variable and define 
\[
\chi_u^{\ss}(A) := \sum_{g} u^g \chi^{\ss}(A(g,-)) \in \Lambda\llbracket u\rrbracket.
\]
The formulas \eqref{equ:chi box} and \eqref{equ:chi bullet} extend to this bigraded setting by replacing $\chi^\ss$ with $\chi_u^{\ss}$.

\subsection{The weight 13 Euler characteristic of \texorpdfstring{$\M_{g,n}$}{Mgn}}
We need to use the $\ss_n$-characters of $H^k(\MM_{g,n})$ for $k \in \{2, 11, 13\}$. 
Let $s_\lambda \in \Lambda$ denote the Schur function for the partition $\lambda$ expressed in terms of power sums $p_j$. We define $$\mathbb{X}_k := (-1)^k \sum_{g,n} \hbar^{g}
\ch_n H^k(\MM_{g,n}).$$ These are power series in the formal variable $\hbar$ with coefficients in $\Lambda$.  All sums are over nonnegative integers $g$, $m$, and $n$, with  $2g + n \geq 3$, except where otherwise specified.  

The following formula for $\mathbb{X}_2$ is a consequence of \cite[Theorem 2.2]{ArbarelloCornalba}:
\[
\mathbb{X}_2 = \sum_{g\geq 3, n} \hbar^{g} s_{n} + \sum_{g\geq 2, \, n\geq 1} \hbar^{g} p_1 s_{n-1} + \sum_{g\geq 1,n } \hbar^{g} s_{n} + s_2\circ \Big( \sum_{2g+n\geq 2} \hbar^{g} s_n \Big) + \sum_{n\geq 4} (p_1 s_{n-1} -s_2 s_{n-2}).
\]
The first four terms correspond to contributions from $\kappa$ for $g \geq 3$, $\psi$-classes for $g \geq 2$, the boundary divisor $\delta_{irr}$, and the other boundary divisors. In the fourth term, $\circ$ denotes plethysm. The fifth and final term accounts for the relations among boundary classes for $g = 0$, by adding the contribution of $\psi$-classes and subtracting the relations among these classes and boundary classes, as in \cite[Section~3.6]{PayneWillwacher24b}.

The analogous statement for $\mathbb{X}_{11}$ follows from  \cite[Theorem 1.1]{CanningLarsonPayne}:
\[
\mathbb{X}_{11} = -2
\sum_{n\geq 11} \hbar  s_{(n-10) 1^{10}}. 
\]
For weight $13$, using Theorem~\ref{thm:H13} and Corollary~\ref{gid}, we have
\[
\mathbb{X}_{13} = -2 \Big( \sum_{g \geq 2, m, n} \hbar^{g} s_{1^{10}} s_m s_n - \sum_{m \geq 0, n \geq 3} \hbar s_{1^{10}} s_m s_n - \sum_m \hbar s_{2 1^{10}} s_m \Big).
\]
The first term expresses the contribution from the generators described in Theorem~\ref{thm:H13} and the second and third terms give the relations for $g =1$, as described in Corollary~\ref{gid}; we use the fact that $\Res^{\ss_{k+1}}_{\ss_k}K^{11}_{k+1} = V_{k-10,1^{10}} \oplus V_{k-9,1^{10}}$, which corresponds to the product of Schur functions $ s_{1^{10}} s_{k-10}$ by the Pieri rule.

\medskip 

We furthermore define the differential operators 
\begin{align*}
    D_{k} &:= \mathbb{X}_k\Big(u, \frac{\partial}{\partial p_1}, \dots, j \frac{\partial}{\partial p_j },\dots\Big)
\end{align*}
by replacing the variable $\hbar$ by $u$ and $p_j$ by the derivative operator $j\frac{\partial}{\partial p_j}$ for all $j$.

Next we recall some special functions from \cite{PayneWillwacher24b}.
Let 
\begin{equation*}
    B(z) := \sum_{r\geq 2}\frac{B_r}{r(r-1)} \frac 1 {z^{r-1}},
  \end{equation*}
for $B_r$ the $r$th Bernoulli number,  and 
  \begin{align*}
    E_\ell&:= \frac 1 \ell \sum_{d\mid \ell}\mu(\ell/d)\frac 1 {u^d},
    &
    \lambda_\ell &:= u^\ell (1-u^\ell)\ell
\end{align*}
with $\mu$ the M\"obius function. 
Then we define 
$U_\ell(X, u)=\exp(\log(U_\ell(X,u)))$
 by
\begin{equation} \label{equ:Uelldef} 
\resizebox{.92\hsize}{!}{
$\begin{aligned}
    \log U_\ell(X,u) 
   &=   
    \log \frac {(-\lambda_\ell)^X \Gamma(-E_\ell+X) }{\Gamma(-E_\ell)}
  \\&=
        X\left(\log(\lambda_\ell E_\ell)-1 \right)+(-E_\ell+X-\textstyle{\frac 1 2} )\log(1-\textstyle{\frac X{E_\ell}}) +
         B(-E_\ell+X)- B(-E_\ell). 
         \end{aligned}
          $}
\end{equation}
Finally, we let
\[
Y:= \prod_{\ell\geq 1} U_\ell\big(\sum_{d\mid \ell}\mu(\ell/d) p_d,u\big)
\in \Lambda\llbracket u \rrbracket.
\]
Then we can state:
\begin{prop}
The generating function for the $\ss_n$-equivariant weight 13 Euler characteristic of $\M_{g,n}$ satisfies 
\begin{equation}\label{equ:chi 13}
\resizebox{.92\hsize}{!}{$
\displaystyle{\sum_{g,n}}\displaystyle{\sum_i} (-1)^i
u^{g+n} \ch_n \gr_{13}^W H^i_c(\M_{g,n}) 
=
I\Big(
\frac 1Y D_{13} Y + 
\frac 1{uY} D_{11} D_2 Y
- 
\frac 1{u Y^2} (D_{11} Y) (D_2 Y)
\Big).$}
\end{equation}
Here $I(f)=f(u,-p_1,-p_2,\dots)$ for $f\in \Lambda\llbracket u\rrbracket$.
\end{prop}
\begin{proof}
We have to compute 
\[
\chi^{\ss}_{13}(\M_{g,n})
:=
\sum_i (-1)^i
\ch_n \gr_{13}^W H^i_c(\M_{g,n})=
\chi^{\ss}(\GK_{g,n}^{13}),
\]
using the Getzler-Kapranov complex of Section \ref{sec:GK def}.
The right-hand side $\GK_{g,n}^{13}$ is a complex of connected graphs with either one special vertex decorated by $H^{13}(\MM_{g',n'})$ (for any $g',n'$) or with two special vertices decorated by $H^{11}(\MM_{g',n'})$ and $H^{2}(\MM_{g'',n''})$ 
respectively.
Accordingly, we split 
\begin{equation}\label{equ:Feyn V1 V2}
\GK_{g,n}^{13} \cong V_1 \oplus V_2,
\end{equation}
where $V_1$ is spanned by graphs with one special vertex decorated by $H^{13}(\MM_{g,n})$ and $V_2$ is spanned by graphs with two special vertices.
Generally, graph complexes with one special vertex such as $V_1$ have been discussed in \cite[Section 3]{PayneWillwacher24b}.
\newcommand{\fG}{\widetilde{fG}}
As in loc.~cit.~one has a collection of graph complexes $\fG=\{\fG(g,n)\}_{g,n}$ such that $\fG(g,n)$ is generated by (possibly non-connected) graphs with $n$ external legs, $v$ vertices and $e$ edges such that $g=e-v$.
Given any symmetric sequence $A$, the graph complex with one special vertex decorated by $A$ and no external legs can be written as 
\[
A\otimes_{\ss} \fG := \oplus_n A(n) \otimes_{\ss_n} \fG(-,n).
\]
\[
\begin{tikzpicture}
\node[int] (w1) at (-.8, .7) {};
\node[int] (w2) at (-.8, 1.7) {};
\node[int] (w3) at (.8, .7) {};
\node[int] (w4) at (.8, 1.7) {};
\node[draw, ellipse, inner sep=1em] (b) at (0,-.7) {$a$};
\draw (w1) edge (w2) edge (w3) edge (w4)
(w2) edge (w3) edge (w4) 
(w3) edge (w4)
(b.north west) edge (w1) 
(b) edge (w2) (b.north east) edge (w3);
\draw [decorate,
    decoration = {brace}] (1.5,2) --  (1.5,0.3)
    node[pos=0.5,left=-35pt,black]{$\in\fG$};
\draw [dashed] (-1.2,0.2)--(1.2,0.2);
    \node at (-1.5,0.2) {$\otimes_\bbS$};
\draw [decorate,
    decoration = {brace}] (1.5,0.1) --  (1.5,-1.2)
    node[pos=0.5,left=-30pt,black]{$\in A$};
\end{tikzpicture}
\]
More generally, if we want to consider graphs with external legs then the corresponding graph vector space may be written as
\[
A\bullet \fG
\]
using the notation $\bullet$ of the previous subsection.
Note that the graphs generating $A\bullet \fG$ might be disconnected, but we have the relation 
\[
A\bullet \fG = (A\bullet \fG)_{conn} \boxtimes \fG
\]
between $A\bullet \fG$ and the subspace $(A\bullet \fG)_{conn}\subset A\bullet \fG$ spanned by the connected graphs.
Using formulas \eqref{equ:chi box} and \eqref{equ:chi bullet} for the Euler characteristics, we may hence deduce that
\[
\chi_u^{\ss}(V_1) = I\left( (D_{13} \chi^\ss_u(\fG)) / \chi^\ss_u(\fG) \right),
\]
with the operation $I$ accounting for the different sign conventions in the definitions of $\fG$ and the Feynman transform respectively; the external legs in generators of $\fG$ have degree $+1$, whereas the external legs in the generators of $\Feyn(-)$ have degree zero.

The second summand $V_2$ of \eqref{equ:Feyn V1 V2} may be handled similarly.
Here we have two vertices decorated by $H^2(\MM)$ and $H^{11}(\MM)$ respectively, but we may equivalently consider both together as one vertex decorated by $H^2(\MM)\boxtimes H^{11}(\MM)$.
One just needs to be careful with connectivity, because fusing the two external vertices can make a disconnected graph connected.
Hence we may see that
\[
\chi^\ss_u(V_2) =
\frac1u
I\left( 
 (D_{2}D_{11}\chi_u^\ss(\fG))/\chi_u^\ss(\fG)
-
 (D_{2}\chi_u^\ss(\fG))(D_{11}\chi_u^\ss(\fG))/\chi_u^\ss(\fG)^2
\right),
\]
where the final term subtracts the Euler characteristic of the subspace spanned by graphs with two connected components, each containing one of the special vertices. The factor $\frac 1u$ corrects for our (mis-)treatment of the two special vertices as one vertex.

Finally, we have from \cite[Corollary 4.3]{PayneWillwacher24b} that 
\[
\chi_u^\ss(\fG) = Y,
\]
so that we obtain the formula of the proposition.
\end{proof}

We have implemented the above formula on the computer. The resulting Euler characteristics are displayed in Table \ref{fig:ec13} for small $g$, $n$.
The computer program can be found at \url{https://github.com/wilthoma/polypointcount}. 
\begin{figure}
\noindent\scalebox{.44}{
\begin{tabular}{|g|M|M|M|M|M|M|M|M|} 
\hline \rowcolor{Gray} $g,n$ & 0 & 1 & 2 & 3 & 4 & 5 & 6 & 7 \\
\hline
0 & $ 0 $ & $ 0 $ & $ 0 $ & $ 0 $ & $ 0 $ & $ 0 $ & $ 0 $ & $ 0 $ \\  
\hline
 1 & $ 0 $ & $ 0 $ & $ 0 $ & $ 0 $ & $ 0 $ & $ 0 $ & $ 0 $ & $ 0 $ \\   
\hline
 2 & $ 0 $ & $ 0 $ & $ 0 $ & $ 0 $ & $ 0 $ & $ 0 $ & $  0 $ & $  0 $ \\ 
\hline
 3 & $ 0 $ & $ 0 $ & $ 0 $ & $ 0 $ & $ 0 $ & $ 0 $ & $ 0 $ & $ 0 $ \\ 
\hline
 4 & $ 0 $ & $ 0 $ & $ 0 $ & $ 0 $ & $ 0 $ & $ 0 $ & $ 0 $ & $ -s_{1,1,1,1,1,1,1} $ \\ 
\hline
 5 & $ 0 $ & $ 0 $ & $ 0 $ & $ 0 $ & $ 0 $ & $ 0 $ & $ -s_{1,1,1,1,1,1} - 2 s_{2,1,1,1,1} $ & $ -2 s_{1,1,1,1,1,1,1} - 4 s_{2,1,1,1,1,1} - s_{2,2,1,1,1} + 3 s_{3,1,1,1,1} + 3 s_{3,2,1,1} + 4 s_{4,1,1,1} $ \\ 
\hline
 6 & $ 0 $ & $ 0 $ & $ 0 $ & $ 0 $ & $ s_{1,1,1,1} $ & $ 2 s_{1,1,1,1,1} - s_{2,1,1,1} - s_{2,2,1} - 3 s_{3,1,1} $ & $ -s_{1,1,1,1,1,1} - 4 s_{2,1,1,1,1} - 2 s_{2,2,1,1} - s_{2,2,2} - 6 s_{3,1,1,1} + 2 s_{3,2,1} + 3 s_{3,3} + 4 s_{4,1,1} + 4 s_{4,2} + 5 s_{5,1} $ & $ -7 s_{1,1,1,1,1,1,1} - 7 s_{2,1,1,1,1,1} + 4 s_{2,2,1,1,1} - 6 s_{2,2,2,1} + 7 s_{3,1,1,1,1} + 14 s_{3,2,1,1} - 3 s_{3,2,2} + 5 s_{3,3,1} + 19 s_{4,1,1,1} + 7 s_{4,2,1} - 6 s_{4,3} + 11 s_{5,1,1} - 7 s_{5,2} - 7 s_{6,1} - 6 s_{7} $ \\ 
\hline
 7 & $ 0 $ & $ 0 $ & $ 0 $ & $ s_{1,1,1} + 2 s_{2,1} $ & $ 2 s_{1,1,1,1} + 4 s_{2,1,1} + s_{2,2} - 3 s_{3,1} - 4 s_{4} $ & $ 2 s_{1,1,1,1,1} - 3 s_{2,1,1,1} + s_{2,2,1} - 13 s_{3,1,1} - 4 s_{3,2} - 11 s_{4,1} $ & $ 10 s_{1,1,1,1,1,1} - s_{2,1,1,1,1} + 16 s_{2,2,1,1} + 13 s_{2,2,2} - 11 s_{3,1,1,1} + 30 s_{3,2,1} + 24 s_{3,3} + 10 s_{4,1,1} + 30 s_{4,2} + 27 s_{5,1} + 16 s_{6} $ & $ -s_{1,1,1,1,1,1,1} - 82 s_{2,1,1,1,1,1} - 96 s_{2,2,1,1,1} - 81 s_{2,2,2,1} - 140 s_{3,1,1,1,1} - 155 s_{3,2,1,1} - 117 s_{3,2,2} - 44 s_{3,3,1} - 62 s_{4,1,1,1} - 121 s_{4,2,1} - 41 s_{4,3} - 7 s_{5,1,1} - 65 s_{5,2} - 18 s_{6,1} - 11 s_{7} $ \\ 
\hline
 8 & $ 0 $ & $ -s_{1} $ & $ -2 s_{1,1} $ & $ s_{1,1,1} + 3 s_{2,1} + 5 s_{3} $ & $ 8 s_{1,1,1,1} + 5 s_{2,1,1} - 7 s_{2,2} - 6 s_{3,1} - 4 s_{4} $ & $ 35 s_{1,1,1,1,1} + 59 s_{2,1,1,1} + 34 s_{2,2,1} + 15 s_{3,1,1} - 10 s_{4,1} - 3 s_{5} $ & $ 46 s_{1,1,1,1,1,1} - 87 s_{2,2,1,1} - 32 s_{2,2,2} - 233 s_{3,1,1,1} - 246 s_{3,2,1} - 49 s_{3,3} - 235 s_{4,1,1} - 96 s_{4,2} - 27 s_{5,1} + 21 s_{6} $ & $ $ \\
\hline
 9 & $ s_{} $ & $ 2 s_{1} $ & $ s_{1,1} + 3 s_{2} $ & $ -9 s_{1,1,1} + 6 s_{2,1} + 8 s_{3} $ & $ -4 s_{1,1,1,1} + 56 s_{2,1,1} + 34 s_{2,2} + 69 s_{3,1} + 7 s_{4} $ & $ 13 s_{1,1,1,1,1} + 79 s_{2,1,1,1} - 12 s_{2,2,1} - 4 s_{3,1,1} - 127 s_{3,2} - 179 s_{4,1} - 110 s_{5} $ & $ $ & $ $ \\ 
\hline
 10 & $ -2 s_{} $ & $ -7 s_{1} $ & $ -20 s_{1,1} - 19 s_{2} $ & $ -38 s_{1,1,1} - 41 s_{2,1} + 14 s_{3} $ & $ -56 s_{1,1,1,1} + 5 s_{2,1,1} + 13 s_{2,2} + 182 s_{3,1} + 132 s_{4} $ & $ $ & $ $ & $ $ \\  
\hline
 11 & $ 3 s_{} $ & $ 13 s_{1} $ & $ s_{1,1} - 15 s_{2} $ & $ -65 s_{1,1,1} - 68 s_{2,1} - 45 s_{3} $ & $ $ & $ $ & $ $ & $ $ \\  
\hline
 12 & $ -3 s_{} $ & $ 0 $ & $ 40 s_{1,1} - 28 s_{2} $ & $ $ & $ $ & $ $ & $ $ & $ $ \\ 
\hline
 13 & $ 12 s_{} $ & $ 93 s_{1} $ & $ $ & $ $ & $ $ & $ $ & $ $ & $ $ \\ 
\hline
\end{tabular}
}

\medskip

\noindent\scalebox{.41}{
\begin{tabular}{|g|M|M|M|M|M|M|} \hline \rowcolor{Gray} $g,n$ & 8 & 9 & 10 & 11 & 12 & 13\\ 
\hline
0 &  $ 0 $ & $ 0 $ & $ 0 $ & $ 0 $ & $ 0 $ & $ 0 $  \\  
\hline
 1 &  $ 0 $ & $ 0 $ & $ 0 $ & $ 0 $ & $ -s_{2,1,1,1,1,1,1,1,1,1,1} $ & $ s_{3,1,1,1,1,1,1,1,1,1,1} + s_{3,2,1,1,1,1,1,1,1,1} + s_{4,1,1,1,1,1,1,1,1,1} $ \\  
\hline
 2 &  $ 0 $ & $ 0 $ & $ s_{1,1,1,1,1,1,1,1,1,1} $ & $ s_{1,1,1,1,1,1,1,1,1,1,1} - s_{2,1,1,1,1,1,1,1,1,1} - s_{2,2,1,1,1,1,1,1,1} - 2 s_{3,1,1,1,1,1,1,1,1} $ & $ s_{2,2,1,1,1,1,1,1,1,1} + 3 s_{3,2,1,1,1,1,1,1,1} + s_{3,2,2,1,1,1,1,1} + 2 s_{3,3,1,1,1,1,1,1} + 3 s_{4,1,1,1,1,1,1,1,1} + 2 s_{4,2,1,1,1,1,1,1} + 2 s_{5,1,1,1,1,1,1,1} $ & $ $ \\  
\hline
 3 & $ 0 $ & $ s_{1,1,1,1,1,1,1,1,1} + 2 s_{2,1,1,1,1,1,1,1} $ & $ s_{1,1,1,1,1,1,1,1,1,1} + 2 s_{2,1,1,1,1,1,1,1,1} - 3 s_{3,1,1,1,1,1,1,1} - 3 s_{3,2,1,1,1,1,1} - 3 s_{4,1,1,1,1,1,1} $ & $ s_{1,1,1,1,1,1,1,1,1,1,1} + 3 s_{2,2,1,1,1,1,1,1,1} + 3 s_{2,2,2,1,1,1,1,1} - 5 s_{3,1,1,1,1,1,1,1,1} + 3 s_{3,2,1,1,1,1,1,1} + 2 s_{3,2,2,1,1,1,1} + 6 s_{3,3,1,1,1,1,1} + 3 s_{3,3,2,1,1,1} - 2 s_{4,1,1,1,1,1,1,1} + 6 s_{4,2,1,1,1,1,1} + s_{4,2,2,1,1,1} + 4 s_{4,3,1,1,1,1} + 5 s_{5,1,1,1,1,1,1} + 4 s_{5,2,1,1,1,1} + 3 s_{6,1,1,1,1,1} $ & $ $ & $ $ \\  
\hline
 4 & $ -2 s_{1,1,1,1,1,1,1,1} + s_{2,1,1,1,1,1,1} + s_{2,2,1,1,1,1} + 3 s_{3,1,1,1,1,1} $ & $ 3 s_{2,1,1,1,1,1,1,1} + s_{2,2,1,1,1,1,1} + s_{2,2,2,1,1,1} + 4 s_{3,1,1,1,1,1,1} - 3 s_{3,2,1,1,1,1} - s_{3,2,2,1,1} - 3 s_{3,3,1,1,1} - 4 s_{4,1,1,1,1,1} - 4 s_{4,2,1,1,1} - 4 s_{5,1,1,1,1} $ & $ 4 s_{1,1,1,1,1,1,1,1,1,1} + 6 s_{2,1,1,1,1,1,1,1,1} + 7 s_{2,2,2,1,1,1,1} + 3 s_{2,2,2,2,1,1} - 2 s_{3,1,1,1,1,1,1,1} - 6 s_{3,2,1,1,1,1,1} + 6 s_{3,2,2,1,1,1} + 6 s_{3,3,2,1,1} + s_{3,3,2,2} + 3 s_{3,3,3,1} - 11 s_{4,1,1,1,1,1,1} + s_{4,2,1,1,1,1} + 4 s_{4,2,2,1,1} + 11 s_{4,3,1,1,1} + 4 s_{4,3,2,1} + 2 s_{4,4,1,1} - 5 s_{5,1,1,1,1,1} + 10 s_{5,2,1,1,1} + 2 s_{5,2,2,1} + 7 s_{5,3,1,1} + 6 s_{6,1,1,1,1} + 5 s_{6,2,1,1} + 4 s_{7,1,1,1} $ & $ $ & $ $ & $ $ \\  
\hline
 5 & $ -2 s_{1,1,1,1,1,1,1,1} + 2 s_{2,1,1,1,1,1,1} - s_{2,2,1,1,1,1} - 2 s_{2,2,2,1,1} + 12 s_{3,1,1,1,1,1} + 3 s_{3,2,1,1,1} - 6 s_{3,3,1,1} - 3 s_{3,3,2} + 8 s_{4,1,1,1,1} - 5 s_{4,2,1,1} - s_{4,2,2} - 5 s_{4,3,1} - 6 s_{5,1,1,1} - 6 s_{5,2,1} - 5 s_{6,1,1} $ & $ -6 s_{1,1,1,1,1,1,1,1,1} + 2 s_{2,1,1,1,1,1,1,1} - 10 s_{2,2,1,1,1,1,1} - 7 s_{2,2,2,1,1,1} + 3 s_{2,2,2,2,1} + 12 s_{3,1,1,1,1,1,1} - 18 s_{3,2,1,1,1,1} + 4 s_{3,2,2,1,1} + 6 s_{3,2,2,2} - 26 s_{3,3,1,1,1} + s_{3,3,2,1} + 4 s_{3,3,3} - 4 s_{4,1,1,1,1,1} - 17 s_{4,2,1,1,1} + 12 s_{4,2,2,1} - 3 s_{4,3,1,1} + 12 s_{4,3,2} + 5 s_{4,4,1} - 21 s_{5,1,1,1,1} - 2 s_{5,2,1,1} + 7 s_{5,2,2} + 20 s_{5,3,1} + 6 s_{5,4} - 9 s_{6,1,1,1} + 12 s_{6,2,1} + 9 s_{6,3} + 8 s_{7,1,1} + 7 s_{7,2} + 5 s_{8,1} $ & $ $ & $ $ & $ $ & $ $ \\  
\hline
 6 & $ -27 s_{1,1,1,1,1,1,1,1} - 48 s_{2,1,1,1,1,1,1} - 43 s_{2,2,1,1,1,1} - 50 s_{2,2,2,1,1} - 14 s_{2,2,2,2} - 10 s_{3,1,1,1,1,1} - 36 s_{3,2,1,1,1} - 49 s_{3,2,2,1} - 39 s_{3,3,1,1} - 31 s_{3,3,2} + 10 s_{4,1,1,1,1} - 43 s_{4,2,1,1} - 14 s_{4,2,2} - 56 s_{4,3,1} - 8 s_{4,4} - 16 s_{5,1,1,1} - 42 s_{5,2,1} - 24 s_{5,3} - 31 s_{6,1,1} - 15 s_{6,2} - 15 s_{7,1} + s_{8} $ & $ $ & $ $ & $ $ & $ $ & $ $ \\  
\hline
\end{tabular}
}
\caption{\label{fig:ec13} The table shows $\frac 12 \chi^{\ss}_{13}(\M_{g,n})$.}
\end{figure}

\section{Analysis of the generating function for weight eleven Euler characteristics}
\label{sec:wt11euler}
\subsection{Results}
\label{sec:asymp1}
Let 
$Z_g =\frac{1}{2}\chi_{11}(\M_{g})=\frac 12 \sum_i (-1)^i \mathrm{dim}(\gr_{11}^W H_c^i(\M_g))$. This section has two main goals, corresponding to the two statements in Theorem~\ref{thm:wt11 intro}.  First, we describe the asymptotic behavior of $Z_g$. Then, we study and understand the deviation from this asymptotic behavior well enough to bound $Z_g$ away from zero for $g > 600$. This, together with computer calculations for $g \leq 600$, is enough to prove Theorem~\ref{thm:wt11 intro} and thereby complete the proofs Theorems~\ref{thm:polynomial}, \ref{thm:Tatetype}, and \ref{thm:smallestn}.

\subsection{Recollection of formula for the generating function}
Our starting point is the generating function for the weight $11$ compactly supported Euler characteristic, which was obtained in \cite[Theorem 7.1]{PayneWillwacher24}:
\begin{equation}
\label{equ:ec11_withn}
\resizebox{.92\hsize}{!}{$
    \frac12 \displaystyle{\sum_{\substack{g,n\geq 0 \\ 2g+n\geq 3}}} u^{g+n} \chi^{\bbS_n}(\gr_{11}^W H^*_c(\M_{g,n}))
      =
      -u\  T_{\leq 10}\bigg(
      \prod_{\ell\geq 1} 
      \frac { 
        U_\ell(\frac 1 \ell \sum_{d\mid \ell} \mu(\ell/d) (-p_d  +1-w^d ), u )
      }
      { 
        U_\ell(\frac 1 \ell \sum_{d\mid \ell} \mu(\ell/d) (-p_d), u )
      }-1\bigg)\, .
      $}
\end{equation}
Here, the truncation operator $T_{\leq \Gamma}(\sum a_i w^i) = \sum_{i=0}^\Gamma a_i$ takes the sum of the first $\Gamma+1$ coefficients of power series in $w$.
When $T_{\leq \Gamma}$ is applied to a power series in $u$ whose coefficients are power series in $w$, it means that $T_{\leq \Gamma}$ is applied to each $u$-coefficient. The rest of the notation is as in Section \ref{sec:GK13}. In particular,
the power series $U_\ell(X, u)$ was defined in \eqref{equ:Uelldef}.
For later use let us introduce the following notation 
\begin{align}
    \log U_\ell(X,u)
&=:V_\ell(X,u) \notag  
 \\&=
        \underbrace{X\left(\log(\lambda_\ell E_\ell)-1 \right)+(-E_\ell+X-\textstyle{\frac 1 2} )\log(1-\textstyle{\frac X{E_\ell}})}_{:=\bA_\ell(X,u)} +
         \underbrace{B(-E_\ell+X)- B(-E_\ell)}_{:=\bB_\ell(X,u)}. \label{ABdefs}
\end{align}

To study the $n=0$ part of the generating function, we set $p_d=0$. Noting that $U_\ell(0, u)=1$, we obtain
\begin{equation} \label{Zdef}
    Z:= \frac12 \sum_{g\geq 2} u^{g} \chi_{11}(\M_{g})
      =
      -u\  T_{\leq 10}\bigg(
      \prod_{\ell\geq 1} 
        U_\ell(\frac 1 \ell \sum_{d\mid \ell} \mu(\ell/d) (1-w^d ), u )
-1\bigg)\, .
\end{equation}
To unpack this formula further, let
\[
W_{\ell} = 
\frac 1 \ell \sum_{d\mid \ell} \mu(\ell/d) (1-w^d )
=
\begin{cases}
1-w & \text{$\ell=1$} \\
-\frac 1 \ell \sum_{d\mid \ell} \mu(\ell/d) w^d & \text{$\ell\geq 2$}
\end{cases}
\]
and abbreviate $\mathbf{A}_\ell := \mathbf{A}_\ell(W_\ell,u)$ and $\mathbf{B}_\ell := \mathbf{B}_\ell(W_\ell,u)$ so that the terms appearing in the product on the right-hand side of  \eqref{Zdef} are $U_\ell(W_\ell) = \exp(\mathbf{A}_\ell + \mathbf{B}_\ell)$.
Additionally, let
\begin{align}
\label{bbAdef}    \mathbb{A} &:= \exp\left(\sum_{\ell \geq 1} \bA_\ell+ \sum_{\ell \geq 2} \bB_\ell\right) -1
   \\
\label{bbBdef}    \mathbb{B}&:= \exp(\mathbf{B}_1) -1 = \sum_{k\geq 1}\frac 1 {k!} \bB_1^k.
\end{align}
With this notation, \eqref{Zdef} becomes
\begin{align} \notag Z &= -u T_{\leq 10}\left((\mathbb{A} + 1)(\mathbb{B} + 1) - 1 \right) 
= -u T_{\leq 10}(\mathbb{A} + \mathbb{B} + \mathbb{A}\mathbb{B}) \\
&=  \underbrace{-uT_{\leq 10}(\mathbf{B}_1)}_{=:L} + \underbrace{-u T_{\leq 10} \left(\sum_{k \geq 2} \frac{1}{k!} \mathbf{B}_1^k \right) - uT_{\leq 10}(\mathbb{A})- uT_{\leq 10}(\mathbb{A}\mathbb{B})}_{=:R}. \label{4terms}
\end{align}
We refer to $L$ as the leading term and $R$ as the remainder.
Denote by $Z_g$ (resp. $R_g$, $L_g$) the $g$-th Taylor coefficient in $u$ of $Z$ (resp. of $R$, $L$).

\begin{thm}\label{thm:lead}
    Asymptotically as $g\to \infty$, we have
\[
L_g \sim 
Z_g^{asymp}:=
\begin{cases}
    C_\infty^{ev} \frac{(-1)^{g/2} (g-2)! }{(2\pi)^g } & \text{for $g$ even} \\
    C_\infty^{odd} \frac{(-1)^{(g-1)/2} (g-2)! }{(2\pi)^{g} }
     & \text{for $g$ odd.} 
\end{cases}
\]
Furthermore, for $g\geq 100$ the relative error satisfies 
\[
    \frac{|L_g-Z_g^{asymp}|}{|Z_g^{asymp}|}<10^{-10}.
\]
\end{thm}

\begin{thm}\label{thm:remainder}
  We have that 
  \[
  E_g:=\frac{|R_g|}{(g-2)!(2\pi)^{-g}} \to 0 \quad \text{as $g\to\infty$},
  \]
and for all $g\geq 600$, we have that $E_g\leq \frac12 \min(C_\infty^{ev}, C_\infty^{odd})$.
\end{thm}

It is clear that Theorems \ref{thm:lead} and \ref{thm:remainder} together imply the first statement of Theorem ~\ref{thm:wt11 intro}.

\medskip
Furthermore $Z_g$ has been computed numerically up to $g=70$ in \cite{PayneWillwacher24}. We extended the computation up to $g=2400$, see Figure \ref{fig:theZplot}.
Looking at the numerical results one sees that $Z_g\neq 0$ for all $g\geq 9$ except $g=12$.
Hence it is clear that Theorems \ref{thm:lead} and \ref{thm:remainder} also imply the second statement of Theorem~\ref{thm:wt11 intro}. 
(The structure of the argument is such that it would suffice prove $E_g\leq c \cdot \min(C_\infty^{ev}, C_\infty^{odd})$ for all $g \geq N$ and then compute $Z_g$ exactly for $g \leq N$ for any suitable choice of $N$ and constant $c < 1 - 10^{-10}$. The choices $N = 600$ and $c = \frac{1}{2}$ achieve a clear exposition while leaving a feasible finite computation of $Z_g$ for $g \leq 600$.)

Our goal will hence be to prove Theorems \ref{thm:lead} and \ref{thm:remainder}.
To this end, we need to estimate each of the four terms in the formula for $Z$ given in \eqref{4terms}. 
The first two terms will be treated in Section \ref{sec:l1terms}. The third term is treated in Section \ref{sec:bbA} and the fourth term (the ``mixed terms") in Section \ref{sec:mixed estimates}.
We show the main Theorems \ref{thm:lead} and \ref{thm:remainder} in Sections \ref{sec:thm lead proof} and \ref{sec:thm rem proof}.

Along the way we need finite numerical computations at several places. These verifications are coded into a Mathematica notebook, which is available together with the other supplementary material on our github repository \url{https://github.com/wilthoma/polypointcount}.
Below we include cross-references into the Mathematica notebook, so that the interested reader can check the numerical computations used.

\begin{figure}
\includegraphics[scale=.7]{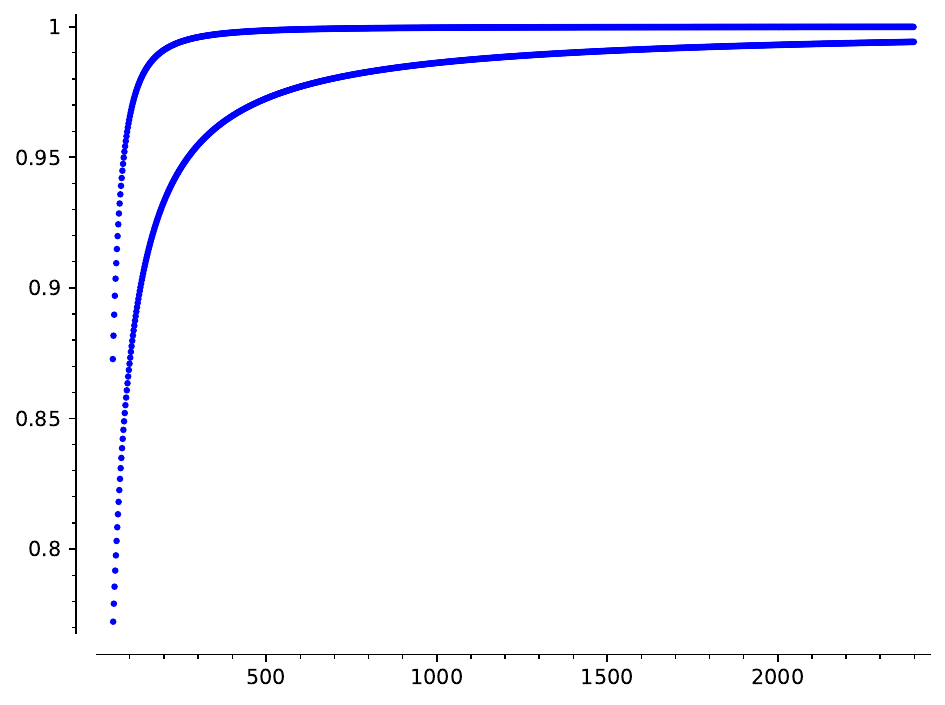}
\caption{\label{fig:theZplot} Plot of the ratio $Z_g/Z_g^{asymp}$ for $g$ between 30 and 2400. Note in particular that $Z_g$ is not zero in this range.
The two curves above correspond to whether $g$ is even or odd: in addition to the formula for $Z_g^{asymp}$ depending on the parity of $g$, the convergence rates are different depending on the parity of $g$.
The Euler characteristics for $g\leq 30$ can be found in \cite[Section 7]{PayneWillwacher24}. 
}
\end{figure}

\subsection{Notation for power series} 
We write a power series $a$ in a formal variable $u$ as $a = \sum_{N} a_N u^N$. If $a=\sum_N a_N u^N$ and $b=\sum_N b_N u^N$ are such power series in $u$, then 
$
a\leq b
$ 
means 
$
a_N\leq b_N \quad \text{for all $N$.}
$ 
Furthermore, if $a=\sum_N a_N(w) u^N$ is a power series where each coefficient $a_N(w)$ is a polynomial in the variable $w$ then we write 
\[
\| a\| := \sum_N \max_{w\in \mathbb C, |w|=1}|a_N(w)| u^N.
\]
We use the same notation when the coefficients  $a_N(w)$ are constants.
We have the formulas 
\begin{align}
    \label{equ:absval rules}
    \|a+b\|&\leq \| a\| + \| b\|
    &
    \| a b\| &\leq  \| a\| \| b\|.
\end{align}
If $a_N(w)  = \sum_{\alpha} a_{N, \alpha} w^\alpha$ is a polynomial in $w$, then, 
 by the Cauchy integral formula, we obtain
 \[|a_{N,\alpha}| = \left|\frac{1}{\alpha!} a_N^{(\alpha)}(0) \right| = \left| \frac{1}{2\pi i} \int_{\substack{w \in \cc \\ |w| = 1}} \frac{a_N(w)}{w^{\alpha+1}} dw \right| \leq \max_{w \in \cc, |w| = 1}|a_N(w)|.\]
It follows that
\begin{equation}
    \label{equ:T ab cheap}
\|T_{\leq 10}(a) \|\leq 
11 \|a\|.
\end{equation}

Because of these formulas, we will estimate $\|\cdots\|$ for $(\cdots)$ the various terms appearing in the remainder term $R$ of \eqref{4terms}.

Next suppose
\[
a = \sum_N \sum_\alpha a_{N,\alpha} u^N w^\alpha 
\quad \quad \mbox{and} \quad \quad
b = \sum_N \sum_\alpha b_{N,\alpha} u^N w^\alpha 
\]
are power series in $u$ and $w$.
In general, the truncation operator satisfies 
\[
T_{\leq \Gamma} (w^\alpha b) = 
\begin{cases}
    T_{\leq \Gamma-\alpha} (b) &\text{for $\alpha\leq \Gamma$} \\
    0 &\text{for $\alpha>\Gamma$}
\end{cases}.
\]
Using \eqref{equ:absval rules} and the above observation,
we have 
\begin{equation}\label{equ:T ab}
\|T_{\leq\Gamma} ab\|
\leq 
\sum_N \sum_{\alpha=0}^\Gamma
|a_{N,\alpha}|u^N \|T_{\leq \Gamma-\alpha} b\|.
\end{equation}

We will use the following lemma several times below.
\begin{lem}\label{lem:T 1-w}
    \[
        T_{\leq N} \left( (1-w)^n\right) =
        \begin{cases}
        1 & \text{for $n=0$} \\
        0 & \text{for $1\leq n\leq N$} \\
        (-1)^N \binom{n-1}{N} & \text{for $n\geq N+1$}
        \end{cases}.
    \]
    \end{lem}
    \begin{proof}
    The cases of $n\leq N$ are obvious. For $n>N$ we use induction on $N$ and have
    \begin{align*}
        T_{\leq N} \left( (1-w)^n\right) 
        &= 
        \sum_{j=0}^N 
        (-1)^j
        \binom{n}{j} 
        =\sum_{j=0}^{N-1}
        (-1)^j
        \binom{n}{j}
        +
        (-1)^N
        \binom{n}{N}
        \\&= 
        (-1)^N \left( \binom{n}{N} -\binom{n-1}{N-1} \right)
        =
         (-1)^N  \binom{n-1}{N} ,
    \end{align*}
    where we used the induction hypothesis to pass to the second line.
    \end{proof}

\subsection{Estimates for powers of \texorpdfstring{$\bB_1$}{bb1}}
\label{sec:l1terms}
Recall that the leading term is $L := -u T_{\leq 10}(\bB_1)$.
Note that $\bB_1$ has the series expansion 
\begin{align*}
    \bB_1&= 
    \sum_{r\geq 2} \frac{B_r}{r(r-1)} \left(\left(\frac{1}{-\frac{1}{u}+1-w}\right)^{r-1}- \left(\frac{1}{-\frac{1}{u}}\right)^{r-1}\right)
    \\&=
    \sum_{r\geq 2} \frac{B_r (-u)^{r-1} }{r(r-1)}
    \left(\left(\frac{1}{1-u(1-w)}\right)^{r-1}- 1\right)
    \\&=
    \sum_{r\geq 2} \frac{B_r (-u)^{r-1} }{r(r-1)}
    \sum_{j\geq 1} \binom{r+j-2}{j} u^j (1-w)^j.
\end{align*}
In the last line above, we used the power series expansion
\[
\frac 1{(1-x)^k} = \sum_{j=0}^\infty  \binom{k+j-1}{j} x^j,
\]
which is a special case of the binomial series.
Alternatively, the formula can be shown directly from the geometric series expansion using a standard ``stars and bars" argument.

The higher powers $\bB_1^k$ with $k > 1$ appear in the other terms in \eqref{4terms}, and we estimate them in this section. The truncations $T_{\leq \Gamma}\bB_1^k$ for different values of $\Gamma$ will be relevant when expanding the mixed terms in Section \ref{sec:mixed estimates}.
Using Lemma \ref{lem:T 1-w} and the formula above, we have
\[
\resizebox{.95\hsize}{!}{
$\begin{aligned}
-u T_{\leq \Gamma} \bB_1^k &= 
(-1)^\Gamma \sum_{r_1,\dots,r_k\geq 2}
\frac{B_{r_1}\cdots B_{r_k}}{r_1(r_1-1)\cdots r_k(r_k-1)}
\sum_{j_1,\dots,j_k\geq 1 \atop j_1+\dots+j_k\geq \Gamma+1}
\binom{r_1+j_1-2}{j_1}\cdots \binom{r_k+j_k-2}{j_k}
\\&\quad\quad\quad\quad
(-1)^{j_1+\dots+j_k}
(-u)^{r_1+\dots+r_k-k+1+j_1+\dots+j_k}
\binom{j_1+\dots+j_k-1}{\Gamma}.
\end{aligned}$}
\]
For $r\geq 2$, we have the following expression for the Bernoulli numbers:
\[
B_r = \frac{2\epsilon_r r!}{(2\pi)^r}\zeta(r),
\]
where $\epsilon_r = (-1)^{n+1}$ when $r = 2n$ is even, and is $0$ otherwise. Changing summation variables to $R_\alpha=r_\alpha+j_\alpha$, using $r_\alpha:=R_\alpha-j_\alpha$, $J=j_1+\dots+j_k$, and $R=R_1+\dots+R_k$, we have
\begin{align}\label{equ:preB1k}
    -u T_{\leq \Gamma} \bB_1^k &= 
    2^k(-1)^\Gamma
    \sum_{R_1,\dots,R_k\geq 3}
    \frac{(-u)^{R-k+1}}{(2\pi)^{R}}\prod_{\alpha=1}^k(R_\alpha-2)! \\  &\qquad \qquad \qquad \qquad \times \sum_{j_1,\dots,j_k\geq 1\atop { j_\alpha\leq R_\alpha-2
    \atop J\geq \Gamma+1}}
    \frac{(-2\pi)^J}{J (J-\Gamma-1)! \Gamma!}
    \binom{J}{j_1,\dots,j_k} 
    \prod_{\alpha=1}^k\epsilon_{r_\alpha}\zeta(r_\alpha), \notag
\end{align}
where 
$
    \binom{J}{j_1,\dots,j_k} = \frac{J!}{j_1!\cdots j_k!}
$
are the multinomial coefficients.

\subsubsection{The leading order term $L_{k,\Gamma}$} \label{sec:lt}

We will need to understand the asymptotic behavior of the leading and sub-leading order terms of \eqref{equ:preB1k}. For a given $R = R_1 + \ldots + R_k$, considering the product of factorials $\prod_{\alpha=1}^k(R_\alpha - 2)!$, one expects that the largest terms in \eqref{equ:preB1k} arise from those where one $R_\alpha$ is maximal and all other $R_\alpha$ are minimal (that is equal to $3$).
We thus define the leading order term of \eqref{equ:preB1k} to be composed of these summands. Concretely, we can assume $R_1 = R - 3(k-1)$ is maximal and then multiply by $k$ to account for the $k$ different choices for the index that is maximal. This assumption overcounts the case when all $R_\alpha = 3$, so we pull this case out in front of the sum:
\begin{equation}\label{equ:LkGamma}
\begin{aligned}
    L_{k,\Gamma} &:= a_{k,\Gamma}(-u)^{2k+1}+
    (-1)^\Gamma 2^k
    k
    \sum_{R\geq 3k+1}
    \frac{(-u)^{R-k+1}}{(2\pi)^{R}}
    (R-3k+1)!
    \\ &\quad\quad\quad\quad\quad\quad
    \sum_{\max(k,\Gamma+1)\leq J\leq R-2k} 
    \frac{(-2\pi)^J}{J (J-\Gamma-1)! \Gamma!}
    \frac{J!}{(J-k+1)!} 
    \epsilon_{r_1}\zeta(r_1)
    \zeta(2)^{k-1} 
    \\&= a_{k,\Gamma}(-u)^{2k+1}
    -(-1)^\Gamma 2^k\zeta(2)^{k-1} k
    \Re 
    \sum_{g\geq 2k+2}
    \frac{u^{g}}{(2\pi i)^{g+k-1}}(g-2k)!
    \\ &\quad\quad\quad\quad\quad\quad
    \underbrace{\sum_{\max(k,\Gamma+1)\leq J\leq g-k-1} 
    \frac{(2\pi i)^J (J-1)!}{ (J-\Gamma-1)! \Gamma!(J-k+1)!}
    \zeta(g-k-J+1)}_{=:C_{k,\Gamma,g}}.
\end{aligned} 
\end{equation}
For the second line we changed variables to $g:=R-k+1$, used that $\epsilon_r=-\Re i^r$ and that $$r_1=R_1-J_1=R-3k+3-(J-k+1)=R-J-2k+2=g-J-k+1.$$
Explicitly, the first term where all $R_\alpha = 3$ has coefficient
\[
    a_{k,\Gamma} := 
    \begin{cases} 
    (-1)^\Gamma 2^k\zeta(2)^k
    \frac{(-1)^k (k-1)!}{(2\pi)^{2k} (k-\Gamma-1)!\Gamma!} 
    & \text{for $k\geq \Gamma+1$} \\
    0 &\text{otherwise.}
    \end{cases}
\]
However, this term will be largely irrelevant for us because we are interested in the asymptotic behavior of the coefficients as $g\to\infty $. To this end, define the analytic function
\[
    \phi_{k,\Gamma}(x)
    :=
    \sum_{J=\max(k,\Gamma+1)}^\infty
    \frac{x^J (J-1)!}{ (J-\Gamma-1)! \Gamma!(J-k+1)!}.
\]

\begin{lem}\label{lem:err leading}
Let $C_{k,\Gamma,g}$ be as defined at the end of \eqref{equ:LkGamma}. We have that 
\[
    \lim_{g\to\infty}|\phi_{k,\Gamma}(2\pi i)-C_{k,\Gamma,g}| = 0.
\]
Furthermore, we have that 
\[
    |\phi_{k,\Gamma}(2\pi i)-C_{k,\Gamma,g}| \leq 10^{-14}
\]
for all $g\geq 100$, $\Gamma=0,1,\dots ,10$, and $k=1,2,3,4$.
\end{lem}
\newcommand{\ta}{\alpha}
\newcommand{\tb}{\beta}
\begin{proof}
  We have that for $g>\max(k,\Gamma+1)+k$
  \begin{align*}
    |\phi_{k,\Gamma}(2\pi i)-C_{k,\Gamma,g}|
    &\leq 
    \underbrace{
        \sum_{\max(k,\Gamma+1)\leq J\leq g-k-1} 
    \frac{(2\pi)^J (J-1)!}{ (J-\Gamma-1)! \Gamma!(J-k+1)!}
    |\zeta(g-k-J+1)-1|
     }_{=:\ta_{g}}
    \\&\quad\quad\quad +
     \underbrace{
        \sum_{J= g-k}^\infty 
    \frac{(2\pi)^J (J-1)!}{ (J-\Gamma-1)! \Gamma!(J-k+1)!}
     }_{=:\tb_{g}}.
  \end{align*}
  It is clear that $\tb_{g}$ is monotonically decreasing and that $\tb_{g}\to 0$ as $g\to \infty$, since the $\tb_{g}$ are the remainder terms of a convergent power series with non-negative coefficients.
  Furthermore, we have the following estimate valid for $g\geq 100$ and $k,\Gamma\leq 10$
  \begin{align*}
    \tb_{g}&\leq 
    \sum_{J= g-k}^\infty 
    \frac{(2\pi)^J }{ \Gamma!(J-\Gamma-k+1)!} 
    \underbrace{\frac{(J-1)\cdots(J-k+2)}{(J-\Gamma-1)\cdots(J-\Gamma-k+2)}}_{\leq 2^{k-2}}
    \\&\leq
    \frac{2^{k-2}}{\Gamma!}
    (2\pi)^{\Gamma+k-1}
    \sum_{J= g-2k-\Gamma+1}^\infty 
    \frac{(2\pi)^J }{ J!} 
    \\&\leq \frac{2^{k-2}}{\Gamma!}
    (2\pi)^{\Gamma+k-1}
    \frac{e^{2\pi}}{(g-2k-\Gamma+1)!} (2\pi)^{g-2k-\Gamma+1}  
    =\frac{2^{k-2}e^{2\pi}(2\pi)^{g-k}}{\Gamma!(g-2k-\Gamma+1)!}.
  \end{align*}
  In the last line we used Taylor's Theorem to estimate the remainder in the Taylor series for the exponential.

  For $\ta_{g}$ we use the estimate (see \cite[Proof of Proposition 9.4]{Borinsky})
  \[
  |\zeta(m)-1|\leq 2^{1-m}\zeta(2) 
  \]
  for $m\geq 2$
  to find 
  \[
    \ta_{g}\leq \sum_{\max(k,\Gamma+1)\leq J\leq g-k-1} 
    \underbrace{\frac{(2\pi)^J (J-1)!\zeta(2)}{ (J-\Gamma-1)! \Gamma!(J-k+1)!}}_{=:a_J}
    \frac{1}{2^{g-k-J}}:= \delta_{g} \to 0 \quad\text{as $g\to \infty$}.
  \]
  To show the explicit error bound, we check that $\delta_g$ is monotonically decreasing for $g$ large enough and then evaluate $\delta_g+\tb_{g}$ numerically.
  First note that $\delta_g$ satisfies the recursion
\[
    \delta_{g+1} = \frac12\delta_g+a_{g-k} =
    \frac12(\delta_g+2a_{g-k}),
\]
displaying it as the mean of $\delta_g$ and $2a_{g-k}$. Suppose for some $J_0$, we have the following:
\begin{itemize}
\item $a_J$ is monotonically decreasing for $J\geq J_0$;
\item $2a_{J_0} \leq \delta_{J_0+k}$.
\end{itemize}
Then we have 
\[
    \delta_{J_0+k}\geq \delta_{J_0+k+1} \geq 2 a_{J_0}
    \geq 2 a_{J_0+1}.
\]
Proceeding inductively we conclude that $\delta_{g}$ is monotonically decreasing for $g\geq g_0:=J_0+k$.
To check monotonicity of the $a_J$, we compute, using the assumptions $k\leq 4$, $\Gamma\leq 10$, that
\[
\frac{a_{J+1}}{a_J} = \frac{2\pi J}{(J - \Gamma)(J - k + 2)}
\leq 
\frac{2\pi J}{(J - 10)(J - 2)}.
\]
For $J\geq 18$, the right hand side is less than $1$, so that $a_J$ 
is monotonically decreasing in that range.

The second condition above can just be checked numerically by computing $\delta_{g_0}$.
Numerically, we find that indeed $2a_{g_0-k}\leq \delta_{g_0}$ for $g_0=100$ and all $k=1,2,3,4$, $\Gamma=0,\dots,10$.
Hence, we can conclude that for all $g\geq 100$ and $k,\Gamma$ as above
\[
    |\phi_{k,\Gamma}(2\pi i)-C_{k,\Gamma,g}| \leq
    \delta_{100} + \tb_{100}.
\]
A numerical computation using our estimate for $\tb_{100}$ shows that $\delta_{100} + \tb_{100}\leq 10^{-14}$ for all $k,\Gamma$ as above \numcheck{CKASYMP}.
\end{proof}

Let $L_{g,k,\Gamma}$ be the coefficient of $u^g$ in $L_{k,\Gamma}$.
Then by Lemma \ref{lem:err leading} and equation \eqref{equ:LkGamma}, we have for $g\geq 100$
\begin{equation}\label{equ:la def}
    \left|\frac{1}{k!} L_{g,k,\Gamma}\right|
    \leq 
    \frac{(g-2)!}{(2\pi)^g}
    \underbrace{
        \frac{2^k\zeta(2)^{k-1}}{(k-1)!(2\pi)^{k-1}}
    (|\phi_{k,\Gamma}(2\pi i)| + 10^{-14})
    \frac{(g-2k)!}{(g-2)!}.
    }_{=:\lambda_{g,k,\Gamma}}
\end{equation}

\begin{rem}\label{rem:err leading mono}
Numerically, we have \numcheck{LAEVAL}
\begin{align*}
    \lambda_{600,2,10} &\approx 0.000487044 &
    \lambda_{600,3,10} &\approx 4.24646\cdot 10^{-9}&
    \lambda_{600,4,10} &\approx 2.41299\cdot 10^{-14}
\end{align*}
It is obvious that $\lambda_{g,k,\Gamma}$ is monotonically decreasing in $g$. Hence for all $g\geq 600$:
\begin{align*}
    \lambda_{g,2,10} &< 0.000487044 +\epsilon &
    \lambda_{g,3,10} &< 4.24646\cdot 10^{-9}+\epsilon&
    \lambda_{g,4,10} &< 2.41299\cdot 10^{-14}+\epsilon
\end{align*}
with $\epsilon$ some number smaller than the stated numerical precision, e.g., $\epsilon=10^{-8}$. 
\end{rem}
\subsubsection{Proof of Theorem \ref{thm:lead}}
\label{sec:thm lead proof}
In the case $k=1$ and $\Gamma=10$, combining the first statement of Lemma \ref{lem:err leading} with \eqref{equ:LkGamma} states that 
\[
L_g = L_{g,1,10} \sim \left(-2\Re \frac1{i^g}\phi_{1,10}(2\pi i) \right) \frac{(g-2)!}{(2\pi)^g}.
\]
We compute for $g$ even
\begin{align*}
    -2\Re \frac1{i^g}\phi_{1,10}(2\pi i)
&= 
-2(-1)^{g/2} \sum_{J\geq 11\atop J \text{ even}} \frac{(-1)^{J/2} (2\pi)^J}{J(J-11)! 10!}
=
-(-1)^{g/2}
\sum_{j=6}^\infty \frac{(-4\pi^2)^{j} }{j(2j-11)! 10!},
\end{align*}
and similarly for $g$ odd
\begin{align*}
    -2\Re \frac1{i^g}\phi_{1,10}(2\pi i)
    &= 
    -2
(-1)^{(g-1)/2} \sum_{J\geq 11\atop J \text{ odd}} \frac{(-1)^{(J-1)/2} (2\pi)^J}{J(J-11)! 10!} \\
&=
-
(-1)^{(g-1)/2}
\sum_{j=5}^\infty \frac{4\pi (-4\pi^2)^{j} }{(2j+1)(2j-10)! 10!}.
\end{align*}
Hence,
\[
    -2\Re \frac1{i^g}\phi_{1,10}(2\pi i) 
    =
    \begin{cases}
      (-1)^{g/2} C_\infty^{ev}  & \text{for $g$ even} \\
      (-1)^{(g-1)/2} C_\infty^{odd}  & \text{for $g$ odd},
    \end{cases}
\]
and the first statement of Theorem \ref{thm:lead} follows. 
The second statement of Theorem \ref{thm:lead} is then immediately implied by the second statement of Lemma \ref{lem:err leading}.
\hfill\qed

\subsubsection{The sub-leading order term $L_{k,\Gamma}'$} \label{slt}
Similarly to the beginning of Section \ref{sec:lt}, considering the product of factorials $\prod_{\alpha=1}^k(R_\alpha - 2)!$, we expect that the next largest terms in \eqref{equ:preB1k} arise when
all but two $R_\alpha$ are minimal (i.e., equal to 3), and one $R_\alpha$ is equal to $4$. We call the sum of such terms the 
sub-leading term. Here we suppose that $k\geq 2$ for this definition to make sense.

Concretely, we then assume $R_1=R-3k+2\geq 4$, $R_2=4$ and $R_\alpha=3$ for $\alpha\geq 3$ and then multiply by $k(k-1)$ to account for the different choices of the indices.
Note that in the sum over $j_2$, the only contributing term is $j_2=2$ because otherwise $\epsilon_{r_2}=0$.
The sub-leading terms of \eqref{equ:preB1k} are then
\begin{align*}
    L_{k,\Gamma}' &= a_{k,\Gamma}'(-u)^{2k+3}+
    (-1)^\Gamma 2^k k(k-1)
    \sum_{R\geq 3k+3}
    \frac{(-u)^{R-k+1}}{(2\pi)^{R}}(R-3k)! 2!
    \\&\quad\quad\quad
    \sum_{\max(k+1,\Gamma+1)\leq J\leq R-2k} 
    \frac{(-2\pi)^J}{J (J-\Gamma-1)! \Gamma!}
    \binom{J}{J-k,2,1,\dots,1} 
    \epsilon_{r_1}\zeta(r_1)\zeta(2)^{k-1}.
\end{align*} 
Here we again need to take the term corresponding to $R_1=R_2=4$ out of the sum to avoid the overcounting with our prefactor $k(k-1)$. (The value of the constant $a_{k,\Gamma}'$ is irrelevant.)
Changing variables to $g=R-k+1$ and using 
$$r_1=R_1-J_1=R-3k+2-(J-k)=R-J-2k+2=g-J-k+1$$
and $\epsilon_{r_1}=-\Re i^{r_1}=-\Re i^{g-J-k+1}$ yields 
\begin{align*}
    L_{k,\Gamma}' &= a_{k,\Gamma}'(-u)^{2k+3}
    -(-1)^\Gamma 2^k\zeta(2)^{k-1} k(k-1)
    \Re 
    \sum_{g\geq 2k+1}
    \frac{(-u)^{g}i^{g-k+1}}{(2\pi)^{g+k-1}}(g-2k-1)!
    \\&\quad\quad\quad
    \underbrace{\sum_{\max(k+1,\Gamma+1)\leq J\leq g-k-1} 
    \frac{(-2\pi i)^J (J-1)!}{(J-k)! (J-\Gamma-1)! \Gamma!}
    \zeta(g-k-J+1)}_{=:C_{k,\Gamma,g}'}.
\end{align*} 

\begin{lem}\label{lem:err subleading}
    We have that 
    \[
        \lim_{g\to\infty}|\phi_{k+1,\Gamma}(2\pi i)-C_{k,\Gamma,g}'| = 0.
    \]
    Furthermore, we have that 
    \[
        |\phi_{k+1,\Gamma}(2\pi i)-C_{k,\Gamma,g}'| \leq 10^{-13}
    \]
    for all $g\geq 100$ and all $\Gamma=0,1,\dots ,10$ and $k=1,2,3$.
\end{lem}
\begin{proof}
The proof is completely parallel to that of Lemma \ref{lem:err leading}, except that the numerical verifications have to be redone with slightly altered formulas: The estimate $\delta_g$ of the previous proof becomes 
\[
\delta_g' := \sum_{\max(k+1,\Gamma+1)\leq J\leq g-k-1} 
    \frac{(2\pi)^J (J-1)!\zeta(2)}{ (J-\Gamma-1)! \Gamma!(J-k)!}
\frac{1}{2^{g-k-J}},
\]
and $\tb_g$ becomes
\begin{align*}
    \tb_{g}'&:=
    \sum_{J= g-k}^\infty 
    \frac{(2\pi)^J (J-1)!}{ (J-\Gamma-1)! \Gamma!(J-k)!}
    \\&\leq 
    \sum_{J= g-k}^\infty 
    \frac{(2\pi)^J }{ \Gamma!(J-\Gamma-k)!} 
    \underbrace{\frac{(J-1)\cdots(J-k+1)}{(J-\Gamma-1)\cdots(J-\Gamma-k+1)}}_{\leq 2^{k-1}}
    \\&\leq
    \frac{2^{k-1}}{\Gamma!}
    (2\pi)^{\Gamma+k}
    \sum_{J= g-2k-\Gamma}^\infty 
    \frac{(2\pi)^J }{ J!} 
    \\&\leq \frac{2^{k-1}}{\Gamma!}
    (2\pi)^{\Gamma+k}
    \frac{e^{2\pi}}{(g-2k-\Gamma)!} (2\pi)^{g-2k-\Gamma}  
    =\frac{2^{k-1}e^{2\pi}(2\pi)^{g-k}}{\Gamma!(g-2k-\Gamma)!}.
  \end{align*}

For the verification that $\delta_{100}'+\beta_{100}'<10^{-13}$ see \numcheck{CKASYMP2}.
\end{proof}

Let $L_{g,k,\Gamma}'$ be the coefficient of $u^g$ in $L_{k,\Gamma}'$.
Then by the above we have for $g\geq 600$
\begin{equation}\label{equ:lap def}
    \left|\frac{1}{k!} L_{g,k,\Gamma}'\right|
    \leq 
    \frac{(g-2)!}{(2\pi)^g}
    \underbrace{
        \frac{2^k\zeta(2)^{k-1}}{(k-2)!(2\pi)^{k-1}}
    (|\phi_{k+1,\Gamma}(2\pi i)| + 10^{-13})
    \frac{(g-2k-1)!}{(g-2)!}
    .}_{=:\lambda_{g,k,\Gamma}'}
\end{equation}

Recall that we required $k\geq 2$ in our analysis above. To simplify later formulas we define %
\begin{align*}
   L_{g,1,\Gamma}' &:= 0 
   &
   \lambda_{g,1,\Gamma}'&:=0.
\end{align*}

\begin{rem}\label{rem:err subleading mono}
    Numerically, we have \numcheck{LAEVAL}
    \begin{align*}
        \lambda_{600,2,10}' &\approx 9.65108\cdot 10^{-6} 
        &
        \lambda_{600,3,10}' &\approx 1.63969\cdot 10^{-10}
        &
        \lambda_{600,4,10}' &\approx 1.37324\cdot 10^{-15}
        .
    \end{align*}
It is clear that $\lambda_{g,k,\Gamma}'$ is monotonically decreasing in $g$.
Hence, the above numerical values yield bounds for all larger $g\geq 600$ as well:
\begin{align*}
    \lambda_{g,2,10}' &< 9.65108\cdot 10^{-6} +\epsilon
    &
    \lambda_{g,3,10}' &< 1.63969\cdot 10^{-10}+\epsilon
    &
    \lambda_{g,4,10}' &< 1.37324\cdot 10^{-15}+\epsilon
    .
\end{align*}
\end{rem}
\subsubsection{Absolute value estimates}\label{sec:absolute estimates}
From \eqref{equ:preB1k} using that $\zeta(n)\leq \zeta(2)$ for $n\geq 2$ we find (with $R:=R_1+\cdots+R_k$)
\begin{align*}
  \|u T_{\leq \Gamma} \bB_1^k\| &\leq 
  (2\zeta(2))^k 
  \sum_{R_1,\dots,R_k\geq 3}
  u^{R-k+1} \frac{(R_1-2)!\cdots (R_k-2)!}{(2\pi)^R}
  \\&\quad\quad\quad\quad\quad
  \sum_{J=\max(k,\Gamma+1)}^\infty
  \frac{(2\pi)^J}{J(J-\Gamma-1)!\Gamma!} \underbrace{\sum_{1\leq j_1,\dots,j_k\atop {j_1+\dots+j_k=J \atop j_\alpha\leq R_\alpha-2}}\binom{J}{j_1,\dots,j_k} }_{\leq k^J}
  \\&\leq
  (2\zeta(2))^k 
  \sum_{R_1,\dots,R_k\geq 3}
  u^{R-k+1} \frac{(R_1-2)!\cdots (R_k-2)!}{(2\pi)^R}
  \sum_{J=\max(k,\Gamma+1)}^\infty
  \frac{(2\pi k)^J}{J(J-\Gamma-1)!\Gamma!}
  \\&\leq
  (2\zeta(2))^k 
  \sum_{R_1,\dots,R_k\geq 3}
  u^{R-k+1} \frac{(R_1-2)!\cdots (R_k-2)!}{(2\pi)^R}
  \phi_\Gamma(2\pi k)
\end{align*}
with 
\[
  \phi_\Gamma(x):=
  \frac{1}{\Gamma!}
  \sum_{J=\Gamma+1}^\infty
  \frac{x^J}{J(J-\Gamma-1)!}.  
\]
For future reference, let us note the following bounds. 
\begin{lem} \label{lem:phibd}
For any nonnegative real $x$, we have
 \begin{align*}
    \frac{x^{\Gamma+1}}{(\Gamma+1)!} \leq \phi_\Gamma(x)&\leq \frac{1}{\Gamma!} x^\Gamma e^x 
\end{align*}   
\end{lem}
\begin{proof}
The lower bound is the $J = \Gamma + 1$ term in the sum of nonnegative quantities defining $\phi_\Gamma(x)$. For the upper bound, we write 
\[ \sum_{J=\Gamma+1}^\infty
  \frac{x^J}{J(J-\Gamma-1)!} = x^{\Gamma} \sum_{j = 1}^\infty \frac{x^j}{(j + \Gamma)(j-1)!} \leq x^\Gamma\sum_{j=1}^\infty \frac{x^j}{j!} \leq x^\Gamma e^x. \qedhere\]
\end{proof}

Simplifying further and changing summation variables to $g=R-k+1$:
\begin{align}\label{equ:pre Delta}
  \|u T_{\leq \Gamma} \bB_1^k\| &\leq
  (2\zeta(2))^k 
  \phi_\Gamma(2\pi k)
  \sum_{g\geq 2k+1}
  \frac{u^g}{(2\pi)^{g+k-1}}
  \underbrace{\sum_{N_1,\dots,N_k\geq 1\atop N_1+\dots+N_k=g-k-1}
  N_1!\cdots N_k!}_{=:F_k(g-k-1)}.
\end{align}

To go further, we need to study the quantities 
\[
    F_k(N):=\sum_{N_1,\dots,N_k\geq 1\atop N_1+\dots+N_k=N}
    N_1!\cdots N_k!
\]
appearing in the above formula. It is easy to see, though not used here, that they have the asymptotic behavior 
\[
    F_k(N)
    \sim k (N-k+1)! \quad\quad \text{as $N\to \infty$.}
\]
For us, a much coarser estimate shall suffice.
\begin{lem}\label{lem:Fk}
For all $k\geq 2$ and $N\geq k$, we have 
\[
        \frac{F_k(N)}{(N-k+1)!} \leq (3.1)^{k-1}.
\]
\end{lem}
\begin{proof}

We first show the case $k=2$ separately.
Consider $I_n = \sum_{ j =0}^n \binom{n}{j}^{-1}$, which satisfies the recurrence
   \[
   I_n = \frac{n+1}{2n} I_{n-1} + 1,
   \] 
proven in \cite{Rockett}.

Define $J_n := F_2(n) / n!$. Then $J_n = I_n - 2$ satisfies the recurrence
\[
J_n = \frac{n+1}{2n} J_{n-1} + 1/n.
\]
Multiplying by $n$, we find that $K_n := n F_2(n) / n!$ satisfies the recurrence
\[
K_n = \frac{n+1}{2n-2} K_{n-1} + 1.
\]
Note that $(n+1) / (2n-2)$ is decreasing and is less than $21/31$ when $n > 6$.  We compute
\begin{align*}
K_1 &= 0
&
K_2 &= 1
&
K_3 &= 2
&
K_4 &= \frac83
&
K_5 &= 3&
K_6 &= \frac{31}{10}.
\end{align*}
For $n > 6$, we have $(n+1)/(2n-2) < 21/31$, and it follows from the recurrence and induction on $n$ that $K_n < 31/10$, as claimed.

Next, we show the upper bound for general $k\geq 2$ by an induction on $k$.
The induction step from $k$ to $k+1$ is accomplished the following computation, using the induction hypothesis twice:
\begin{align*}
F_{k+1}(N)
&=\sum_{N_1=1}^{N-k}
N_1! F_{k}(N-N_1) \leq \sum_{N_1=1}^{N-k} N_1!(3.1)^{k-1} (N-N_1-k+1)!
\\&= (3.1)^{k-1}  F_2(N-k+1) \leq  (3.1)^{k} (N-k)!. \qedhere
\end{align*}
\end{proof}

We return to our estimate \eqref{equ:pre Delta}. Using the upper bound $F_k(N)\leq (N - k + 1)! (3.1)^{k-1}$
of Lemma \ref{lem:Fk} and summing over $k\geq k_0$ (with $k_0$ some constant), we obtain
\begin{equation} \label{r1}
\sum_{k\geq k_0}\frac1{k!}\|u T_{\leq \Gamma} \bB_1^k\|
\leq 
\sum_g u^g \frac{(g-2)!}{(2\pi)^g} 
\underbrace{
\sum_{k_0\leq k\leq (g-1)/2}(3.1)^{k-1}
\frac{(2\zeta(2))^k }{(2\pi)^{k-1}k!}\phi_\Gamma(2\pi k) \frac{(g-2k)!}{(g-2)!}.
}_{=:\Delta_{g,k_0,\Gamma}}
\end{equation}

\begin{lem}\label{lem:Deltap mono}
    Let $\Gamma\leq 10$ and $5\geq k_0\geq 2$.
Then the sequence $\Delta_{g,k_0,\Gamma}$ is monotonically decreasing in $g$ for $g\geq 150$. 

Furthermore, 
\[
\lim_{g\to \infty} \Delta_{g,k_0,\Gamma}=0.
\]
\end{lem}
\begin{proof}
We may show both statements at once by showing that the sequence
\[
    \tilde\Delta_{g,k_0}:=(g-2)\Delta_{g,k_0,\Gamma}
\]
is monotonically decreasing for $g\geq 150$. Here we suppress $\Gamma$ from the notation.
To this end, let $a_k=(3.1)^{k-1}\frac{(2\zeta(2))^k }{(2\pi)^{k-1}k!}\phi(2\pi k)$. For fixed $k$, the summands $a_k \frac{(g-2k)!}{(g-3)!}$ are monotonically decreasing in $g$.
Hence, if $g$ is odd, then clearly 
\[
    \tilde\Delta_{g+1,k_0}\leq \tilde\Delta_{g,k_0}.
\]
For $g$ even, there is one more summand in the expression for $\tilde\Delta_{g+1,k_0}$, and we can estimate 
\[
    \tilde\Delta_{g+1,k_0}-\tilde\Delta_{g,k_0}
    \leq  a_{k_0}\left(\frac{(g+1-2k_0)!}{(g-2)!} -\frac{(g-2k_0)!}{(g-3)!}\right) + \frac{a_{g/2}}{(g-3)!}.
\]
Hence, our sequence is monotonically decreasing if 
\begin{equation}\label{eq:Condition}
    a_{g/2}
    \leq 
    a_{k_0}
    (g-2k_0)!\left(1-\frac{g+1-2k_0}{g-2}\right).
\end{equation}
Inserting the definition of $a_k$, the inequality \eqref{eq:Condition} becomes
\[
 \left(3.1\frac{2\zeta(2)}{2\pi}\right)^{g/2-k_0}
\frac{\phi_\Gamma(\pi g)}{\phi_\Gamma(2\pi k_0)}
\leq (g-2k_0)! \frac{(g/2)!}{k_0!}\frac{2k_0-3}{g-2}.
\]
Applying Lemma \ref{lem:phibd}, we see that
inequality \eqref{eq:Condition} is implied by 
\[
    (\Gamma+1) \left(3.1\frac{2\zeta(2)}{2\pi}\right)^{g/2-k_0}
    \frac{(\pi g)^\Gamma e^{\pi g}}{(2\pi k_0)^{\Gamma+1}}
    \leq (g-2k_0)! \frac{(g/2)!}{k_0!}\frac{2k_0-3}{g-2}.
\]
Note that $\pi+\log(3.1\frac{2\zeta(2)}{2\pi})/2\leq 4$ so that
\[
    \left(3.1\frac{2\zeta(2)}{2\pi}\right)^{g/2}e^{\pi g} \leq e^{4g}.
\]
Hence, it suffices to show
\[
    \left(3.1\frac{2\zeta(2)}{2\pi}\right)^{-k_0} \frac{(\Gamma+1)}{\pi(2 k_0)^{\Gamma+1}}
    g^\Gamma e^{4 g}
    \leq (g-2k_0)! \frac{(g/2)!}{k_0!}\frac{2k_0-3}{g-2}.
\]

Furthermore, $g! \geq e(g/e)^g$ and $(g-2k_0)!>g! g^{-2k_0}$.
Therefore, \eqref{eq:Condition} is implied by 
\[
    \left(3.1\frac{2\zeta(2)}{2\pi}\right)^{-k_0}    
    \frac{(\Gamma+1)k_0!}{\pi(2k_0-3)(2k_0)^{\Gamma+1}}
    \leq e^{-5g+1} g^g
    \frac{(g/2)!}{g^{\Gamma+2k_0+1}}.
\]
The constant on the left-hand side is, for $k_0=2,\dots, 5$ and $\Gamma=0,\dots, 10$ bounded above by
\[
    \left(3.1\frac{2\zeta(2)}{2\pi}\right)^{-2}    
    \frac{(10+1)5!}{\pi(4-3)(4)^{0+1}}\lessapprox 40.
\]
Suppose that $g\geq e^5\approx 148.413$. Then $e^{-5g+1} g^g\geq 1$. For $k_0\leq 5$ and $\Gamma\leq 10$ 
inequality \eqref{eq:Condition} is hence implied by
\[
    40 
    \leq 
    \frac{(g/2)!}{g^{21}}.
\]
If $g\geq 100$ then $\frac{(g/2)!}{g^{21}}\gg 40$, and thus the inequality \eqref{eq:Condition} is satisfied. 
\end{proof}
    
Numerically we compute \numcheck{DELTAEVAL}
\begin{align*}
\Delta_{600,5,10} &\approx 0.10478. %
\end{align*}
Hence we have for all $g\geq 600$
\[
    \Delta_{g,5,10} < 0.10478 +\epsilon.
\]

\subsubsection{The other terms} \label{ot}
The other (non-leading or sub-leading) terms are estimated as before by 
\begin{align*}
    \|u T_{\leq \Gamma} \bB_1^k-L_{k,\Gamma}-L_{k,\Gamma}'\| &\leq
    (2\zeta(2))^k 
    \phi(2\pi k)
    \sum_{g\geq 2k+1}
    \frac{u^g}{(2\pi)^{g+k-1}}
    F_k'(g-k-1)
  \end{align*}
with  
\[
    F_k'(N)
    :=
    \sum_{1\leq N_1,\dots,N_k\leq N-k-1\atop N_1+\dots+N_k=N}
    N_1!\cdots N_k!.
\]
\begin{lem}
Asymptotically as $N\to\infty$
\[
    F_k'(N) \sim (6k(k-1)+2k(k-1)(k-2)) (N-k-1)!.
\]
Furthermore, there exist constants $A_k'$ such that
\[
    F_k'(N) \leq A_k'(N-k-1)!.
\]
The first few minimal choices of $A_k'$ are
\begin{align*}
    A_2'&=156/7 & A_3'&=6999/70 & A_4'&=9938/35 & A_5'&=13771/21.
\end{align*}
\end{lem}
\begin{proof}
    The asymptotic formula to be shown means that 
    \[
        \frac{F_k'(N)}{(6k(k-1)+2k(k-1)(k-2)) (N-k-1)!} \to 1 \quad\quad\text{as $N\to \infty$}.
    \]
Suppose that $N$ is large enough, specifically $N>3k$. For $0\leq q\leq k$ let $f_q$ be all terms of the sum $F_k'(N)$ for which one $N_\alpha$ is equal to $N-k-1-q$. There are $k$ choices for the $\alpha$ and hence
\[
f_q = k (N-k-1-q)! F_{k-1}(k+q+1).
\]
By Lemma \ref{lem:Fk} we have 
\[
f_q \leq k (N-k-1-q)! (3.1)^{k-2} (q+3)!.
\]
Then we decompose
\[
    F_k'(N) = \sum_{q=0}^{k} f_q +r,
\] 
where 
\[
    r
    :=
    \sum_{1\leq N_1,\dots,N_k\leq N-2k-2\atop N_1+\dots+N_k=N}
    N_1!\cdots N_k!.
\] 
is the sum of the terms of $F_k'(N)$ for which all $N_\alpha\leq N-2k-2$.
The maximum summand in $r$ is $(N-2k-2)!(k+4)!$, and the number of terms can be bounded by the total number of terms in $F_k(N)$ as $
\leq \binom{N-1}{k-1},
$ 
so that 
\[
r\leq \binom{N-1}{k-1} (N-2k-2)!(k+5)!.
\]
Now it is clear that as $N\to \infty$ and for $q\geq 1$
\begin{align*}
\frac{r}{(N-k-1)!} &\to 0 
&
\frac{f_q}{(N-k-1)!} &\to 0 .
\end{align*}
On the other hand,
\[
    \frac{f_0}{(N-k-1)!}
    =k F_{k-1}(k+1)
    = k (6(k-1)+2(k-1)(k-2)),
\]
thus the statement about the asymptotic behavior of $F_k'$ is shown.
But since for all $N$ we have $F_k'(N)>k (6(k-1)+2(k-1)(k-2)) (N-k-1)!$ we know that $F_k'(N)/(N-k-1)!$ must assume a maximal value for $N\in \mathbb{Z}_{\geq k}$, that we define as $A_k'$. This shows the second statement of the Lemma.

The explicit values of $A_k'$ for small $k$ are obtained by evaluating $F_k'(N)/(N-k-1)!$ for $N$ in a large enough range and taking the maximum \numcheck{FKPRIME}.
\end{proof}

We hence find
\begin{align} \label{eq:otherterms}
    \frac{1}{k!}\|u T_{\leq \Gamma} \bB_1^k-L_{k,\Gamma}-L_{k,\Gamma}'\| &\leq
    \frac{1}{k!}
    (2\zeta(2))^k 
    \phi(2\pi k)
    \sum_{g\geq 2k+1}
    \frac{u^g}{(2\pi)^{g+k-1}}
    A_k'
    (g-2k-2)! \notag
    \\&=
    \scalebox{.95}{$\displaystyle{\sum_{g\geq 2k+1}
    u^g\frac{(g-2)!}{(2\pi)^{g}}
    \underbrace{
        \left(
    \frac{1}{k! (2\pi)^{k-1}}A_k'(2\zeta(2))^k 
    \phi_\Gamma(2\pi k)
    \frac{(g-2k-2)!}{(g-2)!}
    \right).
    }_{=:\Delta_{g,k,\Gamma}'}}$}
\end{align}

\begin{rem}\label{rem:Deltapp mono}
    It is clear that $\Delta_{g,k,\Gamma}'$ is monotonically decreasing in $g$.  The concrete values we need are \numcheck{DELTAEVAL}
\begin{align*}
    \Delta_{600,2,10}' &\approx 0.641878 &
    \Delta_{600,3,10}' &\approx 0.0521099 &
    \Delta_{600,4,10}' &\approx 0.000578797 
    .
\end{align*}
Hence we have that for all $g\geq 600$
\begin{align*}
    \Delta_{g,2,10}' &< 0.641878 +\epsilon&
    \Delta_{g,3,10}' &< 0.0521099 +\epsilon&
    \Delta_{g,4,10}' &< 0.000578797 +\epsilon
    .
\end{align*}

\end{rem}

\subsubsection{Summary}
Let us summarize the estimates of this subsection in the form we need them below to show Theorem \ref{thm:remainder}. The proposition below controls the second term in \eqref{4terms}. 
\begin{prop}\label{prop:B1s alone}
Let $a_g$ be the coefficient of $u^g$ in the power series 
\[
-u T_{\leq 10} \sum_{k\geq 2} \frac 1{k!} \bB_1^k.
\]
Then 
\[
\lim_{g\to \infty} \frac{(2\pi)^g }{(g-2)!} a_g =0
\]
and for all $g\geq 600$ we have 
\[
    \frac{(2\pi)^g }{(g-2)!} |a_g| \leq 1.
\]
\end{prop}
\begin{proof}
We decompose 
\begin{align*}
    \Big\|u T_{\leq 10} \sum_{k\geq 2} \frac 1{k!} \bB_1^k\Big\|
    =
    \underbrace{\sum_{k=2}^4 \frac1{k!} \|(L_{k,10} +L_{k,10}')\|}_{*}
    &+ 
    \underbrace{\sum_{k=2}^4 \frac1{k!} \Big\|\Big(-u T_{\leq 10} \sum_{k\geq 2} \frac 1{k!} \bB_1^k- L_{k,10} -L_{k,10}'\Big)\Big\|}_{**} \\
    &
    +
    \underbrace{
        \Big\|u T_{\leq 10} \sum_{k\geq 5} \frac 1{k!} \bB_1^k\Big\|
    }_{***}.
\end{align*}
Asymptotic expressions for the Taylor coefficients of the leading and subleading terms $*$ have been computed above, and by Lemmas \ref{lem:err leading} and \ref{lem:err subleading}, it immediately follows that we have 
\[
\lim_{g\to\infty} (*)_g \frac{(2\pi)^g}{(g-2)!} =0.
\]
Furthermore, by Remarks \ref{rem:err leading mono} and \ref{rem:err subleading mono} we have for all $g\geq 600$
\[
    (*)_g \frac{(2\pi)^g}{(g-2)!} \leq 10^{-3}.
\]
The terms $(***)$ are estimated in Lemma \ref{lem:Deltap mono}, 
which implies that 
\[
    \lim_{g\to\infty} (***)_g \frac{(2\pi)^g}{(g-2)!} =0.
\] 
Furthermore, by the computation below the Lemma we have for all $g\geq 600$
\[
    (***)_g \frac{(2\pi)^g}{(g-2)!} \leq 0.2.
\]
Similarly, the terms $(**)$ are estimated in Equation \eqref{eq:otherterms}. For fixed $k$, it is clear that $\lim_{g \to \infty}\Delta_{g,k,10}' = 0$, so we have
\[
    \lim_{g\to\infty} (**)_g \frac{(2\pi)^g}{(g-2)!} =0.
\] 
By the computation in Remark \ref{rem:Deltapp mono}, we have for all $g \geq 600$
\[
    (**)_g \frac{(2\pi)^g}{(g-2)!} \leq 0.7.
\]
Adding up the above three contributions we see that the Proposition holds.
\end{proof}

The following proposition will be useful for estimating the size of the fourth term (``mixed terms") in \eqref{4terms}.
\begin{prop}\label{prop:tilde A estimate}
There is a constant $\tilde A\leq 10^{20}$ such that 
\[
   \sum_{\Gamma=0}^{10}
    \sum_{k\geq 1} \frac 1{k!} \|u T_{\leq \Gamma} \bB_1^k\|
    \leq 
    \tilde A
    \sum_{g\geq 2} u^g \frac{(g-2)!}{(2\pi)^g}.
\]
\end{prop}
\begin{proof}
Redefine $a_g$ to be the coefficient of $u^g$ in  $\sum_{\Gamma=0}^{10}
\sum_{k\geq 1} \frac 1{k!} \|u T_{\leq \Gamma} \bB_1^k\|$.
Then the statement of the proposition is that for all $g$ 
\[
    a_g \frac{(2\pi)^g}{(g-2)!} \leq 10^{20}.
\]
We may decompose $a_g$ as in the proof of the previous Proposition and conclude that 
\[
a_g \frac{(2\pi)^g}{(g-2)!} \leq 
\sum_{\Gamma=0}^{10}
\Big(\sum_{k=1}^4 (\lambda_{g,k,\Gamma}+\lambda'_{g,k,\Gamma})
+ \sum_{k=2}^4 \Delta_{g,k,\Gamma}'
+
\Delta_{g,5,\Gamma}
 \Big).
\]
Each of the sequences appearing in the sum on the right-hand side is monotonically decreasing for $g\geq 150$ by Lemmas \ref{lem:err leading}, \ref{lem:err subleading}, \ref{lem:Deltap mono} and Remark \ref{rem:Deltapp mono}.
Hence we may just explicitly evaluate the right-hand side for each $g\leq 150$ -- the maximum of the values obtained is the upper bound $\tilde A$. 
Numerically we find this maximum to be approximately $4.29987\cdot 10^{17}$ \numcheck{ATILDE}
Hence we have
\[
    \tilde A \lessapprox 10^{18}. \qedhere
\]
\end{proof}

\subsection{Estimates for \texorpdfstring{$\bA_\ell$}{baell} and \texorpdfstring{$\bB_\ell$}{bbell}} \label{sec:bbA}

We now turn towards estimating the third term in \eqref{4terms}, which is $-uT_{\leq 10}(\mathbb{A})$, where $\mathbb{A}$ is defined in \eqref{bbAdef}. We state three propositions below and prove them in each of the next three subsubsections.

First, we show that the coefficients of the series $\bA_\ell$ have at most exponential growth.
\begin{prop}\label{prop:A bound}
    We have 
    \[
    \|\bA_\ell\| \leq 
    \begin{cases}
   5\sum_{N\geq 1}\frac{2^N}{N} u^N & \text{if $\ell=1$} \\
    \frac{13}{2} \sum_{N\geq \ell/2}  u^N
    e^{\frac{4N}e}
    & \text{if $\ell\geq 2$.}
    \end{cases}
    \]
\end{prop}

Meanwhile, the coefficients of the series $\bB_\ell$ all have super-exponential growth. For $\ell=1$ this was discussed in detail in the previous section.
Our estimates for $\bB_\ell$ for $\ell\geq 2$ depend on a choice of some real number $\lambda\in (1,2)$. (We will eventually take $\lambda=4/3$.) Let $[x]$ denote the integer part of $x$.

\begin{prop}\label{prop:B bound 2}
    Let $\lambda>1$. Then there exists a constant $D_\lambda$ such that for $\ell\geq 2$
    \[
    \|\bB_\ell\| \leq 
    D_\lambda 
    \sum_{N\geq 2\ell}  \left(\frac{\ell}{2}\right)^{2[N/\ell]-1}
    [\lambda N/\ell]!  u^N 
    \]
For $\lambda = \frac43$ we may choose $D_{4/3}<57$.
\end{prop}

Both estimates may be combined to yield estimates on the exponentials:

\begin{prop}\label{prop:combined}
For each $\lambda \in (1, 2)$ the following estimates hold. 
    \begin{enumerate}
\item There is a constant $E_\lambda'$ such that 
\[
\Big\|\sum_{\ell \geq 2} \bB_\ell\Big\| \leq E_\lambda' \sum_{N\geq 4} [\lambda N/2]! u^N.
\]
For $\lambda=4/3$ we may take $E_{4/3}'<114$.
\item There is a constant $E_\lambda$ 
such that 
\[
\Big\|\sum_{\ell \geq 1} \bA_\ell+ \sum_{\ell \geq 2} \bB_\ell\Big\| 
\leq E_\lambda \sum_{N\geq 1} [\lambda N/2]! u^N.
\]
For $\lambda=4/3$ we may take $E_{4/3} <120$.
\item There is a constant $F_\lambda$ such that 
\[
\Big\|\exp\Big(\sum_{\ell \geq 1} \bA_\ell+ \sum_{\ell \geq 2} \bB_\ell\Big)\Big\|
\leq 
(1+F_\lambda\sum_{N\geq 1} (\lambda N/2)! u^N).
\]
In particular,
\begin{equation} \label{term3}
\|\mathbb{A}\|
\leq 
F_\lambda\sum_{N\geq 1} (\lambda N/2)! u^N.
\end{equation}
For $\lambda=\frac43$ we may choose $F_{4/3} <  10^{15}$.
    \end{enumerate}
\end{prop}
Note that by \eqref{equ:T ab cheap}, we have $\|T_{10}(\mathbb{A})\| \leq 11 \| \mathbb{A} \|$, so \eqref{term3} will provide an estimate on the third term in \eqref{4terms}.
\subsubsection{Proof of Proposition \ref{prop:A bound}}
We proceed in two cases.
\subsubsection*{Case $\ell=1$}
From its definition in \eqref{ABdefs}, we have 
\begin{align*}
\bA_1 
&=
(1-w)(\log(1-u) -1)
+(-\tfrac1u +\tfrac12 -w)\log(1-u(1-w))
\\&=
(1-w) \log(1-u)
+ \sum_{m\geq 2} \frac 1m u^{m-1}(1-w)^m
-(\tfrac 1 2-w) \cdot \sum_{m\geq 1} \frac 1m u^m (1-w)^m. 
\end{align*}
Hence
\begin{align*}
   \| \bA_1\| 
&\leq 
-2\log(1-u)
+ \sum_{m\geq 2} \frac 1m u^{m-1}2^m
+\frac32\cdot \sum_{m\geq 1} \frac 1m u^m 2^m  
\\
&\leq 
\sum_{m\geq 1}\left(\frac2m +\frac{2^{m+1}}{m+1}+ \frac32 \cdot \frac{2^m}{m}\right) u^m
\\&\leq
5\sum_{m\geq 1}\frac{2^m}{m} u^m.
\end{align*}

\subsubsection*{Case $\ell\geq 2$}
Let $\ell\geq 2$. 
First we expand the definition of $\mathbf{A}_\ell(X)$ as follows 
\begin{align}
\mathbf{A}_\ell(X) &=X\left(\log(\lambda_\ell E_\ell)-1 \right)+(-E_\ell+X-\textstyle{\frac 1 2} )\log(1-\textstyle{\frac X{E_\ell}})  \notag
\\&=
X\left(\log(\ell (1-u^\ell) u^\ell E_\ell)-1 \right)-(-E_\ell+X-\textstyle{\frac 1 2} )(\sum_{m\geq 1} \frac 1m \frac{X^m}{E_\ell^m} ) \notag
\\&=
X \log(\ell (1-u^\ell) u^\ell E_\ell)
+ \sum_{m\geq 2} \frac 1m \frac{X^m}{E_\ell^{m-1}}
-(X-\textstyle{\frac 1 2} )(\sum_{m\geq 1} \frac 1m \frac{X^m}{E_\ell^m} ). \label{al}
\end{align}
Next we expand the argument
\begin{equation} \label{arg}
    \ell (1-u^\ell) u^\ell E_\ell
    =
    (1-u^\ell)(1+\sum_{d\mid \ell\atop d\neq \ell} \mu(\ell/d) u^{\ell-d})
    =
    1-u^\ell +\sum_{d\mid \ell\atop d\neq \ell} \mu(\ell/d) (1-u^\ell)u^{\ell-d}
\end{equation}
to see that
\[
    \log(\ell (1-u^\ell) u^\ell E_\ell) 
    =
    -\sum_{m\geq 1} \frac1m \bigg(u^\ell -\sum_{d\mid \ell\atop d\neq \ell} \mu(\ell/d) (1-u^\ell)u^{\ell-d}\bigg)^m.
\]
Expanding the $m$-th power, we obtain a linear combination of powers of $u$ with exponents between $m\ell/2$ and $m(2\ell-1)$.
The number of terms in the parenthesis is bounded by $2\ell$. 
Hence the coefficients above are smaller in absolute value than those of the series 
\begin{align*}
    \sum_{m\geq 1} \frac1m 
    \sum_{m\ell/2\leq j \leq m(2\ell-1)} u^j (2\ell)^m 
    &\leq \sum_{j\geq \ell/2} u^j \bigg( \sum_{ j/(2\ell-1) \leq m \leq 2j/\ell } \frac1m (2\ell)^m \bigg) \\
    &\leq
    \sum_{j\geq \ell/2} u^j \frac{2j(2\ell-1)}{j\ell}(2\ell)^{[2j/\ell]} \\
    &\leq 4 \sum_{j\geq \ell/2} u^j(2\ell)^{[2j/\ell]}.
\end{align*} 

In order to estimate $\|\mathbf{A}_\ell\| = \|\mathbf{A}_\ell(W_\ell)\|$, note that for $\ell\geq 2$ we have
\begin{equation}\label{equ:Well leq 1}
\max_{|w|=1} W_\ell \leq 1.
\end{equation}
Hence, the contribution from the first term in \eqref{al} satisfies
\begin{equation}\label{equ:first Al}
\|W_\ell \log(\ell (1-u^\ell) u^\ell E_\ell) \|
\leq 4 \sum_{j\geq \ell/2} u^j(2\ell)^{[2j/\ell]}.
\end{equation}

To estimate the other terms in \eqref{al}, let us define 
\begin{equation}\label{equ:Qell def}
Q_\ell = \frac 1\ell 
\sum_{d\mid \ell\atop d\neq \ell} |\mu(\ell/d)| u^{-d}.
\end{equation}
so that  
\[
\Big\|\frac1{E_\ell^m}\Big\|
\leq 
\frac{(\ell u^\ell)^{m}}{(1-(\ell u^\ell) Q_\ell)^m}
=
\sum_{k\geq 0}
(\ell u^\ell)^{m+k} Q_\ell^k \binom{m+k-1}{k}.
\]
Using \eqref{equ:Well leq 1}, we hence have that 
\begin{align*}
\Big\| 
\sum_{m\geq 2} \frac 1m \frac{W_\ell^m}{E_\ell^{m-1}}
-(W_\ell-\textstyle{\frac 1 2} )(\sum_{m\geq 1} \frac 1m \frac{W_\ell^m}{E_\ell^m} ) 
\Big\|
&\leq 
\sum_{m\geq 2} \frac 1m \Big\| \frac{1}{E_\ell^{m-1}}\Big\| 
+\frac 3 2 \sum_{m\geq 1} \frac 1m \Big\| \frac{1}{E_\ell^m}\Big\|  
\\&\leq
\frac 5 2 \sum_{m\geq 1} \frac 1m \Big\| \frac{1}{E_\ell^m}\Big\| 
\\&\leq
\frac52 \sum_{m\geq 1} \frac 1m \sum_{k\geq 0} (\ell u^\ell)^{m+k} Q_\ell^k \frac{(m+k-1)!}{(m-1)!k!}
\\&=
\frac52 \sum_{m\geq 1,k\geq 0} (\ell u^\ell)^{m+k} Q_\ell^k \frac{(m+k-1)!}{m!k!}
.
\end{align*}
Introducing the variable $R=m+k$, we have 
\[
\frac52 \sum_{R\geq 1}(\ell u^\ell)^{R} \sum_{k=0}^{R-1} Q_\ell^k \frac{(R-1)!}{(R-k)!k!}
=
\frac52 \sum_{R\geq 1}(\ell u^\ell)^{R} \frac1{R-k} (1+ Q_\ell)^{R-1}.
\]

Note that $Q_\ell^k$ is a sum of powers of $1/u$ with exponents ranging between $k$ and $\ell k/2$. Also, the sum $Q_\ell$ has at most $\ell/2$ terms, each having a coefficient $\leq 1/\ell$. Hence we have the (relatively coarse) estimate
\begin{equation}\label{equ:Qell estimate}
  Q_\ell^k\leq  2^{-k} \sum_{k\leq j\leq \ell k/2} u^{-j}.
\end{equation}
Using this we obtain 
\[
\frac52 \sum_{R\geq 1}(\ell u^\ell)^{R} \sum_{k=0}^{R-1} Q_\ell^k \frac{(R-1)!}{(R-k)!k!}
\leq
\frac52 \sum_{R\geq 1}(\ell u^\ell)^{R}
\sum_{k=0}^{R-1}
\sum_{k\leq j\leq \ell k/2} u^{-j} 2^{-k}  \frac{(R-1)!}{(R-k)!k!}.
\]
Note that the lowest power of $u$ that can appear is $u^\ell$.
The coefficient of $u^N$ is 
\begin{align*}
    \frac52 
\sum_{R\geq 1, 0\leq k\leq R-1 \atop  k\leq \ell R-N\leq \ell k /2 }
\ell^R 2^{-k} \frac{(R-1)!}{(R-k)!k!}
&\leq 
\frac52 
\sum_{R=1}^{[2N/\ell]}\ell^R
\sum_{0\leq k\leq R-1}
2^{-k} \frac{(R-1)!}{(R-k)!k!}
\\&\leq
\frac52 
\sum_{R=1}^{[2N/\ell]}\ell^R
\frac{1}{R}
\sum_{0\leq k\leq R}
2^{-k} \frac{R!}{(R-k)!k!}
\\&=
\frac52 
\sum_{R=1}^{[2N/\ell]} \frac{1}{R}\left(\frac32\ell\right)^R
\leq \frac52 \left(\frac32\ell\right)^{[2N/\ell]}.
\end{align*}
For the first step we used that $\ell R - N\leq \ell k/2 \leq \ell(R-1)/2$, and hence $R\leq 2N/\ell -1 \leq [2N/\ell]$.
For the last step we use that the summands are monotonically increasing (since $\ell\geq 2$), and hence the sum is estimated by the number of terms times the last term.

Putting together \eqref{equ:first Al} and the above estimate, we find that 
\begin{align*}
    \|\bA_\ell\| &\leq 
   4 \sum_{N\geq \ell/2} u^N(2\ell)^{[2N/\ell]}
    +
    \frac52 \sum_{N\geq \ell}u^N
    \left(\frac32\ell\right)^{[2N/\ell]} 
    \leq 
    \frac{13}{2}\sum_{N\geq \ell/2}u^N (2\ell)^{[2N/\ell]}
   \\&\leq 
   \frac{13}{2} \sum_{N\geq \ell/2}  u^N
    e^{\frac{4N}e}.
\end{align*}
For the last inequality, we used \eqref{equ:sup xx}.
\hfill\qed

\subsubsection{Proof of Proposition \ref{prop:B bound 2}}
Using again the notation \eqref{equ:Qell def} and the estimate \eqref{equ:Well leq 1}, we have for $\ell\geq 2$
\begin{align*}
\|\bB_\ell\| &\leq \sum_{r\geq 2} \frac{B_r }{r(r-1)}
\bigg( 
  \left(\frac{\ell u^\ell}{1-\ell u^\ell Q_\ell - \ell u^\ell} \right)^{r-1}
  -
  \left(\frac{\ell u^\ell}{1-\ell u^\ell Q_\ell} \right)^{r-1}
\bigg)  
\\&=
\sum_{r\geq 2} \frac{B_r }{r(r-1)} \sum_{j\geq 1 \atop k\geq 0} (\ell u^\ell)^{r+k+j-1} Q_\ell^k \frac{(r+j+k-2)!}{(r-2)!j!k!}.
\end{align*}
Using that 
\[
|B_r| \leq r! \frac{2}{(2\pi)^r} \zeta(r) \leq r! \frac{2}{(2\pi)^r} \zeta(2),
\]
we find 
\[
    \|\bB_\ell\| \leq  2\zeta(2) 
  \sum_{r\geq 2}
  \frac{1}{(2\pi)^r}
  \sum_{j\geq 1 \atop k\geq 0}
  (\ell u^\ell)^{r+k+j-1}
  Q_\ell^k \frac{(r+j+k-2)!}{j!k!}.
\]
Substituting $R:=r+j\geq 3$, we have
\[
    \|\bB_\ell\| \leq  2\zeta(2) 
  \sum_{R\geq 3}
  \sum_{R-2\geq j\geq 1 \atop k\geq 0}
  \frac{1}{(2\pi)^{R-j}}
  (\ell u^\ell)^{R+k-1}
  Q_\ell^k \frac{(R+k-2)!}{j!k!}.
\]
Then the remaining sum over $j$ is 
\[
\sum_{j=1}^{R-2} \frac {(2\pi)^j}{j!}
\leq  \sum_{j=1}^{\infty} \frac {(2\pi)^j}{j!} = e^{2\pi}-1.
\]

Hence, 
\[
    \|\bB_\ell\| \leq 2\zeta(2)(e^{2\pi}-1)
  \sum_{R\geq 3}
  \frac{1}{(2\pi)^R}
  \sum_{ k\geq 0}
  (\ell u^\ell)^{R+k-1}
  Q_\ell^k \frac{(R+k-2)!}{k!}.
\]
Using again \eqref{equ:Qell estimate}, we find 
\[
    \|\bB_\ell\| \leq 2\zeta(2)(e^{2\pi}-1)
  \sum_{R\geq 3}
  \frac{1}{(2\pi)^R}
  \sum_{ k\geq 0}
  (\ell u^\ell)^{R+k-1}
  \frac{(R+k-2)!}{k!}
 \Big( 2^{-k}
  \sum_{k\leq j\leq \ell k/2} u^{-j}
  \Big).
\]
Denote by $v_{\ell,N}$ the coefficient of $u^N$ in $\|\mathbf{B}_\ell\|$.
Taking the coefficient of $u^{N}$ in the above series we then have 
\[
  v_{\ell,N} \leq 2\zeta(2) (e^{2\pi}-1)
  \sum_{R\geq 3, k\geq 0 \atop 
  N +k\leq (R+k-1) \ell \leq N+\ell k /2
  }  
  \frac{1}{(2\pi)^R}
  \underbrace{\frac{(R+k-2)!}{k!}}_{=(k+1)_{(R-2)}}2^{-k} \ell^{R+k-1}.
\]
For fixed $R$ the maximum allowed $k$ in the sum is $k=[2N/\ell]-2R+2$. Furthermore, $k\geq (1-1/\ell)^{-1}(N/\ell -R+1)$, so that the number of terms in the sum over $k$ is bounded by
\[
    2N/\ell-2R+2 -    (1-1/\ell)^{-1}(N/\ell -R+1) +1
    \leq N/\ell -R +2\leq  N/\ell-1.
\] 
The summands increase monotonically in $k$ and hence the sum is bounded by 
\[ v_{\ell,N} \leq
  2\zeta(2) (e^{2\pi}-1)([N/\ell]-1)
  \sum_{3\leq R\leq N/\ell +1
  }  
  \frac{1}{(2\pi)^R}
  \frac{([2N/\ell]-R)!}{([2N/\ell]-2R+2)!} 
  2^{-[2N/\ell]+2R-2}
  \ell^{[2N/\ell]-R+1}.
\]

Change summation variables to 
\[
\alpha := [N/\ell]+1 - R. 
\]
Then the above sum is
\begin{align*}
  v_{\ell,N}\leq &2\zeta(2) (e^{2\pi}-1)([N/\ell]-1) \\
  &\quad \times
  \sum_{0\leq \alpha\leq [N/\ell] - 2}  
  \frac{1}{(2\pi)^{[N/\ell]+1 -\alpha}}
  \frac{([2N/\ell]-[N/\ell]-1 +\alpha )!}
  {([2N/\ell]-2[N/\ell]+2\alpha)!} 
  \frac{\ell^{[2N/\ell]-[N/\ell] +\alpha}}{2^{[2N/\ell]-2[N/\ell]+2\alpha}}.
\end{align*}
We next use that 
\begin{align*}
    [2N/\ell]-1\leq 2[N/\ell]\leq [2N/\ell]
\end{align*}
and simplify our expression slightly to
  \begin{align*}
  v_{\ell,N}&\leq 
  2\zeta(2) (e^{2\pi}-1)([N/\ell]-1)
  \sum_{0\leq \alpha\leq [N/\ell] - 2}  
  \frac{1}{(2\pi)^{[N/\ell]+1 -\alpha}}
  \frac{([N/\ell] +\alpha )!}
  {(2\alpha)!} 
  \frac{\ell^{[2N/\ell]-[N/\ell] +\alpha}}{2^{[2N/\ell]-2[N/\ell]+2\alpha}}.
  \end{align*}
By the same reasoning we have, since $\ell\geq 2$,
\[
\left(\frac\ell 2\right)^{[2N/\ell]}\leq \left(\frac\ell 2\right)^{2[N/\ell]+1}
\]
and hence 
  \begin{align*}
    v_{\ell,N}&\leq 
  2\zeta(2) (e^{2\pi}-1)([N/\ell]-1)
  \sum_{\alpha=0}^{[N/\ell] - 2} 
  \frac{1}{(2\pi)^{[N/\ell]+1 -\alpha}}
  \frac{([N/\ell] +\alpha )!}
  {(2\alpha)!} 
  \frac{\ell^{[N/\ell]+1 +\alpha}}{2^{2\alpha+1}}
 \\&=
  2\zeta(2) (e^{2\pi}-1)([N/\ell]-1)
  \sum_{\alpha=0}^{[N/\ell] - 2} 
  \frac{1}{ 2\cdot \pi^{[N/\ell]+1 -\alpha}}
  \frac{([N/\ell] +\alpha )!}
  {(2\alpha)!} 
  \left(\frac{\ell}{2}\right)^{[N/\ell]+1 +\alpha}
  \\&\leq
  \zeta(2) (e^{2\pi}-1)([N/\ell]-1)
  \frac{1}{\pi^3}
  \left(\frac{\ell}{2}\right)^{2[N/\ell]-1}
  \sum_{\alpha=0}^{[N/\ell] - 2} 
  \frac{([N/\ell] +\alpha )!}
  {(2\alpha)!} 
  .
\end{align*}
For the last estimate we estimated the powers of $1/\pi$ and $\ell/2$ by the minimal (respectively maximal) power appearing in the sum.

\begin{lem}
For every $\lambda>1$ there is a constant $c_\lambda$ such that for all $n$
\[
    (n-1) \sum_{\alpha=0}^{n - 2} 
    \frac{(n +\alpha )!}
    {(2\alpha)!} 
    \leq c_\lambda [\lambda n]!
\]

For $\lambda=\frac43$ we may take $c_\lambda=2$.
\end{lem}
\begin{proof}
    The ratio of the consecutive terms for $\alpha - 1$ and $\alpha$ in the sum is $(2\alpha)(2\alpha - 1)/(n + \alpha)$. Solving for when this ratio is one gives the quadratic equation $n+\alpha- 2\alpha (2\alpha-1) = 0$.
    Let
    \[
    \alpha_0=\frac 18(3+\sqrt{9+16n})
    \]
    be the unique positive root of this equation.
    Then the summands are monotonically increasing for $\alpha\leq \alpha_0$ and monotonically decreasing afterwards.
    Hence the sum is bounded by, for $n$ large enough
    \[
    (n-1)^2 \frac{(n+\alpha_0)!}{(2\alpha_0)!}
    \leq \frac{(n+3+\frac12\sqrt{n})!}{(\frac14(3+4\sqrt{n}))!}
    \]
    It is clear that this grows slower than $[\lambda n]!$ (for any fixed $\lambda>1$), hence a constant $c_\lambda$ as in the Lemma exists.

    Concretely, for $\lambda=4/3$ and $n\geq 20$ we have 
    \[
        \frac{(n+3+\frac12\sqrt{n})!}{(\frac14(3+4\sqrt{n}))!}
        \leq 
        (n+3+\frac12\sqrt{n})!
        \leq 
        (4/3 n)!
        \leq 2
        (4/3 n)!,
    \]
    since 
    \[
        n+3+\frac12\sqrt{n} \leq \frac43 n
    \]
    for all $n\geq 20$.
    One then just checks numerically by explicit evaluation that the assertion of the Lemma (for $\lambda=4/3$, $c_\lambda=2$) also holds for all $n<20$.
\end{proof}

Hence for any $\lambda>1$ we have
\begin{align*}
    v_{\ell,N}
    &\leq 
  \zeta(2) (e^{2\pi}-1)
  \frac{c_\lambda}{ \pi^3}
  \left(\frac{\ell}{2}\right)^{2[N/\ell]-1}
  [\lambda N/\ell]!
  \\&=: D_\lambda 
  \left(\frac{\ell}{2}\right)^{2[N/\ell]-1}
  [\lambda N/\ell]!
  .
\end{align*}
Here $D_\lambda=\zeta(2) (e^{2\pi}-1)
\frac{c_\lambda}{ \pi^3}$, and numerical evaluation using $c_{4/3}=2$ yields $D_{4/3}\approx 56.7113$. 
\hfill\qed

\subsubsection{Proof of Proposition \ref{prop:combined}}
We prove each of the three parts, building up to an estimate on $\|\mathbb{A}\|$ in part (3).

\subsubsection*{Proof of part (1)}
Recall from Proposition \ref{prop:B bound 2} that for any fixed $\lambda>1$:
\[
        v_{\ell,N}
        \leq 
      D_\lambda 
      \left(\frac{\ell}{2}\right)^{2[N/\ell]-1}
      [\lambda N/\ell]!
      .
\]
Note also that 
\[
    v_{\ell,N} = 0 \quad \text{if $N<2\ell$}.
\]
We next estimate for fixed $N$
\begin{align*}
\sum_{\ell \geq 2} v_{\ell,N}
&=
\sum_{\ell = 2}^{[N/2]} v_{\ell,N}
\leq 
D_\lambda 
\sum_{\ell = 2}^{[N/2]}
\left(\frac{\ell}{2}\right)^{2[N/\ell]-1}
[\lambda N/\ell]!.
\end{align*}

\begin{lem}
For every $\lambda>1$ there is a constant $\tilde c_\lambda$ such that for all $N$
\[
    \sum_{\ell = 2}^{[N/2]}
    \left(\frac{\ell}{2}\right)^{2[N/\ell]-1}
    [\lambda N/\ell]!
    \leq \tilde c_\lambda 
    [\lambda N/2]!.
\]
For $\lambda=\frac43$ we may take $\tilde c_\lambda =2$.
\end{lem}
\begin{proof}
    The first statement of the Lemma is equivalent to boundedness of 
    \[
        \sum_{\ell = 2}^{[N/2]}
    \left(\frac{\ell}{2}\right)^{2[N/\ell]-1}
    \frac{[\lambda N/\ell]!}{[\lambda N/2]!}
    \]
as a sequence in $N$.
    Note that for $a > 0$, we have
    \begin{equation}\label{equ:sup xx}
    \sup_{x> 0} \left(a x\right)^{\frac{1}{x}} = e^{\frac ae},
    \end{equation}
    and hence
    \[
        \left(\frac{\ell}{2}\right)^{2[N/\ell]-1}
        \leq \left( N \cdot \frac{\ell}{2N} \right)^{2N/\ell} \leq
        c^N
    \]
    for $c=e^{\frac 1e}$. But then 
    \begin{align*}
        \sum_{\ell = 2}^{[N/2]}
        \left(\frac{\ell}{2}\right)^{2[N/\ell]-1}
        \frac{[\lambda N/\ell]!}{[\lambda N/2]!}
        &=
        1
        +
        \sum_{\ell = 3}^{[N/2]}
        \left(\frac{\ell}{2}\right)^{2[N/\ell]-1}
        \frac{[\lambda N/\ell]!}{[\lambda N/2]!}
        \\&\leq 
        1
        +
        c^N
        \sum_{\ell = 3}^{[N/2]}
        \frac{[\lambda N/\ell]!}{[\lambda N/2]!}
        \leq 1+
        [N/2] c^N \frac{[\lambda N/3]!}{[\lambda N/2]!}.
    \end{align*}
    But now $\frac{[\lambda N/3]!}{[\lambda N/2]!}$ decays super-exponentially in $N$ whereas $[N/2] c^N$ only grows exponentially. Hence the right-hand side converges to $1$ as $N\to \infty$, implying in particular that the sequence is bounded.

    To obtain the explicit value $\tilde c_{4/3}=2$ one sees that for $\lambda=4/3$ we have $[N/2] c^N \frac{[\lambda N/3]!}{[\lambda N/2]!}\ll 1$ already for $N\geq 20$. Hence we may just check numerically the assertion of the Lemma for all $N\leq 20$, see \numcheck{CLAMBDA}.
\end{proof}

Hence setting $E_\lambda':=D_\lambda  \tilde c_\lambda$ we find 
\[
    \sum_{\ell = 2}^{[N/2]} v_{\ell,N}\leq E_\lambda' [\lambda N/2]!.
\]
For $\lambda=4/3$ we have 
\[
    E_{4/3}':=D_{4/3}  \tilde c_{4/3} < 57 \cdot 2 = 114.
\]
This concludes the proof of part (1) of Proposition \ref{prop:combined}.

\subsubsection*{Proof of part (2)}
By Proposition \ref{prop:A bound} we have that 
\[
\|\bA_\ell\| \leq a \sum_{N\geq \ell/2} c^N u^N
\]
for all $\ell$, where we may take $a=\frac{13}{2}$ and $c=4.4$. 
Hence 
\[
\sum_{\ell\geq 1} \|\bA_\ell\| \leq a \sum_{N\geq 1}  2N c^N u^N.
\]

But since $[\lambda N/2]!$ grows super-exponentially there is some constant $a_\lambda$ such that
\[
a_\lambda[\lambda N/2]! > 2 a N c^N
\]
for all $N$.
Hence, using part (1) of Proposition \ref{prop:combined} we have
\[
\sum_{\ell\geq 1} \|\bA_\ell\| + \sum_{\ell\geq 2} \|\bB_\ell\|
\leq 
(E_\lambda'+a_\lambda)
\sum_{N\geq 1} [\lambda N/2]! u^N.
\]

Now, for $\lambda=4/3$ one has that $\frac{2 aN  c^N}{[2N/3]!}\ll 1$ for $N\geq 48$. 
Computing numerically the maximal value of $\frac{2aN c^N}{[2N/3]!}$ we see that the above estimate would yield a rather large constant $a_\lambda$ and hence a large value for $E_\lambda = E_\lambda'+a_\lambda$. 
However, we may also obtain a finer estimate by using the maximum of $E_\lambda'+1$ and the numbers
\[
    \frac1{[2N/3]!}\Big\|\sum_{\ell\geq 1} \bA_\ell + \sum_{\ell\geq 2} \bB_\ell\Big\|_N 
\]
for $N=1,\dots,48$, using that we know these quantities by explicit computation.
The numbers above are in fact all small ($\leq 2$), so that we can just take $E_{4/3}=E_\lambda'+1<115$ \numcheck{ELAMBDA}.

\subsubsection*{Proof of part (3)}
First, by explicitly computing the leading coefficient, we have
\[
    \big\|\sum_{\ell\geq 1} \bA_\ell +\sum_{\ell\geq 2}\bB_\ell \big\| = u + O(u^2).
\]
Using part (2) above, we then have
\begin{equation} \label{1} \exp\Big(\big\|\sum_{\ell\geq 1} \bA_\ell +\sum_{\ell\geq 2}\bB_\ell \big\| \Big) \leq 
\exp(u+E_\lambda \sum_{N\geq 2}[\lambda N/2]! u^N)
=
e^u \exp(E_\lambda \sum_{N\geq 2}[\lambda N/2]! u^N).
\end{equation}
Let $\tilde v_{N}$ be the coefficient of $u^N$ in $\exp(E_\lambda \sum_{N\geq 2}[\lambda N/2]! u^N)$.
Then we have (for $N\geq 1$)
\begin{equation} \label{vt1}
    \tilde v_{N}
    =
    \sum_{k=1}^N
\frac {E_\lambda^k} {k!} 
\sum_{N_1,\dots,N_k \geq 2\atop N_1+\dots+N_k =N}
[\lambda N_1/2]! \cdots [\lambda N_k/2]!.
\end{equation}

For the next result we drop the floor functions and define
\[
x! := \Gamma(x+1).
\]
By known monotonicity properties of the $\Gamma$-function we have for $x\geq y\geq 0.46164$
\[
x! \geq y!.
\]
In particular, for $x\geq 1$ we have $[x]!\leq x! \leq ([x]+1)!$.

\begin{lem} \label{22}
Let $a_\lambda:=2\cdot \lambda!$.
For each $\lambda>1$ and all $k$ and all $N$ we have
\[
    \sum_{N_1,\dots,N_k \geq 2\atop N_1+\dots+N_k =N}
    [\lambda N_1/2]! \cdots [\lambda N_k/2]!
    \leq \frac{a_\lambda^{k-1}}{(k-1)!} {(\lambda N/2)!}.
\]
\end{lem}
\begin{proof}
We drop the floor functions and write
\[
\sum_{N_1,\dots,N_k \geq 2\atop N_1+\dots+N_k =N}
    [\lambda N_1/2]! \cdots [\lambda N_k/2]!
    \leq 
    \sum_{N_1,\dots,N_k \geq 2\atop N_1+\dots+N_k =N}
    (\lambda N_1/2)! \cdots (\lambda N_k/2)!.
\]
The largest summand on the right-hand side is
\[
        (\lambda!)^{k-1} (\lambda (N-2k+2)/2)!.
\]
There are 
\[
\binom{N-2k +k-1}{k-1}
=\frac{(N-k-1)!}{(N-2k)!(k-1)!} 
\]
terms in the sum.
Hence the sum is bounded by 
\begin{align*}
&\frac{(N-k-1)!}{(N-2k)!(k-1)!} (\lambda !)^{k-1} (\lambda (N-2k+2)/2)!
\\&=
(\lambda !)^{k-1}
\frac{(N-k-1)!}{(N-2k)!(k-1)!} \frac{(\lambda (N-2k+2)/2)!}{(\lambda N/2)!} (\lambda N/2)!
\\
&\leq \frac{(\lambda !)^{k-1}}{{(k-1)!}}
    \frac{(N-k-1)(N-k-2)\cdots (N-2k+1)}{(\lambda N/2)(\lambda N/2-1)\cdots (\lambda N/2-k+1)} (\lambda N/2)!
\\&\leq 2^{k-1} \frac{(\lambda !)^{k-1}}{{(k-1)!}}(\lambda N/2)!
\leq (2 \cdot \lambda!)^{k-1} \frac{(\lambda N/2)!}{(k-1)!}. \qedhere
\end{align*}
\end{proof}
Substituting the bound from Lemma \ref{22} into \eqref{vt1}, we obtain 
\[
    \tilde v_{N}
    \leq 
   {(\lambda N/2)!}
    \sum_{k=1}^{N}
\frac {E_\lambda^k a_\lambda^{k-1}} {k!(k-1)!} 
\leq
{(\lambda N/2)!}
\underbrace{\sum_{k=1}^{\infty}\frac {E_\lambda^k a_\lambda^{k-1}} {k!(k-1)!}}_{=:F_\lambda'}
=: F_\lambda' {(\lambda N/2)!}.
\]

We still need to multiply our series by $e^u$. By the above, we have
\begin{align} 
e^u \exp(E_\lambda \sum_{N\geq 2}[\lambda N/2]! u^N) &\leq
e^u \Big(1+\sum_{N\geq 2} F_\lambda' (\lambda N/2)! u^N\Big) \label{2} \\
&=
e^u
+
F_\lambda' \sum_{N\geq 1}  u^N \Big(\sum_{j=0}^{N-2}\frac{1}{j!}(\lambda (N-j)/2)!\Big) \label{l2}
.
\end{align}
For $N=1,2$ the sum over $j$ trivially has only 0 or one term.
For $N\geq 3$ we have 
\begin{align*}
    \sum_{j=0}^{N-2}\frac{1}{j!}(\lambda (N-j)/2)!
    &=
    (\lambda N/2)!
    + (\lambda (N-1)/2)!
    +
    \sum_{j=2}^{N-2}\frac{1}{j!}(\lambda (N-j)/2)! \\
    &\leq 2(\lambda N/2)!+ \sum_{j=2}^{N-2}\frac{1}{j!}(\lambda (N-j)/2)!.
\end{align*}
The largest term in the remaining sum over $j$ is that for $j=2$, and there are $N-3$ terms. Hence the sum is bounded by 
\[
(N-3) (\lambda (N-2)/2)! \leq 2 (\lambda N/2)!.
\]
Hence 
\[
    \sum_{j=0}^{N-2}\frac{1}{j!}(\lambda (N-j)/2)!
    \leq 4 (\lambda N/2)!.
\]
We hence find 
\begin{equation}
 \label{3}    e^u \Big(1+\sum_{N\geq 2} F_\lambda' (\lambda N/2)! u^N\Big)
    \leq 
\sum_{N\geq 0} \frac{1}{N!} u^N
+
4F_\lambda'
\sum_{N\geq 2}(\lambda N/2)! u^N
    \leq 
    1+ F_\lambda \sum_{N\geq 1}(\lambda N/2)! u^N,
\end{equation}
with $F_\lambda=4F_\lambda'+2$. 
Here we used that for $N\geq 2$ we have $2(\lambda N/2)!\geq 2\geq \frac1{N!}$, and for $N=1$ we have $2(\lambda N/2)!\geq 2(1/2)!=\sqrt{\pi} \geq 1$. Combining \eqref{1}, \eqref{2}, and \eqref{3} yields the desired bound in the statement of the proposition.

For the constant $F_\lambda$ we obtain
\[
F_{\lambda } = 4F_{\lambda}' +2 =
4\sum_{k=1}^\infty \frac {E_\lambda^k (2\cdot \lambda!)^{k-1}}{k!(k-1)!}
+2.
\]
Evaluating this for $\lambda=4/3$ we obtain a value for $F_{4/3}$ of approximately (less than) $10^{15}$ \numcheck{FLAMBDA}.

\newcommand{\bbA}{\mathbb{A}}
\newcommand{\bbB}{\mathbb{B}}
\subsection{Estimating the mixed terms}\label{sec:mixed estimates}

We next want to estimate the mixed terms 
\[
\|-uT_{\leq 10}\bbA\,  \bbB \|,
\]
where $\bbA$ and $\bbB$ were defined in \eqref{bbAdef}  and \eqref{bbBdef}. For some $N_0$ (we will take $N_0=60$ later) we split 
$\bbA = \bbA_1 + \bbA_2$
where 
\[
    \bbA_1 = \sum_{N=1}^{N_0-1} \Big(\sum_j a_{1,N,j} w^j\Big) u^N
\]
is a finite polynomial in $u$ and 
\[
    \bbA_2 = \sum_{N=N_0}^{\infty} a_N(w) u^N
\]
is the tail of the series $\bbA$.

For small enough $N_0$ we may compute $\bbA_1$ explicitly on the computer.
We now split up
\begin{align*}
    \|-uT_{\leq 10}\bbA\, \bbB \|
    &\leq
    \|-uT_{\leq 10}\bbA_1\, \bbB \|
    +
    \|-uT_{\leq 10}\bbA_2\, \bbB \|
\\
&\leq \|\bbA_1\|
    \sum_{\Gamma=0}^{10}
    \|uT_{\leq \Gamma}\bbB \|
    +
    11 \|\bbA_2\| \|\bbB\|,
\end{align*}
using \eqref{equ:T ab cheap} and \eqref{equ:T ab}. 
We decompose in turn 
\[
    -uT_{\leq \Gamma}\bbB =
    \underbrace{L_{1,\Gamma}+L_{2,\Gamma}+L_{3,\Gamma}+L_{4,\Gamma} + L_{2,\Gamma}'+L_{3,\Gamma}'+L_{4,\Gamma}'}_{=:\cL_\Gamma}
    +
    \tilde R_\Gamma
\]
into the leading and subleading parts (as defined in Sections \ref{sec:lt} and \ref{slt}) and the remainder.
Then we have 
\begin{align*}
    \|-uT_{\leq 10}\bbA\, \bbB \|
&\leq \underbrace{\|\bbA_1\|
    \sum_{\Gamma=0}^{10}
    \|\cL_\Gamma \|}_{=:X_1}
    +
    \underbrace{\|\bbA_1\|
    \sum_{\Gamma=0}^{10}
    \|\tilde R_\Gamma \|}_{=:X_2}
    +
    \underbrace{11 \|\bbA_2\| \|\bbB\|}_{=:X_3}.
\end{align*}
We will discuss and estimate in turn each of the summands on the right.
We will denote by $X_{j,g}$ the coefficient of $u^g$ in $X_j$ so that 
\[
    X_{j} = \sum_g u^g X_{j,g}.
\]

\subsubsection{The first term $X_1$}
For small enough $N_0$ we may explicitly compute $\bbA_1$ on the computer.
For $\|\cL_\Gamma \|$ we have the estimates \eqref{equ:la def} and \eqref{equ:lap def} for the coefficients beyond the 100th.
Hence for $g\geq N_0+100$ we have 
\[
    X_{1,g}\frac{(2\pi)^g}{(g-2)!}
    \leq 
    \sum_{\Gamma=0}^{10}
    \sum_{N=1}^{N_0-1}
    |a_{1,N,10-\Gamma}| \frac{(g-N-2)! (2\pi)^N}{(g-2)!}
    \sum_{k=1}^4 (\lambda_{g-N,k,\Gamma}+\lambda_{g-N,k,\Gamma}')
\]
The right-hand side is a (non-negative-coefficient) linear combination of monotonically decreasing series that converge to zero as $g\to \infty$, and is hence itself monotonically decreasing (for $g\geq N_0+100$) and converging to zero.

We can hence find a bound on the error for $g\geq 600$ by explicit evaluation at $g=600$.
This yields an error bound of 
\[
    X_{1,g}\frac{(2\pi)^g}{(g-2)!} < 0.790506+\epsilon
\]
for all $g\geq 600$ \numcheck{XEVAL}.

\subsubsection{The second term $X_2$}
For the second term we proceed analogously, but we use the bound of Sections \ref{sec:absolute estimates} and \ref{ot} on the remainder term instead.
Precisely, we make use of \eqref{r1} with $k_0 = 5$ and \eqref{eq:otherterms} with $k = 2, 3, 4$.
This yields that for $g\geq 150+N_0$ we have 
\[
\resizebox{.95\hsize}{!}{$\displaystyle{
X_{2,g}\frac{(2\pi)^g}{(g-2)!}
\leq 
    \sum_{\Gamma=0}^{10}
    \sum_{N=1}^{N_0-1}
    |a_{1,N,10-\Gamma}| 
    \frac{(g-N-2)! (2\pi)^N}{(g-2)!}
    (\Delta_{g-N,5,\Gamma} 
    + \Delta_{g-N,4,\Gamma}'
    + \Delta_{g-N,3,\Gamma}'
    + \Delta_{g-N,2,\Gamma}').
}$}\]
Again the expression is monotonically decreasing in $g$ for $g\geq N_0+150$ by Lemma \ref{lem:Deltap mono} and Remark \ref{rem:Deltapp mono} and converging to zero. 
Hence it can be bounded for all $g\geq 600$ by evaluating the expression at $g=600$, and this yields the error bound valid for all $g\geq 600$:
\[
    X_{2,g}\frac{(2\pi)^g}{(g-2)!} < 0.0148095+\epsilon,
\]
see \numcheck{XEVAL}.

\subsubsection{The third term $X_3$}
For $\bbA_2$ we have the estimate (see Proposition \ref{prop:combined}) 
\[
   \|\bbA_2\| \leq \sum_{N=N_0}^{\infty}  F_\lambda (\lambda N/2)!  u^N.
\] 
On the other hand we know from Proposition \ref{prop:tilde A estimate} that 
\[
    \sum_{\Gamma=0}^{10}
    \Big\|uT_{\leq \Gamma}\Big(\sum_{k\geq 1}\frac 1 {k!} \bB_1^k\Big) \Big\|
    \leq \tilde A \sum_{g\geq 2} u^g
    \frac{(g-2)!}{(2\pi)^g}.
\]
Hence, we obtain
\[
\|\bbA_2\| 
    \sum_{\Gamma=0}^{10}
    \Big\|uT_{\leq \Gamma}\Big(\sum_{k\geq 1}\frac 1 {k!} \bB_1^k\Big) \Big\|
\leq 
F_\lambda \tilde A \sum_{g\geq N_0+2}u^g 
\sum_{N=N_0}^{g-2}
(\lambda N/2)! (g-N-2)!\frac{1}{(2\pi)^{g-N-2}}. 
\]
We are interested in computing the ratio of the $u^g$ coefficient to the asymptotic term $\frac{(g-2)!}{(2\pi)^g}$:
\newcommand{\eb}{\eta}
\begin{align*}
& \frac{(2\pi)^g}{(g-2)!}
 \Big(\|\bbA_2\| 
    \sum_{\Gamma=0}^{10}
    \Big\|uT_{\leq \Gamma}\big(\sum_{k\geq 1}\frac 1 {k!} \bB_1^k\big) \Big\|\Big)_{g}
\leq 
F_\lambda \tilde A
\underbrace{
\sum_{N=N_0}^{g-2}
\frac{(\lambda N/2)!(g-N-2)!}{(g-2)!}(2\pi)^{N+2}}_{=:\eb_g}.
\end{align*}
\begin{lem}
The expression $\eb_g$ above is monotonically decreasing with $g$ for $g$ large enough and converges to zero as $g\to \infty$.

Furthermore, for $N_0=60$, $\lambda=\frac43$ the sequence $\eb_g$ is monotonically decreasing at least for $g\geq 400$.
\end{lem}
\begin{proof}
Let $\tilde \eb_g=(g-2) \eb_g$. Then both monotonicity and convergence to zero of $\eb_g$ is implied by showing that $\tilde \eb_g$ decreases monotonically for $g$ large enough. 

It is clear that every summand 
\[
    \tilde a_{g,N}:=\frac{(\lambda N/2)!(g-N-2)!}{(g-3)!}(2\pi)^{N+2}
\]
contributing to $\tilde \eb_g$ is monotonically decreasing in $g$.
Hence we have 
\begin{align*}
\tilde \eb_g - \tilde \eb_{g+1}&\geq \tilde a_{g,N_0} - \tilde a_{g+1,N_0}-\tilde a_{g+1,g-1}
\\&=
(\lambda N_0/2)!(2\pi)^{N_0+2} (g-N_0-2)!\frac{(g-2)-(g-N_0-1)}{(g-2)!} - (\lambda(g-1)/2)! \frac{(2\pi)^{g+1}}{(g-2)!}.
\end{align*}
This being $\geq 0$ is equivalent to 
\[
 (\lambda N_0/2)!(2\pi)^{N_0+1}(N_0-1)
 \frac{(g-N_0-2)!}{(\lambda(g-1)/2)!} \geq  (2\pi)^{g}.
\]
The term on the left-hand side is super-exponentially growing in $g$ and hence it is clear that it will dominate the exponential term on the right-hand side for $g$ large enough, thus establishing that 
the sequence $\tilde \eb_g$ is
monotonically decreasing for large enough $g$.

To obtain the explicit numeric estimate for monotonicity of $\eb_g$ we redo the above argument (for $\eb_g$ instead of $\tilde \eb_g$) and find that monotonicity (i.e., $\eb_{g+1}\leq \eb_g$) holds if 
\[
 (\lambda N_0/2)!(2\pi)^{N_0+1}N_0
 \frac{(g-N_0-2)!}{(\lambda(g-1)/2)!} \geq  (2\pi)^{g}.
\] 
For the explicit values $N_0=60$, $\lambda=4/3$ this becomes 
\[
C \frac{(g-N_0-2)!}{(2(g-1)/3)!} \geq  (2\pi)^{g}
\]
with $C\approx 2\cdot 10^{169}$ a constant.
We may estimate 
\[
    \frac{(g-N_0-2)!}{(2(g-1)/3)!} \geq (2(g-1)/3)^{g/3 -N_0-2}
\]
For $g\geq 400$ we have $2(g-1)/3>(2\pi)^3$ and we have 
\[
    \frac{(g-N_0-2)!}{(2(g-1)/3)!} \geq (2\pi)^{g-3N_0-6}.
\]
But since $C>(2\pi)^{3N_0+6}$ (for $N_0=60$) we have established monotonicity.
\end{proof}

We may evaluate numerically (for $\lambda=4/3$, $N_0=60$) \numcheck{X3EVAL}
\[
F_{4/3} \tilde A \eb_{600} < 10^{-7}.
\]
Hence by the lemma above we conclude that for all $g\geq 600$
we have 
\[
    \frac{(2\pi)^g}{(g-2)!}
    \Big(\|\bbA_2\| 
       \sum_{\Gamma=0}^{10}
       \Big\|uT_{\leq \Gamma}\big(\sum_{k\geq 1}\frac 1 {k!} \bB_1^k\big) \Big\|\Big)_{g}
     < 10^{-7}.
\]

\subsubsection{Summary}
Summarizing, in this subsection we have shown the following result.
\begin{prop}\label{prop:mixed summary}
The Taylor coefficients $\|uT_{\leq 10}\bbA\,  \bbB \|_g$ of the series 
\[
\|uT_{\leq 10}\bbA\,  \bbB \|
\]
satisfy
\[
\lim_{g\to \infty} \frac{(2\pi)^g}{(g-2)!} \|uT_{\leq 10}\bbA\,  \bbB \|_g =0.
\]
Furthermore, for all $g\geq 600$ we have 
\[
    \frac{(2\pi)^g}{(g-2)!} \|uT_{\leq 10}\bbA\,  \bbB \|_g 
    \leq 1.
\]
\end{prop}
\begin{proof}
The statement about the limit follows since this is true for each $X_{j,g}$ separately ($j=1,2,3$).
For the error bound we just add up the error bounds found above for each of the $X_{j,g}$ to obtain
\[
    \frac{(2\pi)^g}{(g-2)!} \|uT_{\leq 10}\bbA\,  \bbB \|_g 
    \leq \frac{(2\pi)^g}{(g-2)!}( X_{1,g}+X_{2,g}+X_{3,g})\leq 0.790506+0.0137104+10^{-7}<1.
\]
\end{proof}

\subsection{Proof of Theorem \ref{thm:remainder}}
\label{sec:thm rem proof}
\label{sec:asympend}

Finally we prove Theorem \ref{thm:remainder} on the size and decay properties of the error term $R_g$, or respectively that of the relative error $E_g = \frac{|R_g| (2\pi)^g}{(g-2)!}$.

We decompose $\|R\|$ from \eqref{4terms} as follows:
\begin{align*}
    \|R\|
    \leq& 
u \Big\| T_{\leq 10}\sum_{k\geq 2} \frac1{k!}\bB_1^k \Big\|
+ u \| T_{\leq 10}(\mathbb{A})\|
 +u\|T_{\leq 10}(\mathbb{A} \mathbb{B})\|.
\end{align*}

The first term has been estimated in Proposition \ref{prop:B1s alone} above. Its contribution to $E_g$ also goes to zero as $g\to\infty$ and the total contribution to the error $E_g$  is $<1$ for all $g\geq 600$.

The second term satisfies $u\| T_{\leq 10} (\mathbb{A})\| \leq 11 u\|\mathbb{A} \|$. By Proposition \ref{prop:combined}(3) with $\lambda = 4/3$, we can bound the growth of the coefficients by a constant times $[2g/3]!$; in particular,
their contribution to the relative error $E_g$ goes to zero as $g\to \infty$. Furthermore, these make only a negligible contribution to the error $E_g$ ($<10^{-100}$) for all $g\geq 600$.

The third term has been estimated in Section \ref{sec:mixed estimates} above, see Proposition \ref{prop:mixed summary}. As shown there the coefficients all approach zero as $g\to \infty$, and the joint contribution to the error for $g\geq 600$ is bounded by $1$.
Summing all contributions we see that still 
\[
E_g < \frac12 \min \{ C_{\infty}^{ev},  C_{\infty}^{odd}\}.
\]
for all $g\geq 600$, showing the error bound in Theorem \ref{thm:remainder}.
\hfill\qed

\subsection{Outlook: Extensions to $n>0$ and Conjecture \ref{conj:polynomial}}
\label{sec:conj polynomial}
As discussed in the introduction, several parts of Conjecture \ref{conj:polynomial} are in fact theorems. The only remaining open implication is $(1)\Rightarrow (3)$ for $n\geq 1$. We have checked this part on the computer for a finite range of $g,n$ by computing $\chi_{11}(\M_{g,n})$ using formula \eqref{equ:ec11_withn}.
It turns out that the only $(g,n)$ with $g\geq 1$, $3g+2n\geq 25$, $g+n<150$ such that $\chi_{11}(\M_{g,n})=0$ are $(g,n)=(8,1)$ and $(g,n)=(12,0)$. But in these cases $\chi_{13}(\M_{g,n})\neq 0$. Hence we conclude that Conjecture \ref{conj:polynomial} holds true for all $g$, $n$ such that $g+n<150$.

Generally, we expect that our asymptotic analysis of $\chi_{11}(\M_{g})$ above can be generalized to $\chi_{11}(\M_{g,n})$, $n>0$.
We expect a similar asymptotic behavior as $g\to \infty$ for each finite $n$.
However, there is also an additional difficulty that is apparent from Figure \ref{fig:two ecs}. Our leading order term that dominates at large $g$ can be seen not to contribute in the region of low $g$, but possibly large $n$.
Hence to show non-vanishing outside of a finite domain in the $(g,n)$-plane, one would need to consider also a transition region, in which several terms compete. These produce a zero-crossing of the Euler characteristic, as is visible in the plot of Figure \ref{fig:two ecs}. 
Here one would need to show that the value zero is in fact never attained. Alternatively, one could use other means, for example the non-vanishing of the weight 13 or 15 Euler characteristic to show that close to the zero crossings one does not have polynomial point counts either.

\begin{figure}
    \includegraphics[scale=.75]{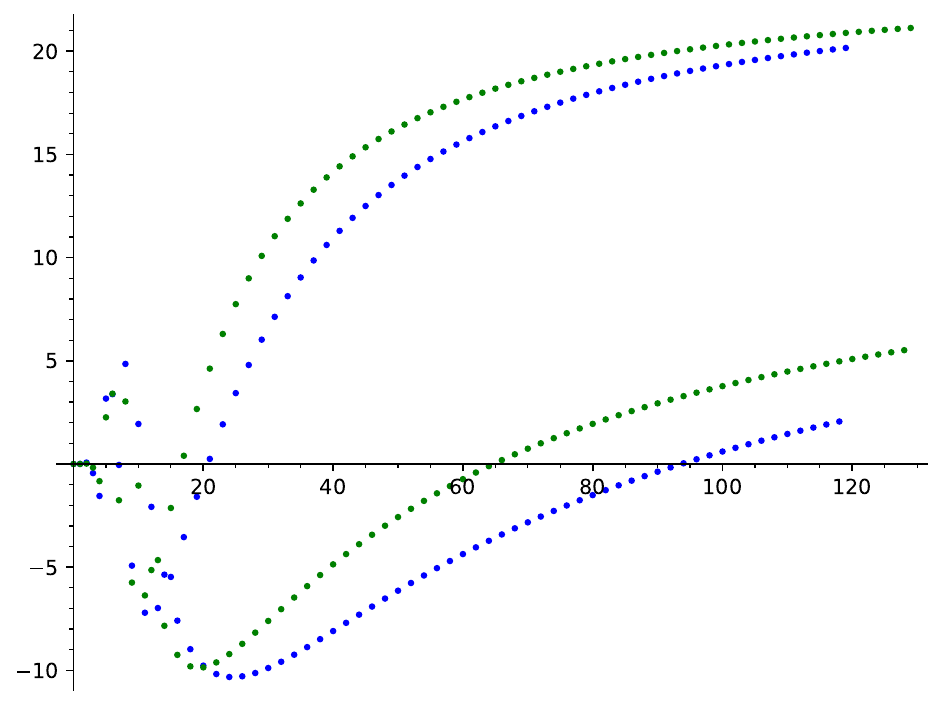}
    \caption{\label{fig:two ecs} The plot shows $\chi_{11}(\M_{g,n})\frac{ (-1)^{g(g-1)/2} (2\pi)^g }{2(n+g-2)!}$ for $n=20$ (green) and $n=30$ (blue).
    }
\end{figure}

\bibliographystyle{amsalpha}
\bibliography{refs}

\providecommand{\bysame}{\leavevmode\hbox to3em{\hrulefill}\thinspace}
\providecommand{\MR}{\relax\ifhmode\unskip\space\fi MR }
\providecommand{\MRhref}[2]{%
  \href{http://www.ams.org/mathscinet-getitem?mr=#1}{#2}
}
\providecommand{\href}[2]{#2}
\begin{thebibliography}{vdBE05}

\bibitem[AC98]{ArbarelloCornalba}
Enrico Arbarello and Maurizio Cornalba, \emph{Calculating cohomology groups of
  moduli spaces of curves via algebraic geometry}, Inst. Hautes \'{E}tudes Sci.
  Publ. Math. (1998), no.~88, 97--127 (1999). \MR{1733327}

\bibitem[Beh93]{Behrend}
Kai Behrend, \emph{The {L}efschetz trace formula for algebraic stacks}, Invent.
  Math. \textbf{112} (1993), no.~1, 127--149. \MR{1207479}

\bibitem[Ber]{BergstromData}
Jonas Bergstr{\"o}m, \emph{Cohomology of moduli spaces of curves},
  \url{https://github.com/jonasbergstroem/Cohomology-of-moduli-spaces-of-curves}.

\bibitem[BF23]{BergstromFaber}
Jonas Bergstr{\"o}m and Carel Faber, \emph{Cohomology of moduli spaces via a
  result of {C}henevier and {L}annes}, \'{E}pijournal G\'{e}om. Alg\'{e}brique
  \textbf{7} (2023), no.~20, 14 pp.

\bibitem[BFP24]{BergstromFaberPayne}
Jonas Bergstr\"{o}m, Carel Faber, and Sam Payne, \emph{Polynomial point counts
  and odd cohomology vanishing on moduli spaces of stable curves}, Ann. of
  Math. (2) \textbf{199} (2024), no.~3, 1323--1365. \MR{4740541}

\bibitem[Bor24]{Borinsky}
Michael Borinsky, \emph{On the {E}uler characteristic of the commutative graph
  complex and the top weight cohomology of {$\mathcal{M}_{g}$}}, preprint
  arXiv:2405.04190, 2024.

\bibitem[CGP21]{CGP21}
Melody Chan, S{\o}ren Galatius, and Sam Payne, \emph{Tropical curves, graph
  complexes, and top weight cohomology of $\mathcal{M}_g$}, J. Amer. Math. Soc.
  \textbf{34} (2021), no.~2, 565--594. \MR{4280867}

\bibitem[CL19]{ChenevierLannes}
Ga\"{e}tan Chenevier and Jean Lannes, \emph{Automorphic forms and even
  unimodular lattices}, Ergebnisse der Mathematik und ihrer Grenzgebiete. 3.
  Folge. A Series of Modern Surveys in Mathematics, vol.~69, Springer, Cham,
  2019. \MR{3929692}

\bibitem[CL22]{CL-CKgP}
Samir Canning and Hannah Larson, \emph{On the {C}how and cohomology rings of
  moduli spaces of stable curves}, To appear in J. Eur. Math. Soc.
  arXiv:2208.02357, 2022.

\bibitem[CL24]{CanningLarson789}
\bysame, \emph{The {C}how rings of the moduli spaces of curves of genus 7, 8,
  and 9}, J. Algebraic Geom. \textbf{33} (2024), no.~1, 55--116. \MR{4693574}

\bibitem[CLP23]{CanningLarsonPayne}
Samir Canning, Hannah Larson, and Sam Payne, \emph{The eleventh cohomology
  group of {$\overline{\mathcal{M}}_{g,n}$}}, Forum Math. Sigma \textbf{11}
  (2023), no.~Paper No. e62, 18 pp.

\bibitem[CLP24]{CLP-STE}
\bysame, \emph{Extensions of tautological rings and motivic structures in the
  cohomology of {$\overline{\mathcal{M}}_{g,n}$}}, Forum Math. Pi \textbf{12}
  (2024), Paper No. e23, 31 pp.

\bibitem[CLPW]{SupplementaryNotebook}
Samir Canning, Hannah Larson, Sam Payne, and Thomas Willwacher,
  \emph{Supplementary mathematica notebook}, available at
  \url{https://github.com/wilthoma/polypointcount/wt11asymptotics.nb}.

\bibitem[Del71]{DeligneHodgeII}
Pierre Deligne, \emph{Th\'{e}orie de {H}odge. {II}}, Inst. Hautes \'{E}tudes
  Sci. Publ. Math. (1971), no.~40, 5--57. \MR{498551}

\bibitem[Del74a]{DeligneWeilI}
\bysame, \emph{La conjecture de {W}eil. {I}}, Inst. Hautes \'Etudes Sci. Publ.
  Math. (1974), no.~43, 273--307. \MR{340258}

\bibitem[Del74b]{DeligneHodgeIII}
\bysame, \emph{Th\'eorie de {H}odge. {III}}, Inst. Hautes \'Etudes Sci. Publ.
  Math. (1974), no.~44, 5--77. \MR{498552}

\bibitem[FP13]{FaberPandharipande13}
Carel Faber and Rahul Pandharipande, \emph{Tautological and non-tautological
  cohomology of the moduli space of curves}, Handbook of moduli. {V}ol. {I},
  Adv. Lect. Math. (ALM), vol.~24, Int. Press, Somerville, MA, 2013,
  pp.~293--330. \MR{3184167}

\bibitem[Get98]{Getzler}
Ezra Getzler, \emph{The semi-classical approximation for modular operads},
  Comm. Math. Phys. \textbf{194} (1998), no.~2, 481--492. \MR{1627677}

\bibitem[GK98]{GetzlerKapranov}
Ezra Getzler and Mikhail Kapranov, \emph{Modular operads}, Compositio Math.
  \textbf{110} (1998), no.~1, 65--126. \MR{1601666}

\bibitem[GP03]{GraberPandharipande}
Tom Graber and Rahul Pandharipande, \emph{Constructions of nontautological
  classes on moduli spaces of curves}, Michigan Math. J. \textbf{51} (2003),
  no.~1, 93--109. \MR{1960923}

\bibitem[HZ86]{HarerZagier}
John Harer and Don Zagier, \emph{The {E}uler characteristic of the moduli space
  of curves}, Invent. Math. \textbf{85} (1986), no.~3, 457--485. \MR{848681}

\bibitem[Kee92]{Keel}
Sean Keel, \emph{Intersection theory of moduli space of stable {$n$}-pointed
  curves of genus zero}, Trans. Amer. Math. Soc. \textbf{330} (1992), no.~2,
  545--574. \MR{1034665}

\bibitem[Loo93]{Looijenga-M3}
Eduard Looijenga, \emph{Cohomology of {${\scr M}_3$} and {${\scr M}^1_3$}},
  Mapping class groups and moduli spaces of {R}iemann surfaces ({G}\"ottingen,
  1991/{S}eattle, {WA}, 1991), Contemp. Math., vol. 150, Amer. Math. Soc.,
  Providence, RI, 1993, pp.~205--228. \MR{1234266}

\bibitem[Mil08]{MilneLectures}
James~S. Milne, \emph{Lectures on etale cohomology (v2.10)}, 2008, Available at
  www.jmilne.org/math/, p.~196.

\bibitem[Pet15]{Petersenlocalsystems}
Dan Petersen, \emph{Cohomology of local systems on the moduli of principally
  polarized abelian surfaces}, Pacific J. Math. \textbf{275} (2015), no.~1,
  39--61. \MR{3336928}

\bibitem[Pet16]{Petersen}
\bysame, \emph{Tautological rings of spaces of pointed genus two curves of
  compact type}, Compos. Math. \textbf{152} (2016), no.~7, 1398--1420.
  \MR{3530445}

\bibitem[PW24a]{PayneWillwacher24}
Sam Payne and Thomas Willwacher, \emph{Weight 11 {C}ompactly {S}upported
  {C}ohomology of {M}oduli {S}paces of {C}urves}, Int. Math. Res. Not. IMRN
  (2024), no.~8, 7060--7098. \MR{4735654}

\bibitem[PW24b]{PayneWillwacher24b}
\bysame, \emph{Weight 2 compactly supported cohomology of moduli spaces of
  curves}, Duke Math. J. \textbf{173} (2024), no.~16, 3107--3178. \MR{4846192}

\bibitem[PZP19]{Petersenappendix}
Rahul Pandharipande, Dmitri Zvonkine, and Dan Petersen, \emph{Cohomological
  field theories with non-tautological classes}, Ark. Mat. \textbf{57} (2019),
  no.~1, 191--213. \MR{3951280}

\bibitem[Roc81]{Rockett}
Andrew~M. Rockett, \emph{Sums of the inverses of binomial coefficients},
  Fibonacci Quart. \textbf{19} (1981), no.~5, 433--437. \MR{644702}

\bibitem[Tom05]{Tommasi}
Orsola Tommasi, \emph{Rational cohomology of the moduli space of genus 4
  curves}, Compos. Math. \textbf{141} (2005), no.~2, 359--384. \MR{2134272}

\bibitem[vdBE05]{BogaartEdixhoven}
Theo van~den Bogaart and Bas Edixhoven, \emph{Algebraic stacks whose number of
  points over finite fields is a polynomial}, Number fields and function
  fields---two parallel worlds, Progr. Math., vol. 239, Birkh\"auser Boston,
  Boston, MA, 2005, pp.~39--49. \MR{2176584}

\bibitem[vdG23]{vdg}
Gerard van~der Geer, \emph{Curves over finite fields and moduli spaces}, Curves
  over finite fields---past, present and future, Panor. Synth\`eses, vol.~60,
  Soc. Math. France, Paris, 2023, pp.~113--144. \MR{4727507}

\end{thebibliography}
\end{document}